\documentclass[11pt, letterpaper]{article}
\usepackage[german, french, american]{babel}
\usepackage[latin1]{inputenc}
\usepackage[T1]{fontenc}
\usepackage{amsmath, amssymb, amsfonts}
\usepackage{amsthm}
\usepackage{mathtools}
\usepackage{mathrsfs}
\usepackage{enumerate}
\usepackage{graphicx}
\usepackage{color}
\usepackage{enumitem}
\usepackage{xfrac}

\usepackage{epstopdf}

\usepackage{thumbpdf}
\usepackage[hidelinks]{hyperref}
\usepackage[margin = 1in]{geometry}

\theoremstyle{plain}
\newtheorem{theorem}{Theorem}[section]
\newtheorem{lemma}[theorem]{Lemma}
\newtheorem{remark}[theorem]{Remark}
\newtheorem{corollary}[theorem]{Corollary}
\newtheorem{definition}[theorem]{Definition}
\newtheorem{assumption}[theorem]{Assumption}

\numberwithin{equation}{section}

\newcommand{\inner}[1]{{\left\langle #1 \right\rangle}}

\newcommand{\CAL}[1]{{\mathcal #1}}

\newcommand{\BR}{\mathbb{R}}

\newcommand{\ds}{\displaystyle}

\newcommand{\tz}{\widetilde{z}}

\newcommand{\dt}{\mathrm{d}t}

\newcommand{\ltwo}[1]{\big\| #1 \big\|_{L^2(\Omega)}}
\newcommand{\inltwo}[1]{\left\langle #1 \right\rangle}

\newcommand{\calZ}{\CAL{Z}}
\newcommand{\calT}{\CAL{T}}

\allowdisplaybreaks

\begin{document}

\title{Global Existence and Exponential Stability for a Nonlinear Thermoelastic Kirchhoff-Love Plate}

\author{Irena Lasiecka\thanks{Department of Mathematical Sciences, University of Memphis, Memphis, TN \hfill \texttt{lasiecka@memphis.edu}} \and
Michael Pokojovy\thanks{Department of Mathematics, Karlsruhe Institute of Technology, Karlsruhe, Germany \hfill \texttt{michael.pokojovy@kit.edu}} \and
Xiang Wan\thanks{University of Virginia, Charlottesville, VA \hfill \texttt{xw5he@virginia.edu}}}

\date{\today}

\maketitle

\begin{abstract}
	We study an initial-boundary-value problem for a quasilinear thermoelastic plate 
	of Kirchhoff \& Love-type with parabolic heat conduction due to Fourier,
	mechanically simply supported and held at the reference temperature on the boundary.
	For this problem, we show the short-time existence and uniqueness of classical solutions
	under appropriate regularity and compatibility assumptions on the data.
	Further, we use barrier techniques to prove the global existence and exponential stability of solutions under a smallness condition on the initial data.
	It is the first result of this kind established for a quasilinear non-parabolic thermoelastic Kirchhoff \& Love plate in multiple dimensions.
\end{abstract}

\begin{center}
\begin{tabular}{p{1.0in}p{5.0in}}
	\textbf{Key words:} & Kirchhoff-Love plates; nonlinear thermoelasticity; Hyperbolic-Parabolic PDE systems;
			global well-posedness; classical solutions; exponential stability \\
	\textbf{MSC (2010):} & Primary
	35M33,  
	35Q74,  
	74B20,  
	74F05,  
	74K20;  
	\\ &
	Secondary
	35A01,  
	35A02,  
	35A09,  
	35B40   
\end{tabular}
\end{center}

%

\section{Introduction}
\label{SECTION_INTRODUCTION}
Let $\Omega \subset \mathbb{R}^{d}$ ($d =2 \mbox{ or } 3$) be a bounded domain with a smooth boundary
representing the mid-plane of a thermoelastic plate.
With $w$ and $\theta$
denoting the vertical deflection and an appropriately weighted thermal moment with respect to the plate thickness,
both depending on a scaled time variable $t > 0$ and the space variable $(x_{1}, x_{2}) \in \Omega$,
the nonlinear Kirchhoff \& Love thermoelastic plate system reads as
\begin{subequations}
\begin{align}
	w_{tt} - \gamma \triangle w_{tt} + a(-\triangle w) \triangle^{2} w + \alpha \triangle \theta &= f(-\triangle w, -\nabla \triangle w)\phantom{0} \text{ in } (0, \infty) \times \Omega, 
	\label{EQUATION_INTRODUCTION_KIRCHHOFF_LOVE_PDE_1} \\
	\beta \theta_{t} - \eta \triangle \theta + \sigma \theta - \alpha \triangle w_{t} &= 0\phantom{f(-\triangle w, -\nabla \triangle w)} \text{ in } (0, \infty) \times \Omega
	\label{EQUATION_INTRODUCTION_KIRCHHOFF_LOVE_PDE_2}
\end{align}
along with the boundary conditions (hinged mechanical/Dirichlet thermal)
\begin{equation}
	w = \triangle w = \theta = 0 \text{ in } (0, \infty) \times \Omega
	\label{EQUATION_INTRODUCTION_KIRCHHOFF_LOVE_BC}
\end{equation}
and the initial conditions
\begin{equation}
	w(0, \cdot) = w^{0}, \quad w_{t}(0, \cdot) = w^{1}, \quad \theta(0, \cdot) = \theta^{0} \text{ in } \Omega.
	\label{EQUATION_INTRODUCTION_KIRCHHOFF_LOVE_IC}
\end{equation}
\end{subequations}
Here, $\alpha, \beta, \gamma, \eta, \sigma$ are positive constants and
$a \colon \mathbb{R} \to (0, \infty)$ as well as $f \colon \mathbb{R} \times \mathbb{R}^{d} \to \mathbb{R}$ are smooth functions.
For thin plates, $\gamma$ behaves like $h^{2}$ as $h \to 0$ (cf. \cite[Equation (2.16), p. 13]{LaLi1988}) and is, therefore, neglected in some literature.
In Section \ref{SECTION_MODEL_DESCRIPTION} below,
we present a short physical deduction of Equations (\ref{EQUATION_INTRODUCTION_KIRCHHOFF_LOVE_PDE_1})--(\ref{EQUATION_INTRODUCTION_KIRCHHOFF_LOVE_IC}).

Lasiecka et al. \cite{LaMaSa2008} studied a quasilinear PDE system
similar to (\ref{EQUATION_INTRODUCTION_KIRCHHOFF_LOVE_PDE_1})--(\ref{EQUATION_INTRODUCTION_KIRCHHOFF_LOVE_IC})
in a smooth, bounded domain $\Omega$ of $\mathbb{R}^{d}$ with $d \leq 3$
given by a Kirchhoff \& Love plate with parabolic heat conduction
\begin{subequations}
\begin{align}
	w_{tt} + \triangle^{2} w - \triangle \theta + a\triangle \big((\triangle w)^{3}\big) &= 0 \text{ in } (0, T) \times \Omega, 
	\label{EQUATION_QUASILINEAR_THERMOELASTIC_PLATE_NO_ROTATIONAL_MOMENTUM_PDE_1} \\
	\theta_{t} - \triangle \theta + \triangle w_{t} &= 0 \text{ in } (0, T) \times \Omega
	\label{EQUATION_QUASILINEAR_THERMOELASTIC_PLATE_NO_ROTATIONAL_MOMENTUM_PDE_2}
\end{align}
\end{subequations}
together with boundary conditions (\ref{EQUATION_INTRODUCTION_KIRCHHOFF_LOVE_BC})
and initial conditions (\ref{EQUATION_INTRODUCTION_KIRCHHOFF_LOVE_IC})
for an arbitrary $T > 0$.
For the initial-boundary-value problem 
(\ref{EQUATION_QUASILINEAR_THERMOELASTIC_PLATE_NO_ROTATIONAL_MOMENTUM_PDE_1})--(\ref{EQUATION_QUASILINEAR_THERMOELASTIC_PLATE_NO_ROTATIONAL_MOMENTUM_PDE_2}),
(\ref{EQUATION_INTRODUCTION_KIRCHHOFF_LOVE_BC})--(\ref{EQUATION_INTRODUCTION_KIRCHHOFF_LOVE_IC}),
they proved the global existence of weak solutions $(w, \theta)$ and their uniform decay in the norm of
\begin{equation}
	\Big(W^{1, \infty}\big(0, T; L^{2}(\Omega)\big) \cap L^{\infty}(0, T; W^{2, 4}(\Omega)\big)\Big) \times L^{\infty}\big(0, T; W^{1,2}(\Omega)\big). \notag
\end{equation}
The existence proof was based on a Galerkin approximation and compactness theorems,
while the uniform stability was obtained with the aid of energy techniques.

In their monograph \cite{ChuLa2010}, Chueshov and Lasiecka give an extensive study on the von K\'{a}rm\'{a}n
plate system both in pure elastic and thermoelastic cases.
With $w \colon \Omega \to \mathbb{R}$ denoting the vertical displacement 
and $v \colon \Omega \to \mathbb{R}$ standing for the Airy stress function
of a plate with its mid-plane occupying in the reference configuration a domain $\Omega \subset \mathbb{R}^{2}$,
the pure elastic version of K\'{a}rm\'{a}n plate system reads as
\begin{subequations}
\begin{align}
	w_{tt} - \alpha \triangle w_{tt} + \triangle^{2} u - [u, v + F_{0}] + Lu &= p
	\text{ in } (0, \infty) \times \Omega, 
	\label{EQUATION_INTRODUCTION_VON_KARMAN_PDE_1} \\
	\triangle^{2} v + [u, u] &= 0
	\text{ in } (0, \infty) \times \Omega, 
	\label{EQUATION_INTRODUCTION_VON_KARMAN_PDE_2}
\end{align}
\end{subequations}
where $[v, w] := v_{x_{1} x_{1}} w_{x_{2} x_{2}} + v_{x_{2} x_{2}} w_{x_{1} x_{1}} - 2 v_{x_{1} x_{2}} w_{x_{1} x_{2}}$,
$L$ is a first-order differential operator and $F_{0}, p \colon \Omega \to \mathbb{R}$ are given ``force'' functions.
Imposing standard initial conditions, under various sets of boundary conditions,
Chueshov and Lasiecka proved Equations (\ref{EQUATION_INTRODUCTION_VON_KARMAN_PDE_1})--(\ref{EQUATION_INTRODUCTION_VON_KARMAN_PDE_2})
possess a unique generalized, weak or strong solution depending on the data regularity.
The proof was based on a nonlinear Galerkin-type approximation.
Further, they studied the semiflow associated with the solution to Equations (\ref{EQUATION_INTRODUCTION_VON_KARMAN_PDE_1})--(\ref{EQUATION_INTRODUCTION_VON_KARMAN_PDE_2}),
in particular, they analyzed its long-time behavior and the existence of attracting sets.
Various damping mechanisms, thermoelastic effects, structurally coupled systems
such as acoustic chambers or gas flow past a plate were studied.
An extremely detailed and comprehensive literature overview was also given.

Denk et al. \cite{DeRaShi2009} considered a linearization of (\ref{EQUATION_QUASILINEAR_THERMOELASTIC_PLATE_NO_ROTATIONAL_MOMENTUM_PDE_1})--(\ref{EQUATION_QUASILINEAR_THERMOELASTIC_PLATE_NO_ROTATIONAL_MOMENTUM_PDE_2}), 
which corresponds to letting $a \equiv 0$, in a bounded or exterior $C^{4}$-domain of $\mathbb{R}^{d}$ for $d \geq 2$ subject to the initial conditions from Equation (\ref{EQUATION_INTRODUCTION_KIRCHHOFF_LOVE_IC})
and the boundary conditions
\begin{equation}
	w = \partial_{\nu} w = \theta = 0 \text{ on } (0, T) \times \partial \Omega,
	\label{EQUATION_ELASTIC_CLAMPED_BOUNDARY_CONDITION_THERMAL_DIRICHLET_BOUNDARY_CONDITION}
\end{equation}
where $\partial_{\nu} = (\nabla \cdot)^{T} \nu$ and $\nu$ denotes the outer unit normal vector to $\Omega$ on $\partial \Omega$.
By proving a resolvent estimate both in the whole space and in the half-space and employing localization techniques,
they showed that the $C_{0}$-semigroup for $(w, w_{t}, \theta)$ on the space
\begin{equation}
	W^{2, p}_{D}(\Omega) \times L^{p}(\Omega) \times L^{p}(\Omega) \text{ with }
	W^{2, p}_{D}(\Omega) = \{u \in W^{2, p}(\Omega) \,|\, u = \partial_{\nu} u = 0 \text{ on } \partial \Omega\} \notag
\end{equation}
is analytic. In case $\Omega$ is bounded, they also proved an exponential stability result for the semigroup.

Lasiecka and Wilke \cite{LaWi2012} presented an $L^{p}$-space treatment of
Equations (\ref{EQUATION_QUASILINEAR_THERMOELASTIC_PLATE_NO_ROTATIONAL_MOMENTUM_PDE_1})--(\ref{EQUATION_QUASILINEAR_THERMOELASTIC_PLATE_NO_ROTATIONAL_MOMENTUM_PDE_2}),
(\ref{EQUATION_INTRODUCTION_KIRCHHOFF_LOVE_BC})--(\ref{EQUATION_INTRODUCTION_KIRCHHOFF_LOVE_IC})
in bounded $C^{2}$-domains $\Omega$ of $\mathbb{R}^{d}$.
By proving the maximal $L^{p}$-regularity for the linearized problem,
they adopted the classical approach to prove the existence and uniquess of strong solutions satisfying
\begin{equation}
	(\triangle w, w_{t}, \theta) \in
	\Big(L^{p}_{\mu}\big(0, T; W^{2, p}(\Omega)\big) \cap
	W^{1, p}_{\mu}\big(0, T; L^{p}(\Omega)\big) \cap
	BUC\big(0, T; W^{2\mu - 2/p, p}(\Omega)\big)\Big)^{3} \notag
\end{equation}
for $p > 1 + \tfrac{d}{2}$,
where $L^{p}_{\mu}(\Omega)$ is the space of strongly measurable functions $u$
for which $t \mapsto t^{1 - \mu} u(t)$ lies in $L^{p}(\Omega)$
and $W^{1, p}_{\mu}(\Omega)$ stands for the space of weakly differentiable functions from $L^{p}_{\mu}(\Omega)$
whose first-order weak derivatives also lie in $L^{p}_{\mu}(\Omega)$.
For $T \leq \infty$, they showed a global strong solvability result for sufficiently small initial data in the interpolation space
\begin{equation}
	(\triangle w^{0}, w^{1}, \theta^{0}) \in \Big(\big(L^{p}(\Omega), W^{2, p}(\Omega) \cap W^{1, p}_{0}(\Omega)\big)_{\mu - 1/p, p}\Big)^{3}. \notag
\end{equation}
They pointed out that similar arguments can be used to obtain a short-time existence for arbitrarily large initial data.
Finally, they studied the first- and higher-order differentiability as well as analyticity of solutions under appropriate assumptions on the data.

Recently, Denk and Schnaubelt \cite{DeSch2015} considered a structurally damped elastic plate equation
\begin{equation}
	w_{tt} + \triangle^{2} w - \rho \triangle w_{t} = f \text{ in } (0, \infty) \times \Omega
\end{equation}
in a domain $\Omega \subset \mathbb{R}^{d}$, being either the whole space, a half-space or a bounded $C^{4}$-domain,
subject to inhomogeneous Dirichlet-Neumann boundary conditions
\begin{equation}
	w = g_{0}, \quad \partial_{\nu} w = g_{1} \text{ on } (0, \infty) \times \Omega \notag
\end{equation}
and the initial conditions
\begin{equation}
	w(0, \cdot) = w^{0}, \quad w_{t}(0, \cdot) = w^{1} \text{ in } \Omega, \notag
\end{equation}
with the data coming from appropriate $L^{p}$-Sobolev spaces for $p \in (1, \infty) \backslash \{3/2, 3\}$.
By showing the $\mathcal{R}$-sectoriality of the operator driving the flow $t \mapsto \big(w(t), w_{t}(t)\big)$ both in the whole space and the half-space scenarios,
they proved the $L^{p}$-maximum regularity for the generator on any finite time horizon $T > 0$.
In case of bounded $C^{4}$-domains, a standard localization technique was adopted to
deduce the maximum $L^{p}$-regularity for any time horizon $T \in (0, \infty]$.

In the present article, we study the quasilinear PDE system
associated with Equations (\ref{EQUATION_INTRODUCTION_KIRCHHOFF_LOVE_PDE_1})--(\ref{EQUATION_INTRODUCTION_KIRCHHOFF_LOVE_IC}).
In contrast to earlier works, dealing with a quasilinear system without maximal $L^{p}$-regularity property,
it is technically beneficial to look for classical rather than weak or strong solutions.
The necessity of studying smooth solutions results in a much higher complexity of the existence and uniqueness proof
as it has to be carried out at a higher energy level,
which, in turn, is based on a Kato-type approximation procedure rather than a Galerkin scheme.
The paper is structured as follows. After a short introduction Section \ref{SECTION_INTRODUCTION},
we present in Section \ref{SECTION_MODEL_DESCRIPTION} a brief physical deduction of the Kirchhoff \& Love plate from Equations
(\ref{EQUATION_INTRODUCTION_KIRCHHOFF_LOVE_PDE_1})--(\ref{EQUATION_INTRODUCTION_KIRCHHOFF_LOVE_IC}).
In Section \ref{SECTION_EXISTENCE_AND_UNIQUENESS},
an existence and uniquess result for Equations
(\ref{EQUATION_INTRODUCTION_KIRCHHOFF_LOVE_PDE_1})--(\ref{EQUATION_INTRODUCTION_KIRCHHOFF_LOVE_IC})
in the class of classical solutions is shown.
The long-time behavior of Equations (\ref{EQUATION_INTRODUCTION_KIRCHHOFF_LOVE_PDE_1})--(\ref{EQUATION_INTRODUCTION_KIRCHHOFF_LOVE_IC})
is studied in Section \ref{SECTION_LONG_TIME_BEHAVIOR}.
Under a smallness assumption on the initial data, the global existence and uniquess of solution is proved using energy estimates and the barrier method.
Further, this global solution is shown to decay at an exponential rate to the zero equilibrium state.
Finally, in the appendix Section \ref{APPENDIX},
we present a well-posedness theory along with higher energy estimates
for a linear wave equation with time- and space-dependent coefficients
as well as the homogeneus isotropic heat equation.

\section{Model Description $(d=2)$} \label{SECTION_MODEL_DESCRIPTION}
Consider a prismatic solid plate of uniform thickness $h > 0$ and constant material density $\rho > 0$
occupying in a reference configuration the domain $\mathcal{B}_{h} := \Omega \times (-\tfrac{h}{2}, \tfrac{h}{2})$ of $\mathbb{R}^{3}$,
where $\Omega \subset \mathbb{R}^{2}$ is bounded.
The underlying material is assumed to be elastically and thermally isotropic.
Further, we restrict ourselves to the case of infinitesimal thermoelasticity
with both stresses/strains and temperature gradient/heat flux being small.
Additionally, we assume the strains linearly decompose into elastic and thermal ones.
Despite of these linearity assumptions, a nonlinear (hypo)elastic law will be postulated 
allowing for materials with genuinely nonlinear response such as rubber, liquid crystal elastomers, etc.
Figure \ref{FIGURE_PLATE} below (adopted from \cite[Chapter 1]{Po2011})
displays a prismatic plate together with its mid-plane in the reference configuration.
\begin{figure}[!h]
	\begin{centering}
		\setlength{\unitlength}{0.8mm}
		\begin{picture}(90, 60)(0,0)
			\linethickness{1pt}
			\put(0,-2){\includegraphics[angle=0, scale=0.8]{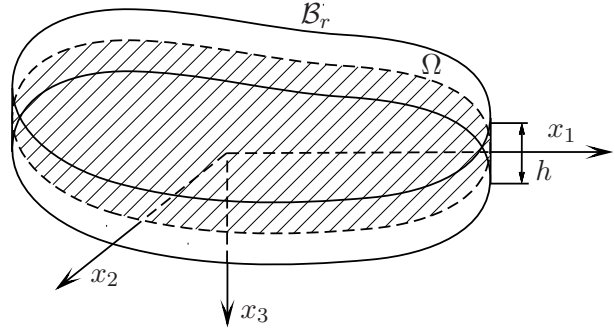}}
			\put(50, 50){$\mathcal{B}_{r}$}
			\put(70, 41.5){$\Omega$}
			\put(89, 24){$h$}
			\put(15,  7){$x_{2}$}
			\put(91, 31){$x_{1}$}
			\put(40,  1){$x_{3}$}
		\end{picture}
		\caption{Prismatic plate \label{FIGURE_PLATE}} 
	\end{centering}
\end{figure}

We start by interpreting the plate as a 3D body.
Let $\mathbf{U} = (U_{1}, U_{2}, U_{3})^{T}$ be the displacement vector in Lagrangian coordinates, $T$ stand for the absolute temperature
and $\mathbf{q} = (q_{1}, q_{2}, q_{3})^{T}$ be the associated heat flux.
Denote by $T_{0} > 0$ a reference temperature for which the body occupies the reference configuration and is free of any stresses or strains.
Further, let $S$ denote the entropy and 
\begin{equation}
	\boldsymbol{\sigma} = (\sigma_{ij})_{i = 1, 2, 3}^{j = 1, 2, 3}
	\mbox{\quad and \quad}
	\boldsymbol{\varepsilon} = \tfrac{1}{2}\big(\nabla \mathbf{U} + (\nabla \mathbf{U})^{T}\big) \notag
\end{equation}
stand for the first Piola \& Kirchhoff stress tensor and the infinitesimal Cauchy strain tensor.
We assume the total stress tensor decomposes into elastic and thermal stresses according to
\begin{equation}
	\boldsymbol{\sigma} = \boldsymbol{\sigma}^{\mathrm{elast}} - \boldsymbol{\sigma}^{\mathrm{therm}}.
	\label{EQUATION_LINEAR_STRESS_DECOMPOSITION}
\end{equation}
In the absense of external body forces and heat sources,
the momentum and energy balance equations  (cf. \cite[p. 142]{AmBeMi1984} and \cite[Chapter 1]{LaLi1988}) read then as
\begin{subequations}
\begin{align}
	\rho \mathbf{U_{tt}} + \operatorname{div} \boldsymbol{\sigma} &= 0 \text{ in } (0, \infty) \times \mathcal{B}_{h}, 
	\label{EQUATION_BALANCE_OF_MOMENTUM_3D_BASIC} \\
	T S_{t} + \operatorname{div} \mathbf{q} &= 0 \text{ in } (0, \infty) \times \mathcal{B}_{h}.
	\label{EQUATION_BALANACE_OF_THERMAL_ENERGY_3D_BASIC}
\end{align}
\end{subequations}

Similar to Ilyushin \cite[p. 42]{Il2004},
we define the elastic strain intensity $\varepsilon_{\mathrm{int}}$
as a properly scaled second invariant of the elastic strain deviator tensor by means of
\begin{equation}
	\varepsilon_{\mathrm{int}}^{\mathrm{elast}} 
	= \tfrac{\sqrt{2}}{3} \Big((\operatorname{tr} \boldsymbol{\varepsilon}^{\mathrm{elast}})^{2} 
	- \operatorname{tr}\big((\boldsymbol{\varepsilon}^{\mathrm{elast}})^{2}\big)\Big). \notag
\end{equation}
Similarly, we can define the elastic stress intensity via
\begin{equation}
	\sigma_{\mathrm{int}}^{\mathrm{elast}} 
	= \tfrac{\sqrt{2}}{3} \Big((\operatorname{tr} \boldsymbol{\sigma}^{\mathrm{elast}})^{2} 
	- \operatorname{tr}\big((\boldsymbol{\sigma}^{\mathrm{elast}})^{2}\big)\Big). \notag
\end{equation}
Within the classical hypoelasticity, we need to postulate a relation between these two quantities.
Here, we consider a general material law given by
\begin{equation}
	\sigma_{\mathrm{int}} = \kappa(\varepsilon_{\mathrm{int}}),
	\label{EQUATION_HYPOELASTIC_LAW_ELASTIC_PART}
\end{equation}
which generalizes power-law-type materials considered by Ambartsumian et al. in \cite[Equation (6)]{AmBeMi1984}.
For the thermal stresses and strains, we select a linear material law
\begin{equation}
	\boldsymbol{\sigma}^{\mathrm{therm}} = \tfrac{E}{1 - 2\mu} \boldsymbol{\varepsilon}^{\mathrm{therm}},
	\label{EQUATION_HYPOELASTIC_LAW_THERMAL_PART}
\end{equation}
where $E$ and $\mu$ play the role of Young's modulus and Poisson's ratio and can be reconstructed from the Hooke's law resulting
from linearizing Equation (\ref{EQUATION_HYPOELASTIC_LAW_ELASTIC_PART}) around zero.

With $\tau = T - T_{0}$ denoting the relative temperature, the thermal linearity and isotropy assumptions imply
\begin{equation}
	\boldsymbol{\varepsilon}^{\mathrm{therm}} = \alpha \tau \mathbf{I}_{3 \times 3},
	\label{EQUATION_THERMAL_STRAIN_MATERIAL_LAW}
\end{equation}
where $\alpha > 0$ is the thermal expansion coefficient (cf. \cite[p. 29]{LaLi1988}).
According to Nowicki \cite[Chapter 1]{No1986},
a linear approximation for the entropy reads as
\begin{equation}
	S = \gamma \operatorname{tr}\big(\varepsilon^{\mathrm{elast}}\big) + \tfrac{\rho c}{T_{0}} \tau,
	\label{EQUATION_ENTROPY_LINEARIZED}
\end{equation}
where $c > 0$ is the heat capacity and $\gamma = \frac{E \alpha}{1 - 2\mu}$.
Plugging Equations (\ref{EQUATION_LINEAR_STRESS_DECOMPOSITION}), 
(\ref{EQUATION_HYPOELASTIC_LAW_THERMAL_PART}), (\ref{EQUATION_THERMAL_STRAIN_MATERIAL_LAW}) and (\ref{EQUATION_ENTROPY_LINEARIZED})
into Equations (\ref{EQUATION_BALANCE_OF_MOMENTUM_3D_BASIC})--(\ref{EQUATION_BALANACE_OF_THERMAL_ENERGY_3D_BASIC})
and linearizing with respect to $\tau$ around zero, we get
\begin{subequations}
\begin{align}
	\rho \mathbf{U}_{tt} + \operatorname{div} \boldsymbol{\sigma}^{\mathrm{elastic}} + \gamma \nabla \tau
	&= 0 \text{ in } (0, \infty) \times \mathcal{B}_{h}, 
	\label{EQUATION_BALANCE_OF_MOMENTUM_3D_TRANSFORMED} \\
	\rho c \tau_{t} - \lambda_{0} \triangle \tau + \gamma T_{0} \operatorname{tr}\big(\boldsymbol{\varepsilon}^{\mathrm{elast}}_{t}\big)
	&= 0 \text{ in } (0, \infty) \times \mathcal{B}_{h}.
	\label{EQUATION_BALANACE_OF_THERMAL_ENERGY_3D_TRANSFORMED}
\end{align}
\end{subequations}
Together with Equation (\ref{EQUATION_HYPOELASTIC_LAW_ELASTIC_PART}),
Equations (\ref{EQUATION_BALANCE_OF_MOMENTUM_3D_TRANSFORMED})--(\ref{EQUATION_BALANACE_OF_THERMAL_ENERGY_3D_TRANSFORMED})
constitute the PDE system of 3D thermoelasticity.
In the following, we exploit these equations to deduce our thermoelastic plate model.

As it is typical for most plate theories,
we postulate the hypothesis of undeformable normals,
i.e., the linear filaments being perpendicular to the mid-plane before deformation
should also remain linear after the deformation.
Since we are interested in obtaining a Kirchhoff \& Love-type plate model,
we additionally assume these deformed filaments remain perpendicular to the deformed mid-plane.
The in-plane displacements are assumed negligible.
Mathematically, these structural assumptions can be written as
\begin{equation}
	\begin{split}
		U_{1}(x_{1}, x_{2}, x_{3}) &= -x_{3} w_{x_{1}}(x_{1}, x_{2}), \quad
		U_{2}(x_{1}, x_{2}, x_{3})  = -x_{3} w_{x_{2}}(x_{1}, x_{2}), \\
		U_{3}(x_{1}, x_{2}, x_{3}) &= \phantom{-x_{3}}w(x_{1}, x_{2}),
	\end{split}
	\label{EQUATION_KIRCHHOFF_LOVE_STRUCTURAL_ASSUMPTIONS}
\end{equation}
where $w$ is referred to as the bending component or the vertical displacement.
Thus, the elastic behavior of our plate can fully be described merely by $w$.
Figure \ref{PLATE_DEFORMATION} is self-describing and illustrates these structural assumptions.
\begin{figure}[!h]
	\begin{centering}
		\setlength{\unitlength}{0.7143mm} 
		\begin{picture}(150, 100)(0, 0)
			\linethickness{1pt} 
			\put(0, -5){\includegraphics[angle=0, scale=0.5]{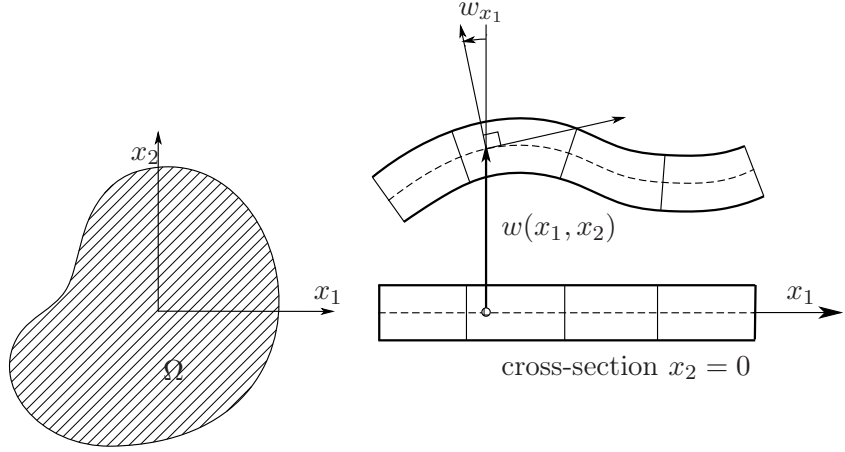}}
			\put(30, 18){$\Omega$}
			\put(58, 33){$x_{1}$}
			\put(24, 59){$x_{2}$}
			
			\put(146, 33){$x_{1}$}
			\put(93, 45){$w(x_{1}, x_{2})$}
			\put(93, 19){cross-section $x_{2} = 0$}
			\put(85, 86){$w_{x_{1}}$}
		\end{picture}
		\vspace{-0.4cm}
		
		\caption{Mid-plane of a plate as well as 
		plate cross-sections $x_{2} = 0$ before and after the deformation
		\label{PLATE_DEFORMATION}} 
	\end{centering}
\end{figure}

As for the thermal part of the system,
a properly weighted momentum of the relative temperature $\tau$ with respect to $x_{3}$ given by
\begin{equation}
	\theta(x_{1}, x_{2}) = \frac{12 \alpha}{h^{3}} \int_{-h/2}^{h/2} x_{3} \tau(x_{1}, x_{2}, x_{3}) \mathrm{d}x_{3} \notag
\end{equation}
will play a crucial role.
Proceeding as Lagnese and Lions \cite[pp. 29--31]{LaLi1988},
Equation (\ref{EQUATION_BALANACE_OF_THERMAL_ENERGY_3D_TRANSFORMED}) can be reduced to
\begin{equation}
	\rho c \theta_{t} - \lambda_{0} \triangle \theta
	+ \tfrac{12 \lambda_{0}}{\rho c h^{2}} \big(\tfrac{h \lambda_{1}}{2} + 1\big) \theta
	+ \tfrac{\alpha \gamma}{\rho c} \triangle w_{t} = 0
	\text{ in } (0, \infty) \times \Omega, \notag
\end{equation}
where $\lambda_{1} \geq 0$ is the parameter from the Newton's cooling law applied to the lower and upper faces of the plate.

Returning to the elastic part and assuming for a moment the material response $\kappa(\cdot)$ 
from Equation (\ref{EQUATION_HYPOELASTIC_LAW_ELASTIC_PART}) is an analytic function
possessing a Taylor expansion with the vanishing constant term
\begin{equation}
	\kappa(s) = \sum_{m = 1} a_{m} s^{m} \text{ for some } a_{m} \in \mathbb{R}, \notag
\end{equation}
we combine the approaches of Ambartsumian et al. \cite{AmBeMi1984} and Lagnese \& Lions \cite[Chapter 1]{LaLi1988} to deduce
\begin{equation}
	\rho h w_{tt} - \tfrac{\rho h^{3}}{12} \triangle w_{tt} + \triangle K(\triangle w) 
	+ D \tfrac{1 + \mu}{2} \triangle \theta = 0 \text{ in } (0, \infty) \times \Omega, \notag
\end{equation}
where $D = \frac{E h^{3}}{12(1 - \mu^{2})}$ denotes the flexural rigidity
and $K(\cdot)$ is obtained from $\kappa(\cdot)$ by means of
\begin{equation}
	K(s) = \sum_{n = 1}^{\infty} \big(\tfrac{2}{\sqrt{3}}\big)^{n + 1} \tfrac{h^{n + 1}}{n + 2} a_{n} s^{n}. \notag
\end{equation}
In contrast to \cite{AmBeMi1984},
the $\triangle w_{tt}$-term is not neglected here allowing for an adequate description of thicker plates
than those accounted for by the standard theory.
Obviously, $K(\cdot)$ is also analytic and its Taylor series has the same absolute convergence region as $\kappa(\cdot)$.
Taking into account
\begin{align*}
	K(s) &= \sum_{n = 1}^{\infty} \big(\tfrac{2}{\sqrt{3}}\big)^{n + 1} \tfrac{h^{n + 2}}{n + 2} a_{n} s^{n - 1}
	= s^{-3} \sum_{n = 1}^{\infty} \big(\tfrac{2}{\sqrt{3}}\big)^{n + 1} \tfrac{h^{n + 2}}{n + 2} a_{n} s^{n + 2} \\
	&= s^{-3} \sum_{m = 1}^{\infty} \big(\tfrac{2}{\sqrt{3}}\big)^{-1}
	\tfrac{\big(\tfrac{2}{\sqrt{3}} h s\big)^{n + 2}}{n + 2} a_{n} s^{n}
	= h^{3} \big(\tfrac{2}{\sqrt{3}} h s\big)^{-3} \big(\tfrac{2}{\sqrt{3}}\big)^{2} 
	\sum_{n = 1}^{\infty} a_{n} \tfrac{\big(\tfrac{2}{\sqrt{3}} h s\big)^{n + 2}}{n + 2} \\
	&= \tfrac{4h^{3}}{3} \big(\tfrac{2}{\sqrt{3}} h s\big)^{-3} \sum_{n = 1}^{\infty} a_{n} \tfrac{\big(\tfrac{2}{\sqrt{3}} h s\big)^{n + 2}}{n + 2}
	= \tfrac{4h^{3}}{3} \sum_{n = 1}^{\infty} a_{n} \big[I (\cdot)^{n}\big]\big(\tfrac{2}{\sqrt{3}} hs\big)
\end{align*}
with the linear operator
\begin{equation}
	\big(I f\big)(s) = s^{-3} \int_{0}^{s} \xi f(\xi) \mathrm{d}\xi \text{ for } s \in \mathbb{R} \backslash \{0\}, \notag
\end{equation}
the function $K$ can equivalently be written as
\begin{equation}
	K(s) = \tfrac{4}{3} h^{3} [I \kappa]\big(\tfrac{2}{\sqrt{3}} h s\big) 
	\text{ for } s \in \mathbb{R} \backslash \{0\}. \notag
\end{equation}
By density and continuity, $I$ can uniquely be extended to a mapping
from the set of continuous functions differentiable and vanishing in 0 with the following norm being bounded
\begin{equation}
	\|f\| = \max\big\{\sup_{x \in \mathbb{R}} |f(x)|, |f'(0)|\big\} \notag
\end{equation}
into the space of continuous functions vanishing at $0$.

Summarizing, our thermoelastic plate system reads as
\begin{subequations}
\begin{align}
	\rho h w_{tt} - \tfrac{\rho h^{3}}{12} \triangle w_{tt} + \triangle K(\triangle w) 
	+ D \tfrac{1 + \mu}{2} \triangle \theta &= 0 \text{ in } (0, \infty) \times \Omega, \\
	\rho c \partial_{t} \theta - \lambda_{0} \triangle \theta
	+ \tfrac{12 \lambda_{0}}{\rho c h^{2}} \big(\tfrac{h \lambda_{1}}{2} + 1\big) \theta
	+ \tfrac{\alpha \gamma \lambda_{0}}{\rho c} \triangle w_{t} &= 0
	\text{ in } (0, \infty) \times \Omega.
\end{align}
\end{subequations}
Various boundary conditions can be adopted.
We refer to \cite[Chapter 2]{Am1970}, \cite[Chapter 4]{Il2004} and \cite[Chapter 1]{LaLi1988} for further details.
Here, we consider a simply supported plate held at the reference temperature at the boundary:
\begin{equation}
	w = \triangle w = \theta = 0 \text{ in } (0, \infty) \times \partial \Omega. \notag
\end{equation}

\section{The Main Results} \label{SECTION_MAIN_RESULTS}
In this section, we state the main results on the well-posedness and long-time behavior of Equations (\ref{EQUATION_INTRODUCTION_KIRCHHOFF_LOVE_PDE_1})--(\ref{EQUATION_INTRODUCTION_KIRCHHOFF_LOVE_IC}). 
While the local result assumes smoothness of the boundary of the domain, regularity of the initial data and nonlinearities as well as certain compatibility conditions, 
the global results rely additionally and critically on some further smallness assumption on the initial data.
Recall $\Omega \subset \BR^d$ ($d = 2$ or $3$) is a bounded domain throughout this paper.

\begin{definition}
	\label{w_DEFINITION_CLASSICAL_SOLUTION}
	Let $s \geq 2$. Under a classical solution
	to Equations (\ref{EQUATION_INTRODUCTION_KIRCHHOFF_LOVE_PDE_1})--(\ref{EQUATION_INTRODUCTION_KIRCHHOFF_LOVE_IC}) on $[0, T]$ at the energy level $s$,
	we understand a function pair $(w, \theta) \colon [0, T] \times \bar{\Omega} \to \mathbb{R} \times \mathbb{R}$ satisfying
	\begin{align*}
		w      &\in \Big(\bigcap_{m = 0}^{s - 1} C^{m}\big([0, T], H^{s +2 - m}(\Omega) \cap H^{1}_{0}(\Omega)\big)\Big) \cap
		C^{s}\big([0, T], H^{2}(\Omega) \cap H^{1}_{0}(\Omega)\big), \\
		\theta &\in \Big(\bigcap_{k = 0}^{s - 2} C^{k}\big([0, T], H^{s + 1 - k}(\Omega) \cap H^{1}_{0}(\Omega)\big)\Big) \cap 
		C^{s - 1}\big([0, T], H^{1}_{0}(\Omega)\big)
	\end{align*}
	and, being plugged into Equations (\ref{EQUATION_INTRODUCTION_KIRCHHOFF_LOVE_PDE_1})--(\ref{EQUATION_INTRODUCTION_KIRCHHOFF_LOVE_IC}), turns them into tautology.
	Classical solutions on $[0, T)$ and $[0, \infty)$ are defined correspondingly. 
\end{definition}

\begin{definition} \label{w_definition_classic_solution}
	Let $w^{m}$, $m \geq 2$, and $\theta^{k}$, $k \geq 1$, denote the ``initial values'' 
	for $\partial_{t}^{m} w$ and $\partial_{t}^{k} \theta$
	formally and recursively computed in terms of $w^{0}, w^{1}$ and $\theta^{0}$
	based on Equations (\ref{EQUATION_NONLINEAR_PLATE_EQUATION_PDE_1})--(\ref{EQUATION_NONLINEAR_PLATE_EQUATION_IC}) (cf. \cite[p. 96]{JiaRa2000})
\end{definition}

\noindent
To proceed with our well-posedness result, we require the following assumptions.
\begin{assumption}
	\label{w_ASSUMPTION_LOCAL_EXISTENCE}
	Let $s \geq 3$ be an integer and let $\partial \Omega \in C^{s}$.
	\begin{enumerate}
		\item Let $a \in C^{s - 1}(\mathbb{R}, \mathbb{R})$.

		\item Let $f \in C^{s - 1}(\mathbb{R} \times \mathbb{R}^{2}, \mathbb{R})$.

		\item Let the initial data satisfy the regularity
		\begin{equation}
			w^{0}, \triangle w^{0} \in H^{s}(\Omega) \cap H^{1}_{0}(\Omega), \quad w^{1}, \triangle w^{1} \in H^{s-1}(\Omega) \cap H^{1}_{0}(\Omega), \quad \theta^{0} \in H^{s + 1}(\Omega) \cap H^{1}_{0}(\Omega) \notag
		\end{equation}
		as well as compatibility conditions
		\begin{equation}
			\begin{split}
				w^{m}, \triangle w^{m} &\in H^{s - m}(\Omega) \cap H^{1}_{0}(\Omega) \text{ for } m = 2, \dots, s - 1 \text{ and } w^{s} \in H^{2}(\Omega) \cap H^{1}_{0}(\Omega), \\
				\theta^{k} &\in H^{s + 1 - k}(\Omega) \cap H^{1}_{0}(\Omega) \text{ for } k = 1, \dots, s - 2 \text{ and } \theta^{s - 1} \in H^{1}_{0}(\Omega).
			\end{split}
			\notag
		\end{equation}
		
		\item Further, assume the ``initial ellipticity'' condition for $a(\triangle w^{0}) A$, i.e.,
		\begin{equation}
			\min_{x \in \bar{\Omega}} a\big(\triangle w^{0}(x)\big) > 0, \quad
			\text{where } \triangle w^{0} \in C^{0}(\bar{\Omega}) \text{ by virtue of Sobolev's imbedding theorem.} \notag
		\end{equation}
	\end{enumerate}
\end{assumption}

\noindent
Now, we can formulate our local well-posedness result.
\begin{theorem}[Local Well-Posedness]
	\label{w_THEOREM_LOCAL_EXISTENCE}
	
	If Assumption \ref{ASSUMPTION_LOCAL_EXISTENCE} is satisfied for some $s \geq 3$,
	Equations (\ref{EQUATION_INTRODUCTION_KIRCHHOFF_LOVE_PDE_1})--(\ref{EQUATION_INTRODUCTION_KIRCHHOFF_LOVE_IC})
	possess a unique classical solution $(w, \theta)$ at the energy level $s$ on a maximal interval $[0, T_{\mathrm{max}}) \neq \emptyset$ additionally satisfying
	\begin{equation}
		\partial_{t}^{s - 1} \theta \in L^{2}\big(0, T; H^{2}(\Omega) \cap H^{1}_{0}(\Omega)\big) \text{ and }
		\partial_{t}^{s} \theta \in L^{2}\big(0, T; L^{2}(\Omega)\big) \notag
	\end{equation}
	along with
	\begin{equation}
		\min\limits_{x \in \bar{\Omega}} a\big(\triangle w(t, x)\big) > 0 \text{ for any } t \in [0, T_{\mathrm{max}}). \notag
	\end{equation}
	Unless $T_{\mathrm{max}} = \infty$, we have $\min\limits_{x \in \bar{\Omega}} a\big(\triangle w(t, x)\big) \to 0$ as $\nearrow T_{\mathrm{max}}$
	or/and
	\begin{equation}
		\sum_{k = 0}^{s} \big\|\partial_{t}^{k} w(t, \cdot)\big\|_{H^{s +2- k}(\Omega)}^{2} +
		\sum_{k = 0}^{s-2} \big\|\partial_{t}^{k} \theta (t, \cdot)\big\|_{H^{s + 1 - k}(\Omega)}^{2} +
		\big\|\partial_{t}^{s-1} \theta(t, \cdot)\big\|_{H^{1}(\Omega)}^{2} \to \infty
		\text{ as } t\nearrow T_{\mathrm{max}}.
		\notag
	\end{equation}
\end{theorem}

Let $\|(w, \theta)\|_{\calZ_s \times \calT_s} \equiv \|(\partial_{t}^{\leq s} w, \partial_{t}^{\leq s - 1} \theta)\|_{\calZ_s \times \calT_s}$ 
with $\partial_{t}^{\leq k} = (1, \partial_{t}, \dots, \partial_{t}^{k})$ denote the standard norm associated with the solution space in Definition \ref{w_DEFINITION_CLASSICAL_SOLUTION} 
(a precise definition is given in Equation \eqref{s_spaces}). We now present our global results:
\begin{theorem}[Global Well-Posedness] \label{w_global_wellposedness}
	\label{w_global_wellposedness_proof}
	Let Assumption \ref{w_ASSUMPTION_LOCAL_EXISTENCE} be satisfied for some $s \geq 3$.
	Then, there exists a positive number $\epsilon$ (defined in Theorem \ref{global_wellposedness_proof} of Section \ref{SECTION_LONG_TIME_BEHAVIOR}) 
	such that for any initial data $(w^0, w^1, \theta^0)$ satisfying 
	$\|(w^0,w^1,\theta^0)\|_{\calZ_s \times \calT_s} \equiv \|(w^0, w^1, \dots, w^{s}, \theta^0, \dots, \theta^{s - 1})\|_{\calZ_s \times \calT_s} < \epsilon$
	(which roughly means the smallness of $\|w^{0}\|_{H^{5}(\Omega)}^{2} + \|w^{1}\|_{H^{4}(\Omega)}^{2} + \|\theta^{0}\|_{H^{4}(\Omega)}^{2}$ when $s=3$),
	the unique local solution of system from Theorem \ref{w_THEOREM_LOCAL_EXISTENCE} exists globally, i.e., $T_{\mathrm{max}} = \infty$.
\end{theorem}

\begin{theorem}[Uniform Stability] \label{w_uniform_stability}
	Under the conditions of Theorem \ref{w_global_wellposedness_proof} and given additionally
	$\|(w^0,w^1,\theta^0)\|_{\calZ_s \times \calT_s} \equiv \|(w^0, w^1, \dots, w^{s}, \theta^0, \dots, \theta^{s - 1})\|_{\calZ_s \times \calT_s} < \tilde{\epsilon}$ 
	for some small positive $\tilde{\epsilon}$ 
	(to be defined in Corollary \ref{cor_unif_stability} of Section \ref{SECTION_LONG_TIME_BEHAVIOR}), there exist positive constants $C$ and $k$ such that
	\begin{equation}
		\big\|(\partial_{t}^{\leq s} w, \partial_{t}^{\leq s - 1} \theta)(t, \cdot)\big\|_{\calZ_s \times \calT_s} 
		\leq Ce^{-kt} \|(w^0,w^1,\theta^0)\|_{\calZ_s \times \calT_s} \text{ for } t \geq 0. \notag
	\end{equation}
\end{theorem}


\section{Proof of Theorem \ref{w_THEOREM_LOCAL_EXISTENCE}: Local Well-Posedness} \label{SECTION_EXISTENCE_AND_UNIQUENESS}
To facilitate the analytical treatment of  (\ref{EQUATION_INTRODUCTION_KIRCHHOFF_LOVE_PDE_1})--(\ref{EQUATION_INTRODUCTION_KIRCHHOFF_LOVE_IC}),
we first reduce the order in space from four to two.
To this end, let $A$ denote the $L^{2}(\Omega)$-realization of the negative Dirichlet-Laplacian, i.e.,
\begin{equation}
	A := -\triangle, \quad
	D(A) := \big\{u \in H^{1}_{0}(\Omega) \,|\, \triangle u \in L^{2}(\Omega)\big\}. \notag
\end{equation}
Assuming $\partial \Omega$ is of class $C^{2}$, the elliptic regularity theory yields $D(A) = H^{2}(\Omega) \cap H^{1}_{0}(\Omega)$.
Moreover, $A$ is an isomorphism between $D(A)$ and $L^{2}(\Omega)$, $A^{-1}$ is a compact self-adjoint operator and
$(-\infty, 0]$ is contained both in the resolvent set of $A$ and $A^{-1}$.
Letting
\begin{equation}
	z := A w = -\triangle w, \label{EQUATION_EQUATION_ANSATZ_FOR_Z}
\end{equation}
Equations (\ref{EQUATION_INTRODUCTION_KIRCHHOFF_LOVE_PDE_1})--(\ref{EQUATION_INTRODUCTION_KIRCHHOFF_LOVE_IC}) rewrite
as an initial-boundary value problem for a system of partial \hbox{(pseudo-)}differential equations given by
\begin{subequations}
\begin{align}
	\big(A^{-1} + \gamma\big) z_{tt} + a(z) A z - \alpha A \theta &= f(z, \nabla z)\phantom{0} \text{ in } (0, \infty) \times \Omega,
	\label{EQUATION_NONLINEAR_PLATE_EQUATION_PDE_1} \\
	\beta \theta_{t} + \eta A \theta + \sigma \theta + \alpha z_{t} &= 0\phantom{f(z, \nabla z)} \text{ in } (0, \infty) \times \Omega,
	\label{EQUATION_NONLINEAR_PLATE_EQUATION_PDE_2} \\
	z = \theta &= 0\phantom{f(z, \nabla z)} \text{ in } (0, \infty) \times \partial \Omega,
	\label{EQUATION_NONLINEAR_PLATE_EQUATION_BC} \\
	z(0, \cdot) = z^{0}, \quad z_{t}(0, \cdot) = z^{1}, \quad \theta(0, \cdot) &= \theta^{0}\phantom{f(z, \nabla z} \text{ in } \Omega,
	\label{EQUATION_NONLINEAR_PLATE_EQUATION_IC}
\end{align}
\end{subequations}
where $z^{0} := -\triangle w^{0}$ and $z^{1} := -\triangle w^{1}$.
Note that, for any $s \geq 0$, the operator $A^{-1} + \gamma$ restricted onto $H^{s}(\Omega)$
is an automorphism of $H^{s}(\Omega)$. Therefore, Definition \ref{w_DEFINITION_CLASSICAL_SOLUTION} is equivalent with the following one in the new variable $z$:

\begin{definition}
	\label{DEFINITION_CLASSICAL_SOLUTION}
	Let $s \geq 2$. Under a classical solution
	to Equations (\ref{EQUATION_NONLINEAR_PLATE_EQUATION_PDE_1})--(\ref{EQUATION_NONLINEAR_PLATE_EQUATION_IC}) on $[0, T]$ at the energy level $s$,
	we understand a function pair $(z, \theta) \colon [0, T] \times \bar{\Omega} \to \mathbb{R} \times \mathbb{R}$ satisfying
	\begin{align*}
		z      &\in \Big(\bigcap_{m = 0}^{s - 1} C^{m}\big([0, T], H^{s - m}(\Omega) \cap H^{1}_{0}(\Omega)\big)\Big) \cap
		C^{s}\big([0, T], L^{2}(\Omega)\big), \\
		\theta &\in \Big(\bigcap_{k = 0}^{s - 2} C^{k}\big([0, T], H^{s + 1 - k}(\Omega) \cap H^{1}_{0}(\Omega)\big)\Big) \cap 
		C^{s - 1}\big([0, T], H^{1}_{0}(\Omega)\big)
	\end{align*}
	and, being plugged into Equations (\ref{EQUATION_NONLINEAR_PLATE_EQUATION_PDE_1})--(\ref{EQUATION_NONLINEAR_PLATE_EQUATION_IC}), turns them into tautology.
	Classical solutions on $[0, T)$ and $[0, \infty)$ are defined correspondingly. 
\end{definition}

\begin{remark}
	The choice $s = 2$ in Definition \ref{DEFINITION_CLASSICAL_SOLUTION}
	is standard for the linear situation, i.e., when $a(\cdot)$ is constant and the function $f(\cdot, \cdot)$ is linear.
	In this case, by virtue of the standard semigroup theory,
	for any initial data $(z^{0}, z^{1}, \theta^{0}) \in
	\big(H^{2}(\Omega) \cap H^{1}_{0}(\Omega)\big) \times H^{1}_{0}(\Omega) \times \big(H^{3}(\Omega) \cap H^{1}_{0}(\Omega)\big)$ with $\triangle \theta^{0} \in H^{1}_{0}(\Omega)$,
	there exists a unique classical solution at the energy level $s = 2$.

	On the contrary, if $a(\cdot)$ and $f(\cdot, \cdot)$ are both genuinely nonlinear,
	one usually can not expect obtaining a classical solution for the initial data at the energy level $s = 2$ (cf. \cite[Remark 14.4]{Ka1985}).
	Therefore, taking a higher energy level is inevitable to obtain classical solutions in the general nonlinear case.
	Unfortunately, this not only amounts to putting an additional Sobolev regularity assumption on the initial data
	and smoothness conditions on $a(\cdot)$ and $f(\cdot, \cdot)$,
	but also makes it necessary to postulate appropriate compatibility conditions.
\end{remark}

To better understand the nature of compatibility conditions, we make the following observation.
Assuming there exists a classical solution at an energy level $s \geq 2$,
we can use the smoothness in $t = 0$ and Equations (\ref{EQUATION_NONLINEAR_PLATE_EQUATION_PDE_1})--(\ref{EQUATION_NONLINEAR_PLATE_EQUATION_PDE_2}) to compute
\begin{equation}
	\begin{split}
		z_{tt} &= \big(A^{-1} + \gamma\big)^{-1} \Big(f(z, \nabla z) - a(z) A z + \alpha A \theta\Big), \\
		\theta_{t} &= -\tfrac{1}{\beta} \big(\eta A \theta + \sigma \theta + \alpha z_{t}\big).
	\end{split}
	\label{EQUATION_Z_TT_AND_THETA_T_RESOLVED}
\end{equation}
Evaluating these equations at $t = 0$, we obtain
\begin{equation}
	\begin{split}
		z_{tt}(0, \cdot) &= \big(A^{-1} + \gamma\big)^{-1} \Big(f(z^{0}, \nabla z^{0}) - a(z^{0}) A z^{0} + \alpha A \theta^{0}\Big), \\
		\theta_{t}(0, \cdot) &= -\tfrac{1}{\beta} \big(\eta A \theta^{0} + \sigma \theta^{0} + \alpha z^{1}\big).
	\end{split}
	\notag
\end{equation}
Assuming both $a(\cdot)$ and $f(\cdot, \cdot)$ are sufficiently smooth,
we can differentiate Equation (\ref{EQUATION_Z_TT_AND_THETA_T_RESOLVED}) with respect to $t$
and repeat the procedure to explicitely evaluate
$\partial_{t}^{m} z(0, \cdot)$ or $\partial_{t}^{k} \theta$ for $m = 2, \dots, s$ or $k = 1, \dots, s - 1$, respectively.
Thus, Definition \ref{w_definition_classic_solution} and Assumption \ref{w_ASSUMPTION_LOCAL_EXISTENCE} are equivalent to the following ones:

\begin{definition}
	Let $z^{m}$, $m \geq 2$, and $\theta^{k}$, $k \geq 1$, denote the ``initial values'' 
	for $\partial_{t}^{m} z$ and $\partial_{t}^{k} \theta$
	formally and recursively computed in terms of $z^{0}, z^{1}$ and $\theta^{0}$
	based on Equations (\ref{EQUATION_NONLINEAR_PLATE_EQUATION_PDE_1})--(\ref{EQUATION_NONLINEAR_PLATE_EQUATION_IC}) (cf. \cite[p. 96]{JiaRa2000}).
\end{definition}

\begin{assumption}
	\label{ASSUMPTION_LOCAL_EXISTENCE}
	Let $s \geq 3$ be an integer and let $\partial \Omega \in C^{s}$.
	\begin{enumerate}
		\item Let $a \in C^{s - 1}(\mathbb{R}, \mathbb{R})$.

		\item Let $f \in C^{s - 1}(\mathbb{R} \times \mathbb{R}^{2}, \mathbb{R})$.

		\item Let the initial data satisfy the regularity
		\begin{equation}
			z^{0} \in H^{s}(\Omega) \cap H^{1}_{0}(\Omega), \quad z^{1} \in H^{s-1}(\Omega) \cap H^{1}_{0}(\Omega), \quad \theta^{0} \in H^{s + 1}(\Omega) \cap H^{1}_{0}(\Omega) \notag
		\end{equation}
		as well as compatibility conditions
		\begin{equation}
			\begin{split}
				z^{m} &\in H^{s - m}(\Omega) \cap H^{1}_{0}(\Omega) \text{ for } m = 2, \dots, s - 1 \text{ and } z^{s} \in L^{2}(\Omega), \\
				\theta^{k} &\in H^{s + 1 - k}(\Omega) \cap H^{1}_{0}(\Omega) \text{ for } k = 1, \dots, s - 2 \text{ and } \theta^{s - 1} \in H^{1}_{0}(\Omega).
			\end{split}
			\notag
		\end{equation}
		
		\item Further, assume the ``initial ellipticity'' condition for $a(z^{0}) A$, i.e.,
		\begin{equation}
			\min_{x \in \bar{\Omega}} a\big(z_{0}(x)\big) > 0, \quad
			\text{where } z^{0} \in C^{0}(\bar{\Omega}) \text{ by virtue of Sobolev's imbedding theorem.} \notag
		\end{equation}
	\end{enumerate}
\end{assumption}

\noindent
We can reformulate our local well-posedness result Theorem \ref{w_THEOREM_LOCAL_EXISTENCE} in term of $z$ as follows:
\begin{theorem}
	\label{THEOREM_LOCAL_EXISTENCE}
	If Assumption \ref{ASSUMPTION_LOCAL_EXISTENCE} is satisfied for some $s \geq 3$,
	Equations (\ref{EQUATION_NONLINEAR_PLATE_EQUATION_PDE_1})--(\ref{EQUATION_NONLINEAR_PLATE_EQUATION_IC})
	possess a unique classical solution $(z, \theta)$ at the energy level $s$ on a maximal interval $[0, T_{\mathrm{max}}) \neq \emptyset$ additionally satisfying
	\begin{equation}
		\partial_{t}^{s - 1} \theta \in L^{2}\big(0, T; H^{2}(\Omega) \cap H^{1}_{0}(\Omega)\big) \text{ and }
		\partial_{t}^{s} \theta \in L^{2}\big(0, T; L^{2}(\Omega)\big) \notag
	\end{equation}
	along with
	\begin{equation}
		\min\limits_{x \in \bar{\Omega}} a\big(z(t, x)\big) > 0 \text{ for any } t \in [0, T_{\mathrm{max}}). \notag
	\end{equation}
	Unless $T_{\mathrm{max}} = \infty$, we have
	\begin{equation}
		\min\limits_{x \in \bar{\Omega}} a\big(z(t, x)\big) \to 0 \text{ as } t \nearrow T_{\mathrm{max}}
		\label{EQUATION_ELLIPTICITY_VIOLATION_AT_T_MAX}
	\end{equation}
	or/and
	\begin{equation}
		\sum_{k = 0}^{s} \big\|\partial_{t}^{k} z(t, \cdot)\big\|_{H^{s - k}(\Omega)}^{2} +
		\sum_{k = 0}^{s-2} \big\|\partial_{t}^{k} \theta (t, \cdot)\big\|_{H^{s + 1 - k}(\Omega)}^{2} +
		\big\|\partial_{t}^{s-1} \theta(t, \cdot)\big\|_{H^{1}(\Omega)}^{2} \to \infty
		\text{ as } t\nearrow T_{\mathrm{max}}.
		\label{EQUATION_BLOW_UP_AT_T_MAX}
	\end{equation}
\end{theorem}

\begin{proof}
	First, exploiting the second Hilbert's resolvent identity
	\begin{equation}
		\big(A^{-1} + \gamma\big)^{-1} = \tfrac{1}{\gamma} - \tfrac{1}{\gamma} A^{-1} \big(A^{-1} + \gamma\big)^{-1}, \notag
	\end{equation}
	we rewrite Equations (\ref{EQUATION_NONLINEAR_PLATE_EQUATION_PDE_1})--(\ref{EQUATION_NONLINEAR_PLATE_EQUATION_IC}) as
	\begin{subequations}
	\begin{align}
		z_{tt} + \tfrac{1}{\gamma} a(z) A z - \tfrac{\alpha}{\gamma} A \theta &= \phantom{-}F(z, \theta) & &\text{ in } (0, \infty) \times \Omega,
		\label{EQUATION_NONLINEAR_PLATE_EQUATION_TRANSFORMED_PDE_1} \\
		\theta_{t} + \tfrac{\eta}{\beta} A \theta &= -\tfrac{1}{\beta} \big(\alpha z_{t} + \sigma \theta\big) & &\text{ in } (0, \infty) \times \Omega,
		\label{EQUATION_NONLINEAR_PLATE_EQUATION_TRANSFORMED_PDE_2} \\
		z = \theta &= 0 & &\text{ on } (0, \infty) \times \partial \Omega,
		\label{EQUATION_NONLINEAR_PLATE_EQUATION_TRANSFORMED_BC} \\
		z(0, \cdot) = z^{0}, \quad z_{t}(t, \cdot) = z^{1}, \quad \theta(0, \cdot) &= \theta^{0} & &\text{ in } \Omega,
		\label{EQUATION_NONLINEAR_PLATE_EQUATION_TRANSFORMED_IC}
	\end{align}
	\end{subequations}
	where the nonlinear operator $F$ is given by
	\begin{equation}
		F(z, \theta) = \tfrac{1}{\gamma} (1 - K) f(z, \nabla z) + \tfrac{1}{\gamma} K \big(a(z) A z\big) - \tfrac{\alpha}{\gamma} K A \theta \notag
	\end{equation}
	with the compact linear operator
	\begin{equation}
		K := A^{-1} \big(A^{-1} + \gamma\big)^{-1} \notag
	\end{equation}
	continuously mapping $H^{s}(\Omega)$ to $H^{s + 2}(\Omega) \cap H^{1}_{0}(\Omega)$ for any $s \geq 0$
	(cf. proof of Theorem \ref{THEOREM_APPENDIX_LINEAR_HEAT_EQUATION}).
	Now, Equations (\ref{EQUATION_NONLINEAR_PLATE_EQUATION_TRANSFORMED_PDE_1})--(\ref{EQUATION_NONLINEAR_PLATE_EQUATION_TRANSFORMED_IC})
	are a pseudo-differential perturbation of a second-order hyperbolic-parabolic system
	constituted by a quasi-linear wave equation coupled to a linear heat equation. \medskip \\
	\noindent {\it Step 1: Modify the nonlinearity $a(\cdot)$.}
	Since no global positivity assumption is imposed on the nonlinearity $a(\cdot)$,
	the elipticity condition for $a(z) A$ can be violated at any time $t > 0$.
	To (preliminarily) rule out this possible degeneracy, the following construction is performed.
	
	Taking into account the continuity of $z^{0}$ and the connectedness of $\Omega$, we have
	\begin{equation}
		z^{0}(\bar{\Omega}) = \big[\min_{x \in \bar{\Omega}} z^{0}(x), \max_{x \in \bar{\Omega}} z^{0}(x)\big] =: J_{0}.
		\label{EQUATION_IMAGE_OF_DOMAIN_OMEGA_WRT_Z_NOUGHT}
	\end{equation}
	By Assumption \ref{ASSUMPTION_LOCAL_EXISTENCE}.4, $a(\cdot)$ is strictly positive on $J_{0}$.
	Consider an arbitrary {\em closed} set $J$ such that
	\begin{equation}
		J_{0} \subset \operatorname{int}(J) \text{ and } a(z) > 0 \text{ for } z \in J, \label{EQUATION_SET_J_NONTRIVIAL}
	\end{equation}
	which must exist due to the continuity of $a(\cdot)$.
	By standard continuation arguments, there exists a $C^{s}$-function $\hat{a}(\cdot)$ such that
	\begin{equation}
		\hat{a}(\zeta) = a(\zeta) \text{ for } \zeta \in J \text{ and } \inf_{\zeta \not \in J} \hat{a}(\zeta) > 0. \notag
	\end{equation}
	Now, we replace Equation (\ref{EQUATION_NONLINEAR_PLATE_EQUATION_TRANSFORMED_PDE_1}) with
	\begin{align}
		z_{tt} + \tfrac{1}{\gamma} \hat{a}(z) A z - \tfrac{\alpha}{\gamma} A \theta = F(z, \theta) \quad \text{ in } (0, \infty) \times \Omega
		\label{EQUATION_NONLINEAR_PLATE_EQUATION_TRANSFORMED_NEW_NONLINEARITY_PDE_1}
	\end{align}
	and first consider Equations 
	(\ref{EQUATION_NONLINEAR_PLATE_EQUATION_TRANSFORMED_NEW_NONLINEARITY_PDE_1}),
	(\ref{EQUATION_NONLINEAR_PLATE_EQUATION_TRANSFORMED_PDE_2})--(\ref{EQUATION_NONLINEAR_PLATE_EQUATION_TRANSFORMED_IC}).
	To solve this new problem, we tranform it to a fixed-point problem and use the Banach fixed-point theorem.
	Our proof will be reminiscent of that one by Jiang and Racke \cite[Theorem 5.2]{JiaRa2000}
	carried out for the quasilinear system of thermoelasticity. \medskip \\
	\noindent {\it Step 2: Define the fixed-point mapping.}
	Here and in the sequel, $H^{0}_{0}(\Omega) \equiv H^{0}(\Omega) := L^{2}(\Omega)$.
	For $N > 0$ and $T > 0$, let $X(N, T)$ denote the set of all regular distributions $(z, \theta)$ such that
	$(z, \theta)$ together with their weak derivatives satisfy the regularity conditions
	\begin{align*}
		\partial_{t}^{m} z      &\in C^{0}\big([0, T], H^{s - m}(\Omega)\big) \text{ for } m = 0, 1, \dots, s, \\
		\partial_{t}^{k} \theta &\in C^{0}\big([0, T], H^{s + 1 - k}(\Omega)\big) \text{ for } k = 0, 1, \dots, s - 2, \quad
		\partial_{t}^{s - 1} \theta \in C^{0}\big([0, T], H^{1}_{0}(\Omega)\big), \\
		\partial_{t}^{s - 1} \theta &\in L^{2}\big(0, T; H^{2}(\Omega) \cap H^{1}_{0}(\Omega)\big)  \text{ and }
		\partial_{t}^{s} \theta \in L^{2}\big(0, T; L^{2}(\Omega)\big)
	\end{align*}
	the boundary
	\begin{equation}
		\partial_{t}^{m} z = \partial_{t}^{k} \theta = 0 \text{ in } [0, T] \times \partial \Omega \text{ for } m, k = 0, 1, \dots, s - 1
		\notag
	\end{equation}
	and the initial conditions
	\begin{equation}
		\partial_{t}^{m} z(0, \cdot) = z^{m} \text{ for } m = 0, 1, \dots, s \text{ and }
		\partial_{t}^{k} \theta(0, \cdot) = \theta^{k} \text{ for } k = 0, 1, \dots, s - 1 \text{ in } \Omega
		\label{EQUATION_DEFINITION_OF_X_N_T_INITIAL_CONDITIONS}
	\end{equation}
	as well as the energy inequality
	\begin{equation}
		\begin{split}
			\max_{0 \leq t \leq T} \|\bar{D}^{s} z(t, \cdot)\|_{L^{2}(\Omega)}^{2} &+
			\sum_{k = 0}^{s - 2} \max_{0 \leq t \leq T} \|\partial_{t}^{k} \theta(t, \cdot)\|_{H^{s + 1 - k}(\Omega)}^{2} +
			\max_{0 \leq t \leq T} \|\partial_{t}^{s - 1} \theta(t, \cdot)\|_{H^{1}(\Omega)}^{2} \\
			&+ \int_{0}^{T} \big(\|\triangle \partial_{t}^{s - 1} \theta(t, \cdot)\|_{L^{2}(\Omega)}^{2} + \|\partial_{t}^{s} \theta(t, \cdot)\|_{L^{2}(\Omega)}^{2}\big) \mathrm{d}t \leq N^{2}.
		\end{split}
		\label{EQUATION_ENERGY_CONSTRAINT_FOR_X_N_T}
	\end{equation}
	Here, for $n \geq 0$, we let
	\begin{equation}
		  \bar{D}^{n} := \big((\partial_{t}, \nabla)^{\alpha} \,|\, 0 \leq |\alpha| \leq n\big). \notag
	\end{equation}
	
	For any $T_{0} > 0$ and sufficiently large $N > 0$,
	the set $X(N, T)$ is not empty for any $T \in (0, T_{0}]$.
	Indeed, if $N$ is sufficiently large, any pair $(z, \theta)$ of Taylor polynomials
	\begin{equation}
		z(t, \cdot) = \sum_{k = 0}^{s} \frac{z^{k} t^{k}}{k!} + P_{z}(t, \cdot) t^{s + 1}, \quad
		\theta(t, \cdot) = \sum_{k = 0}^{s-1} \frac{\theta^{k} t^{k}}{k!} + P_{\theta}(t, \cdot) t^{s}, \notag
	\end{equation}
	is contained in $X(N, T)$, where $P_{z}$, $P_{\theta}$ are arbitrary $C^{\infty}_{0}\big(\Omega)$-valued polynomials w.r.t. $t$.
	
	For $(\bar{z}, \bar{\theta}) \in X(N, T)$, consider the linear operator $\mathscr{F}$
	mapping $(\bar{z}, \bar{\theta})$ to a function pair $(z, \theta)$ such that $\theta$ is the unique classical solution to the linear heat equation
	\begin{equation}
		\begin{split}
			\theta_{t}(t, x) - \tfrac{\eta}{\beta} \triangle \theta(t, x) &= \bar{g}(t, x)\phantom{0} \text{ for } (t, x) \in (0, T) \times \Omega, \\
			\theta &= 0\phantom{\bar{g}(t, x)} \text{ for } (t, x) \in (0, T) \times \partial \Omega, \\
			\theta(0, \cdot) &= \theta^{0}(x)\;\;\text{ for } x \in \Omega
		\end{split}
		\label{EQUATION_HEAT_EQUATION_FIX_POINT_MAPPING_DEFINITION}
	\end{equation}
	with
	\begin{equation}
		\bar{g}(t, x) = -\tfrac{1}{\beta} \big(\alpha \bar{z}_{t}(t, x) + 
		\sigma \bar{\theta}(t, x)\big) \text{ for } (t, x) \in [0, T] \times \bar{\Omega}
		\label{EQUATION_DEFINITION_BAR_G_LOCAL_EXISTENCE}
	\end{equation}
	and, subsequently, define $z$ to be the unique classical solution to the linear wave equation
	\begin{equation}
		\begin{split}
			z_{tt}(t, x) - \bar{a}_{ij}(t, x) \triangle z(t, x) &= \bar{f}(t, x)\phantom{0} \text{ for } (t, x) \in (0, T) \times \Omega, \\
			z(t, x) &= 0\phantom{\bar{f}(t, x)} \text{ for } (t, x) \in (0, T) \times \partial \Omega, \\
			z(0, x) = z^{0}(x), \quad z_{t}(0, x) &= z^{1}(x)\phantom{00} \text{ for } x \in \Omega
		\end{split}
		\label{EQUATION_WAVE_EQUATION_FIX_POINT_MAPPING_DEFINITION}
	\end{equation}
	with
	\begin{equation}
		\begin{split}
			\bar{a}_{ij}(t, x) &:= \tfrac{1}{\gamma} \hat{a}\big(\bar{z}(t, x)\big) \delta_{ij} \text{ and } \\
			\bar{f}(t, x) &:= \tfrac{1}{\gamma} \big((1 - K) f(\bar{z}, \nabla \bar{z})\big)(t, x) + 
			\tfrac{1}{\gamma} \big(K \big(\hat{a}(\bar{z}) A \bar{z}\big)\big)(t, x) - \tfrac{\alpha}{\gamma} \big((1 - K) A \theta\big)(t, x)
		\end{split}
		\label{EQUATION_DEFINITION_BAR_A_AND_BAR_F_LOCAL_EXISTENCE}
	\end{equation}
	for $(t, x) \in [0, T] \times \Omega$.
	Note that the right-hand side $\bar{f}$ depends on $A \theta$ and not $A \bar{\theta}$ as the standard procedure would suggest.
	
	We prove $\mathscr{F}$ is well-defined.
	By the definition of $\bar{g}$ in Equation (\ref{EQUATION_DEFINITION_BAR_G_LOCAL_EXISTENCE}) 
	and the regularity of $(\bar{z}, \bar{g}) \in X(N, T)$, we trivially have
	\begin{equation}
		\partial_{t}^{k} \bar{g} \in C^{0}\big([0, T], H^{s - 1 - k}(\Omega)\big) \text{ for } k = 0, 1, \dots, s - 1. \notag
	\end{equation}
	By virtue of Theorem \ref{THEOREM_APPENDIX_LINEAR_HEAT_EQUATION},
	Equation (\ref{EQUATION_HEAT_EQUATION_FIX_POINT_MAPPING_DEFINITION})
	possesses a unique classical solution $\theta$ satisfying
	\begin{align*}
		\partial_{t}^{k} \theta &\in C^{0}\big([0, T], H^{s + 1 - k}(\Omega) \cap H^{1}_{0}(\Omega)\big) \text{ for } k = 0, 1, \dots, s - 2, \\
		\partial_{t}^{s - 1} \theta &\in C^{0}\big([0, T], H^{1}_{0}(\Omega)\big) \cap L^{2}\big(0, T; H^{2}(\Omega) \cap H^{1}_{0}(\Omega)\big)
		\text{ and } \partial_{t}^{s} \theta \in L^{2}\big(0, T; L^{2}(\Omega)\big). \notag
	\end{align*}
	Now, taking into account the regularity of $\bar{z}$ and $\theta$,
	exploiting Assumption \ref{ASSUMPTION_LOCAL_EXISTENCE} and applying Sobolev's imbedding theorem,
	we can verify that Assumption \ref{ASSUMPTION_LINEAR_WAVE_EQUATION} is satisfied with
	\begin{equation}
		\gamma_{i} = \max_{0 \leq t \leq T} \bar{\gamma}_{i}\big(\big\|\bar{z}(t, \cdot)\big\|_{H^{s - 1}(\Omega)}\big)
		\text{ for } i = 0, 1,
		\label{EQUATION_LOCAL_EXISTENCE_DEFINITION_OF_GAMMA_I}
	\end{equation}
	where $\gamma_{0}, \gamma_{1} \colon [0, \infty) \to (0, \infty)$ are continuous functions.
	Here, we used the Sobolev imbedding $\nabla \bar{z}(t, \cdot) \in H^{2}(\Omega) \hookrightarrow L^{\infty}(\Omega)$ along with the estimate
	\begin{equation}
		\big\|K\big(\hat{a}(\bar{z}) A \bar{z}\big)\big\|_{H^{m + 2}(\Omega)} \leq
		C \big\|\hat{a}(\bar{z}) A \bar{z}\big\|_{H^{m}(\Omega)} \text{ for } m = 0, 1, \dots, s - 2. \notag
	\end{equation}
	Here and in the following, $C > 0$ denotes a generic constant.
	Hence, by Theorem \ref{THEOREM_APPENDIX_LINEAR_WAVE_EQUATION},
	Equation (\ref{EQUATION_WAVE_EQUATION_FIX_POINT_MAPPING_DEFINITION})
	possesses a unique classical solution
	\begin{equation}
		z \in \bigcap_{m = 0}^{s - 1} C^{m}\big([0, T], H^{s - m}(\Omega) \cap H^{1}_{0}(\Omega)\big) \cap
		C^{s}\big([0, T], L^{2}(\Omega)\big) \notag
	\end{equation}
	implying $(z, \theta) \in X(N, T)$. Therefore, the mapping $\mathscr{F}$ is well-defined. \medskip \\
	\noindent {\it Step 3: Show the self-mapping property.}
	We prove that $\mathscr{F}$ maps $X(N, T)$ into itself provided $N$ is sufficiently large and $T$ is sufficiently small. We define
	\begin{equation}
		E_{0}(T) := \sum_{m = 0}^{s} \|z^{m}\|_{H^{s - m}(\Omega)}^{2} + \sum_{k = 0}^{s - 2} \|\theta^{k}\|_{H^{s + 1 - k}(\Omega)}^{2} + \|\theta^{s - 1}\|_{H^{1}(\Omega)}^{2}.
		\notag
	\end{equation}
	Recalling the definition of $\bar{g}$ in Equation (\ref{EQUATION_DEFINITION_BAR_G_LOCAL_EXISTENCE}),
	applying Theorem \ref{THEOREM_APPENDIX_LINEAR_HEAT_EQUATION} and using Equation (\ref{EQUATION_ENERGY_CONSTRAINT_FOR_X_N_T}), 
	we can estimate
	\begin{equation}
		\begin{split}
			\sum_{k = 0}^{s - 2} \max_{0 \leq t \leq T} \|\partial_{t}^{k} \theta(t, \cdot)&\|_{H^{s + 1 - k}(\Omega)}^{2} +
			\max_{0 \leq t \leq T} \|\partial_{t}^{s - 1} \theta(t, \cdot)\|_{H^{1}(\Omega)}^{2} \\
			&+ \int_{0}^{T} \big(\|\triangle \partial_{t}^{s - 1} \theta(t, \cdot)\|_{L^{2}(\Omega)}^{2} + \|\partial_{t}^{s} \theta(t, \cdot)\|_{L^{2}(\Omega)}^{2}\big) \mathrm{d}t
			\leq C N^{2} + C E_{0}.
		\end{split}
		\label{EQUATION_LOCAL_EXISTENCE_ESTIMATE_FOR_THETA}
	\end{equation}
	Further, taking into account Equations
	(\ref{EQUATION_DEFINITION_OF_X_N_T_INITIAL_CONDITIONS}), 
	(\ref{EQUATION_ENERGY_CONSTRAINT_FOR_X_N_T}) and (\ref{EQUATION_DEFINITION_BAR_A_AND_BAR_F_LOCAL_EXISTENCE})
	and applying Sobolev imbedding theorem and \cite[Theorem B.6]{JiaRa2000}, we obtain
	\begin{equation}
		\int_{0}^{T} \|\partial_{t}^{s - 1} \bar{f}(t, \cdot)\|_{L^{2}(\Omega)}^{2} \mathrm{d}t
		\leq C(N) (1 + T)
		\label{EQUATION_NONLINEAR_WAVE_EQUATION_RIGHT_HAND_SIDE_ESTIMATE_HIGHEST_ENERGY_LEVEL}
	\end{equation}
	and
	\begin{align}
		\sum_{m = 0}^{s - 2} &\max_{0 \leq t \leq T} \|\partial_{t}^{m} \bar{f}(t, \cdot)\|_{H^{s - 2 - m}(\Omega)}^{2} \notag \\
		&\leq
		\sum_{m = 0}^{s - 2} \max_{0 \leq t \leq T} \big\|\partial_{t}^{m} \big(\tfrac{1}{\gamma} (1 - K) f(\bar{z}, \nabla \bar{z}) + 
		\tfrac{1}{\gamma} K \big(\hat{a}(\bar{z}) A \bar{z}\big)
		-\tfrac{\alpha}{\gamma} (1 - K) A \theta\big)(t, \cdot)\big\|_{H^{s - 2 - m}(\Omega)}^{2} \notag \\
		&\leq
		\sum_{m = 0}^{s - 2} \big\|\partial_{t}^{m} \big(\tfrac{1}{\gamma} (1 - K) f(\bar{z}, \nabla \bar{z}) + 
		\tfrac{1}{\gamma} K \big(\hat{a}(\bar{z}) A \bar{z}\big)
		-\tfrac{\alpha}{\gamma} (1 - K) A \theta\big)(0, \cdot)\big\|_{H^{s - 2 - m}(\Omega)}^{2} \notag \\
		&+
		\sum_{m = 0}^{s - 2} \int_{0}^{T} \big\|\partial_{t}^{m} \big(\tfrac{1}{\gamma} (1 - K) f(\bar{z}, \nabla \bar{z}) + 
		\tfrac{1}{\gamma} K \big(\hat{a}(\bar{z}) A \bar{z}\big)
		-\tfrac{\alpha}{\gamma} (1 - K) A \theta\big)(t, \cdot)\big\|_{H^{s - 2 - m}(\Omega)}^{2} \mathrm{d}t \notag \\
		&\leq C(E_{0}) + C(N)(1 + T),
		\label{EQUATION_NONLINEAR_WAVE_EQUATION_RIGHT_HAND_SIDE_ESTIMATE}
	\end{align}
	where the fundamental theorem of calculus was employed.
	Plugging Equations (\ref{EQUATION_NONLINEAR_WAVE_EQUATION_RIGHT_HAND_SIDE_ESTIMATE_HIGHEST_ENERGY_LEVEL})
	and (\ref{EQUATION_NONLINEAR_WAVE_EQUATION_RIGHT_HAND_SIDE_ESTIMATE})
	into the energy estimate in Theorem \ref{THEOREM_APPENDIX_LINEAR_WAVE_EQUATION},
	we obtain
	\begin{equation}
		\max_{0 \leq t \leq T} \|\bar{D}^{s} z(t, \cdot)\|_{L^{2}(\Omega)}^{2} \leq
		\bar{K}(E_{0}, \gamma_{0}, \gamma_{1}) \zeta(N, T)
		\label{EQUATION_LOCAL_EXISTENCE_ESTIMATE_FOR_Z}
	\end{equation}
	with positive constants $\gamma_{0}, \gamma_{1}$ defined in Equation (\ref{EQUATION_LOCAL_EXISTENCE_DEFINITION_OF_GAMMA_I}),
	a positive constant $K$ being a continuous function of its variables and
	\begin{equation}
		\zeta(N, T) =
		\Big(1 + C(N) T^{1/2} \sum_{i = 0}^{5} T^{i/2}\Big)
		\exp\big(T^{1/2} C(N) (1 + T^{1/2} + T + T^{3/2})\big). \notag
	\end{equation}
	Combining the estimates in Equations
	(\ref{EQUATION_LOCAL_EXISTENCE_ESTIMATE_FOR_THETA}) and (\ref{EQUATION_LOCAL_EXISTENCE_ESTIMATE_FOR_Z}),
	we obtain
	\begin{equation}
		\begin{split}
			\max_{0 \leq t \leq T} &\|\bar{D}^{s} z(t, \cdot)\|_{L^{2}(\Omega)}^{2} +
			\sum_{k = 0}^{s - 2} \max_{0 \leq t \leq T} \|\partial_{t}^{k} \theta(t, \cdot)\|_{H^{s + 1 - k}(\Omega)}^{2} +
			\max_{0 \leq t \leq T} \|\partial_{t}^{s - 1} \theta(t, \cdot)\|_{H^{1}(\Omega)}^{2} \\
			&+ \int_{0}^{T} \big(\|\triangle \partial_{t}^{s - 1} \theta(t, \cdot)\|_{L^{2}(\Omega)}^{2} + \|\partial_{t}^{s} \theta(t, \cdot)\|_{L^{2}(\Omega)}^{2}\big) \mathrm{d}t 
			\leq \bar{K}(E_{0}, \gamma_{0}, \gamma_{1}) \zeta(N, T),
		\end{split}
		\label{EQUATION_LOCAL_EXISTENCE_ESTIMATE_FOR_Z_AND_THETA}
	\end{equation}
	possibly, with an increased constant $\bar{K}$.
	
	We now select $N$ such that
	\begin{equation}
		N^{2} \leq \tfrac{1}{2} \bar{K}(E_{0}, \gamma_{0}, \gamma_{1}). \notag
	\end{equation}
	Due to continuity of $\zeta(N, \cdot)$ in $T = 0$ and $\zeta(N_{0}, 0) = 1$,
	there exists $T > 0$ such that $\zeta\big(N, (0, T]\big) \subset [1, 2]$.
	Hence, the estimate in Equation (\ref{EQUATION_LOCAL_EXISTENCE_ESTIMATE_FOR_Z_AND_THETA})
	is satisfied with $N^{2}$ on its right-hand side.
	Therefore, $(z, \theta) \in X(N, T)$ and $\mathscr{F}$ maps $X(N, T)$ into itself. \medskip \\
	\noindent {\it Step 4: Prove the contraction property.}
	We consider the metric space
	\begin{equation}
		Y := \Big\{(z, \theta) \,\big|\,
		z, z_{t}, \nabla z \in L^{\infty}\big(0, T; L^{2}(\Omega)\big) \text{ and }
		\theta \in L^{\infty}\big(0, T; H^{1}(\Omega)\big)\Big\} \notag
	\end{equation}
	equipped with the metric
	\begin{equation}
		\rho\big((z, \theta), (\bar{z}, \bar{\theta})\big) =
		\Big(\mathop{\operatorname{ess\,sup}}_{0 \leq t \leq T} \big\|\bar{D}^{1} \big(z - \bar{z}\big)(t, \cdot)\big\|_{L^{2}(\Omega)}^{2} +
		\mathop{\operatorname{ess\,sup}}_{0 \leq t \leq T} \big\|(\theta - \bar{\theta})(t, \cdot)\big\|_{H^{1}(\Omega)}^{2}\Big)^{1/2}
		\notag
	\end{equation}
	for $(z, \theta), (\bar{z}, \bar{\theta}) \in Y$.
	Obviously, $Y$ is complete. Further, $X(N, T) \subset Y$. Moreover, $X(N, T)$ is closed in $Y$.
	Indeed, consider a sequence $\big((z_{n}, \theta_{n})\big)_{n \in \mathbb{N}} \subset X(N, T)$
	such that it is a Cauchy sequence in $Y$ and, thus, converges to some $(z, \theta) \in Y$.
	With the uniform energy bound in Equation (\ref{EQUATION_ENERGY_CONSTRAINT_FOR_X_N_T})
	being valid for $\big((z_{n}, \theta_{n})\big)_{n \in \mathbb{N}}$,
	it must possess a subsequence
	which weakly-$\ast$ converges to some element $(z^{\ast}, \theta^{\ast}) \subset X(N, T)$ in respective topologies.
	Since strong and weak-$\ast$ limits coincide, 
	we have $(z, \theta) = (z^{\ast}, \theta^{\ast}) \in X(N, T)$.
	
	We now prove that $\mathscr{F} \colon X(N, T) \to X(N, T)$ is a contraction mapping w.r.t. $\rho$.
	For $(\bar{z}, \bar{\theta}), (\bar{z}^{\ast}, \bar{z}^{\ast}) \in X(N, T)$, let
	$(z, \theta) := \mathscr{F}\big((\bar{z}, \bar{\theta})\big)$,
	$(z^{\ast}, \theta^{\ast}) := \mathscr{F}\big((\bar{z}^{\ast}, \bar{\theta}^{\ast})\big)$.
	With $(\bar{z}, \bar{\theta})$, $(\bar{z}^{\ast}, \bar{\theta}^{\ast})$,
	$(z, \theta)$, $(z^{\ast}, \theta^{\ast})$ all lying in $X(N, T)$,
	Equation (\ref{EQUATION_ENERGY_CONSTRAINT_FOR_X_N_T}) along with Sobolev imbedding theorem imply
	\begin{equation}
		\mathop{\operatorname{ess\,sup}}_{0 \leq t \leq T}
		\big\|\big(\bar{D}^{1}(\bar{z}, \bar{z}^{\ast}, z, z^{\ast})\big)(t, \cdot)\big\|_{L^{\infty}(\Omega)} \leq CN.
		\label{EQUATION_NONLINEAR_PLATE_EQUATION_NABLA_Z_ESTIMATE}
	\end{equation}
	Recalling Equations (\ref{EQUATION_NONLINEAR_PLATE_EQUATION_TRANSFORMED_PDE_1})--(\ref{EQUATION_NONLINEAR_PLATE_EQUATION_TRANSFORMED_PDE_2}),
	we can easily see that $(\tilde{z}, \tilde{\theta}) := (z - z^{\ast}, \theta - \theta^{\ast})$ satisfies
	\begin{align}
		\tilde{z}_{tt} + \tfrac{1}{\gamma} \hat{a}(z) A \tilde{z} - \tfrac{\alpha}{\gamma} A \tilde{\theta} &=
		F(\bar{z}, \bar{\theta}) - F(\bar{z}^{\ast}, \bar{\theta}^{\ast}) -
		\big(\hat{a}(\bar{z}) - \hat{a}(\bar{z}^{\ast})\big) A z^{\ast}, 
		\label{EQUATION_NONLINEAR_PLATE_EQUATION_TRANSFORMED_PDE_SOLUTION_DIFFERENCE_1} \\
		\tilde{\theta}_{t} + \tfrac{\eta}{\beta} A \tilde{\theta} &= -\tfrac{\alpha}{\beta} (\bar{z}_{t} - \bar{z}_{t}^{\ast}) - \tfrac{\sigma}{\beta} (\bar{\theta} - \bar{\theta}^{\ast})
		\label{EQUATION_NONLINEAR_PLATE_EQUATION_TRANSFORMED_PDE_SOLUTION_DIFFERENCE_2}
	\end{align}
	in $(0, T) \times \Omega$. Further, we have
	\begin{align}
		\tilde{z} = \tilde{\theta} &= 0 & &\text{ on } (0, \infty) \times \partial \Omega,
		\label{EQUATION_NONLINEAR_PLATE_EQUATION_TRANSFORMED_SOLUTION_DIFFERENCE_BC} \\
		\tilde{z}(0, \cdot) \equiv 0, \quad \tilde{z}_{t}(t, \cdot) \equiv 0, \quad \tilde{\theta}(0, \cdot) &\equiv 0 & &\text{ in } \Omega.
		\label{EQUATION_NONLINEAR_PLATE_EQUATION_TRANSFORMED_SOLUTION_DIFFERENCE_IC}
	\end{align}
	Multiplying Equation (\ref{EQUATION_NONLINEAR_PLATE_EQUATION_TRANSFORMED_PDE_SOLUTION_DIFFERENCE_2}) in $L^{2}(\Omega)$ with $A \tilde{\theta}$,
	using Young's inequality and integrating w.r.t. $t$, we obtain
	\begin{equation*}
		\begin{split}
			\|A^{1/2} &\tilde{\theta}(t, \cdot)\|_{L^{2}(\Omega)}^{2} +
			\tfrac{\eta}{\beta} \int_{0}^{t} \|A\tilde{\theta}(\tau, \cdot)\|_{L^{2}(\Omega)}^{2} \\
			&\leq \varepsilon \int_{0}^{t} \|A \tilde{\theta}(\tau, \cdot)\|_{L^{2}(\Omega)}^{2} \mathrm{d}\tau
			+ C_{\varepsilon} T \mathop{\operatorname{ess\,sup}}_{0 \leq \tau \leq t} \Big(\big\|\bar{D}^{1} \big(\bar{z} - \bar{z}^{\ast}\big)(\tau, \cdot)\big\|_{L^{2}(\Omega)}^{2} +
			\big\|(\bar{\theta} - \bar{\theta}^{\ast})(\tau, \cdot)\big\|_{L^{2}(\Omega)}^{2}\Big).
		\end{split}
	\end{equation*}
	Hence, by Poincare-Friedrichs' inequality, selecting $\varepsilon > 0$ sufficiently small, we obtain
	\begin{equation}
		\begin{split}
			\|\tilde{\theta}(t, \cdot)\|_{H^{1}(\Omega)}^{2}
			&\leq -\tfrac{\eta}{2 \beta} \int_{0}^{t} \|A\tilde{\theta}(\tau, \cdot)\|_{L^{2}(\Omega)}^{2} \mathrm{d}\tau \\
			&+ C T \mathop{\operatorname{ess\,sup}}_{0 \leq \tau \leq t} \Big(\big\|\bar{D}^{1} \big(\bar{z} - \bar{z}^{\ast}\big)(\tau, \cdot)\big\|_{L^{2}(\Omega)}^{2} +
			\big\|(\bar{\theta} - \bar{\theta}^{\ast})(\tau, \cdot)\big\|_{L^{2}(\Omega)}^{2}\Big)
		\end{split}
		\label{EQUATION_NONLINEAR_PLATE_EQUATION_SOLUTION_DIFFERENCE_THETA_ESTIMATE}
	\end{equation}
	Similarly, multiplying Equations 
	(\ref{EQUATION_NONLINEAR_PLATE_EQUATION_TRANSFORMED_PDE_SOLUTION_DIFFERENCE_1}) in $L^{2}\big(0, T; L^{2}(\Omega)\big)$ with $\tilde{z}_{t}$,
	applying Green's formula, using chain and product rules, taking into account Equation (\ref{EQUATION_NONLINEAR_PLATE_EQUATION_TRANSFORMED_SOLUTION_DIFFERENCE_BC}),
	exploiting the local Lipschitz continuity of $\hat{a}(\cdot)$ and $f(\cdot, \cdot)$,
	using Equations  (\ref{EQUATION_ENERGY_CONSTRAINT_FOR_X_N_T}) and (\ref{EQUATION_NONLINEAR_PLATE_EQUATION_SOLUTION_DIFFERENCE_THETA_ESTIMATE})
	as well as exploiting Young's and Poincar\'{e}-Friedrichs's inequalities,
	we can estimate for any $t \in [0, T]$
	\begin{equation}
		\begin{split}
			\big\|\bar{D}^{1} \tilde{z}(t, \cdot)\big\|_{L^{2}(\Omega)}^{2} &\leq
			\tfrac{\eta}{2 \beta} \int_{0}^{t} \|A\tilde{\theta}(\tau, \cdot)\|_{L^{2}(\Omega)}^{2} \mathrm{d}\tau +
			C T \mathop{\operatorname{ess\,sup}}_{0 \leq \tau \leq t} \big\|\big(\bar{\theta} - \bar{\theta}^{\ast}\big)(\tau, \cdot)\big\|_{L^{2}(\Omega)}^{2} \\
			&+ C(N) \Big((1 + T^{-1/2}) \int_{0}^{t} \big\|\bar{D}^{1} \tilde{z}(\tau, \cdot)\big\|_{L^{2}(\Omega)}^{2} \mathrm{d}\tau \\
			&+ T^{1/2} (1 + T)  \mathop{\operatorname{ess\,sup}}_{0 \leq \tau \leq t} \big\|\bar{D}^{1} \big(\bar{z} - \bar{z}^{\ast}\big)(\tau, \cdot)\big\|_{L^{2}(\Omega)}^{2}\Big)
		\end{split}
		\label{EQUATION_NONLINEAR_PLATE_EQUATION_SOLUTION_DIFFERENCE_Z_ESTIMATE}
	\end{equation}
	Adding up Equations
	(\ref{EQUATION_NONLINEAR_PLATE_EQUATION_SOLUTION_DIFFERENCE_THETA_ESTIMATE})--(\ref{EQUATION_NONLINEAR_PLATE_EQUATION_SOLUTION_DIFFERENCE_Z_ESTIMATE}),
	using Gronwall's inequality, taking into account Equation (\ref{EQUATION_NONLINEAR_PLATE_EQUATION_TRANSFORMED_SOLUTION_DIFFERENCE_IC})
	and selecting $T$ sufficiently small, we can estimate
	\begin{equation}
		\rho\big((z, \theta), (z^{\ast}, \theta^{\ast})\big) \leq
		\lambda \rho\big((\bar{z}, \bar{\theta}), (\bar{z}^{\ast}, \bar{\theta}^{\ast})\big) \notag
	\end{equation}
	for some $\lambda \in (0, 1)$.
	Hence, $\mathscr{F}$ is a contraction on $X(N, T)$ in the metric of space $Y$.
	With $X(N, T)$ being closed, Banach fixed-point theorem implies $\mathscr{F}$ has a unique fixed point $(z, \theta) \in X(N, T)$.
	Finally, due to the smoothness of $(z, \theta)$, we can easily verify $(z, \theta)$ is a unique classical solution to
	Equations (\ref{EQUATION_NONLINEAR_PLATE_EQUATION_TRANSFORMED_NEW_NONLINEARITY_PDE_1}),
	(\ref{EQUATION_NONLINEAR_PLATE_EQUATION_TRANSFORMED_PDE_2})--(\ref{EQUATION_NONLINEAR_PLATE_EQUATION_TRANSFORMED_IC})
	at the energy level $s$. \medskip \\
	\noindent {\it Step 5: Continuation to the maximal interval.}
	Observing that
	$z(T, \cdot), z_{t}(T, \cdot)$ and $\theta(T, \cdot)$ satisfy the regularity and compatibility assumptions
	and carrying out the standard continuation argument,
	we obtain a maximal interval $[0, T^{\ast}_{J})$ for which the classical solution (uniquely) exists.
	Due to the interval's maximality, unless $T_{J}^{\ast} = \infty$, we have
	\begin{equation}
		\sum_{k = 0}^{s} \big\|\partial_{t}^{k} z(t, \cdot)\big\|_{H^{s - k}(\Omega)}^{2} +
		\sum_{k = 0}^{s-2} \big\|\partial_{t}^{k} \theta(t, \cdot)\big\|_{H^{s + 1 - k}(\Omega)}^{2} +
		\big\|\partial_{t}^{s-1} \theta(t, \cdot)\big\|_{H^{1}(\Omega)}^{2} \to \infty \text{ as } t \nearrow T^{\ast}_{J}.
		\label{EQUATION_BLOW_UP_AT_T_AST_NAIVE}
	\end{equation}
	\noindent {\it Step 6: Returning to the original system.}
	By virtue of Sobolev's imbedding theorem,
	the function $a \circ z$ is continuous on $[0, T^{\ast}_{J}) \times \bar{\Omega}$.
	Hence, the number
	\begin{equation}
		T_{\mathrm{max}, J} :=
		\left\{\begin{array}{cl}
			T^{\ast}_{J}, & \text{if } \hat{a} \circ z \equiv a \circ z \text{ in } [0, T^{\ast}) \times \bar{\Omega}, \\
			\min\big\{t \in [0, T^{\ast}) \,\big|\,
			a(t, x) \not \in \operatorname{int}(J) \text{ for } x \in \bar{\Omega}\big\}, & \text{otherwise}
		\end{array}
		\right.
		\notag
	\end{equation}
	is well-defined and positive by Equation (\ref{EQUATION_SET_J_NONTRIVIAL}).
	Denote by $(z_{J}, \theta_{J})$ the unique classical solution to (\ref{EQUATION_NONLINEAR_PLATE_EQUATION_TRANSFORMED_NEW_NONLINEARITY_PDE_1}),
	(\ref{EQUATION_NONLINEAR_PLATE_EQUATION_TRANSFORMED_PDE_2})--(\ref{EQUATION_NONLINEAR_PLATE_EQUATION_TRANSFORMED_IC})
	restricted onto $[0, T_{\mathrm{max}, J})$.
	Consider now an increasing sequence $(J_{n})_{n \in \mathbb{N}}$ of closed sets satisfying Equation (\ref{EQUATION_SET_J_NONTRIVIAL}) such that
	\begin{equation}
		T_{\mathrm{max}, J_{n}} \nearrow T_{\mathrm{max}} := \sup\big\{T_{\mathrm{max}, J} \,\big|\, J \text{ satisfies Equation } (\ref{EQUATION_SET_J_NONTRIVIAL})\big\} \text{ as } n \to \infty.
		\label{EQUATION_DEFINITION_OF_T_MAX}
	\end{equation}
	By construction, $(z_{J_{n}}, \theta_{J_{n}})$ solves the original problem 
	(\ref{EQUATION_NONLINEAR_PLATE_EQUATION_TRANSFORMED_PDE_1})--(\ref{EQUATION_NONLINEAR_PLATE_EQUATION_TRANSFORMED_IC}) on $[0, T_{\mathrm{max}, J_{n}})$ and
	\begin{equation}
		(z_{J_{m}}, \theta_{J_{m}}) \equiv (z_{J_{n}}, \theta_{J_{n}}) \text{ on } [0, T_{\mathrm{max}, J_{m}}) \text{ for } m, n \in \mathbb{N} \text{ with } m \leq n. \notag
	\end{equation}
	Hence, letting for $t \in [0, T_{\mathrm{max}})$,
	\begin{equation}
		(z, \theta)(t) := (z_{J_{n}}, \theta_{J_{n}})(t) \text{ for any } n \in \mathbb{N} \text{ such that } T_{\mathrm{max}, J_{n}} > t, \notag
	\end{equation}
	we oberve $(z, \theta)$ uniquely defines a classical solution to (\ref{EQUATION_NONLINEAR_PLATE_EQUATION_TRANSFORMED_PDE_1})--(\ref{EQUATION_NONLINEAR_PLATE_EQUATION_TRANSFORMED_IC})
	on $[0, T_{\mathrm{max}})$.
	Moreover, unless $T_{\mathrm{max}} = \infty$, we have Equation (\ref{EQUATION_BLOW_UP_AT_T_MAX}) and/or Equation (\ref{EQUATION_ELLIPTICITY_VIOLATION_AT_T_MAX}).
	Indeed, if neither was the case, we could redefine $J_{0}$ from Equation (\ref{EQUATION_IMAGE_OF_DOMAIN_OMEGA_WRT_Z_NOUGHT}) via
	\begin{equation}
		J_{0} := \big[\min_{x \in \bar{\Omega}} a(z(T_{\mathrm{max}}, x)), \max_{x \in \bar{\Omega}} a(z(T_{\mathrm{max}}, x))\big] \notag
	\end{equation}
	and repeat Step 5 to obtain a classical solution $(z_{J}, \theta_{J})$ existing beyond $T_{\mathrm{max}}$,
	which would contradict Equation (\ref{EQUATION_DEFINITION_OF_T_MAX}).
	The overall uniqueness follows similar to Step 4.
\end{proof}

\begin{remark}
	Equation (\ref{EQUATION_BLOW_UP_AT_T_MAX}) is equivalent to
	\begin{equation}
		\big\|z(t, \cdot)\big\|_{H^{s}(\Omega)}^{2} +
		\big\|z(t, \cdot)\big\|_{H^{s-1}(\Omega)}^{2} \to \infty \text{ as } t\nearrow T_{\mathrm{max}}.
		\notag
	\end{equation}
	Indeed, arguing by contradiction, if the norms in Equation (\ref{EQUATION_BLOW_UP_AT_T_MAX}) are bounded,
	Equations (\ref{EQUATION_NONLINEAR_PLATE_EQUATION_TRANSFORMED_PDE_1}), (\ref{EQUATION_NONLINEAR_PLATE_EQUATION_PDE_2})
	suggest the derivatives of $z$ and $\theta$ as well as $\theta$ itself are bounded in respective topologies,
	which contradicts the maximality of $T_{\mathrm{max}}$.
\end{remark}

\section{Proof of Theorems \ref{w_global_wellposedness} and \ref{w_uniform_stability}: Global Existence and Long-Time Behavior} \label{SECTION_LONG_TIME_BEHAVIOR}
In this section, we will show the global existence and the exponential stability of the local solution to
Equations \eqref{EQUATION_INTRODUCTION_KIRCHHOFF_LOVE_PDE_1}--\eqref{EQUATION_INTRODUCTION_KIRCHHOFF_LOVE_IC} 
(or, equivalently, \eqref{EQUATION_NONLINEAR_PLATE_EQUATION_PDE_1}--\eqref{EQUATION_NONLINEAR_PLATE_EQUATION_IC})
established in Theorem \ref{THEOREM_LOCAL_EXISTENCE} 
provided the initial data are `small.'
Motivated by \cite{LaMaSa2008} (cf. Equation (\ref{EQUATION_QUASILINEAR_THERMOELASTIC_PLATE_NO_ROTATIONAL_MOMENTUM_PDE_1})),
we consider here a particular nonlinearity:
\begin{equation} \label{definition_F}
	F(z, \nabla z, Az) \equiv \omega A (z^3) = \omega \big(-3z^2Az + 6z|\nabla z|^2\big).
\end{equation} 
For the sake of convenience, without loss of generality, we assume $\omega = 1$.
Hence, the system we study becomes:
\begin{subequations}
\begin{align}
	\big(A^{-1} + \gamma\big) z_{tt} + A z - \alpha A \theta &= -3z^2Az+6z|\nabla z|^2 \phantom{0} &&\text{ in } (0, \infty) \times \Omega, 
	\label{system_stability_1} \\
	\beta \theta_{t} + \eta A \theta + \sigma \theta + \alpha z_{t} &= 0\phantom{f(z, \nabla z)} &&\text{ in } (0, \infty) \times \Omega, 
	\label{system_stability_2} \\
	z = \theta &= 0\phantom{f(z, \nabla z)} &&\text{ in } (0, \infty) \times \partial \Omega, 
	\label{system_stability_3} \\
	z(0, \cdot) = z^{0}, \quad z_{t}(0, \cdot) = z^{1}, \quad \theta(0, \cdot) &= \theta^{0}\phantom{f(z, \nabla z} && \text{ in } \Omega.
	\label{system_stability_4}
\end{align}
\end{subequations}
Under Assumption \ref{ASSUMPTION_LOCAL_EXISTENCE}, Theorem \ref{THEOREM_LOCAL_EXISTENCE} establishes the local existence of a unique classical solution 
to Equations \eqref{system_stability_1}--\eqref{system_stability_4} at the energy level $s \geq 3$:
\begin{equation} \label{s_spaces}
	\begin{aligned}
		z &\in \Big(\bigcap_{m = 0}^{s - 1} C^{m}\big([0, T_{\max}), H^{s - m}(\Omega) \cap H^{1}_{0}(\Omega)\big)\Big) \cap
		C^{s}\big([0, T_{\mathrm{max}}), L^{2}(\Omega)\big) \\
		&\hspace{3in} =: C^{0}\big([0, T_{\mathrm{max}}), \calZ_s\big), \\
		\theta &\in \Big(\bigcap_{k = 0}^{s - 2} C^{k}\big([0, T_{\max}), H^{s + 1 - k}(\Omega) \cap H^{1}_{0}(\Omega)\big)\Big) \cap
		C^{s - 1}\big([0, T_{\mathrm{max}}), H^{1}(\Omega)\big) \\
		&\hspace{3in} =: C^{0}\big([0, T_{\mathrm{max}}), \calT_s\big), 
	\end{aligned} 
\end{equation}
where $[0, T_{\max})$ is the maximal existence interval (in time) with $T_{\max} \leq \infty$.
Unless $T_{\mathrm{max}} = \infty$, either the solution norm explodes
or the hyperbolicity of Equation (\ref{system_stability_1}) is violated at $T_{\mathrm{max}}$.

For the solution pair $(z, \theta)$, we introduce the squared norm functionals 
$E_{k}(t)$ 
for $k = 1, 2, 3$ and $0 \leq t < T_{\max}$
\begin{align}
	E_1(t) &:=\tfrac{1}{2}\ltwo{A^{-\sfrac{1}{2}}z_t(t,\cdot)}^2 \hspace{-2pt} + \hspace{-2pt} \tfrac{\gamma}{2}\ltwo{z_t(t,\cdot)}^2 
	\hspace{-2pt} + \hspace{-2pt} \tfrac{1}{2}\ltwo{A^{\sfrac{1}{2}}z(t,\cdot)}^2 \hspace{-2pt} + \tfrac{1}{2}\ltwo{A^{\sfrac{1}{2}}\theta(t,\cdot)}^2, \label{E_1_definition} \\
	E_2(t) &:=\tfrac{1}{2}\ltwo{z_t(t,\cdot)}^2+\tfrac{\gamma}{2}\ltwo{A^{1/2}z_t(t,\cdot)}^2 + \tfrac{1}{2}\ltwo{Az(t,\cdot)}^2+\tfrac{1}{2}\ltwo{A\theta(t,\cdot)}^2, \label{E_2_definition}\\
	E_3(t) &:=\tfrac{1}{2}\ltwo{z_{tt}(t,\cdot)}^2 \hspace{-2pt}+\hspace{-2pt} \tfrac{\gamma}{2}\ltwo{A^{1/2}z_{tt}(t,\cdot)}^2 \hspace{-2pt}+\hspace{-2pt} \tfrac{1}{2}\ltwo{Az_t(t,\cdot)}^2 
	\hspace{-2pt}+\hspace{-2pt} \tfrac{1}{2}\ltwo{A\theta_t(t,\cdot)}^2 	\label{E_3_definition}
\end{align}
and define the natural energy at level $s = 3$ by means of
\begin{equation}
	X(t) \equiv \big\|(z, z_{t}, z_{tt}, \theta, \theta_{t})(t, \cdot)\big\|_{X}^{2} := E_2(t)+E_3(t).  \label{X_definition}
\end{equation}
Note that $E_{1}(t)$ represents the basic `natural' energy of the system.
For the sake of brevity -- and slighly abusing the notation -- we will write in the following:
\begin{equation}
	\big\|(z, \theta)\big\|_{X} \text{ instead of } \big\|(z, z_{t}, z_{tt}, \theta, \theta_{t})\big\|_{X} \text{ and }
	\big\|(z, \theta)\big\|_{\mathcal{Z}_{s} \times \mathcal{T}_{s}} \text{ instead of } \big\|(\partial_{t}^{\leq s} z, \partial_{t}^{\leq s - 1} \theta)\big\|_{X}, \notag
\end{equation}
where $\partial_{t}^{\leq k} := (1, \partial_{t}, \dots, \partial_{t}^{k})$.


\begin{remark} \label{equiv_energy_spaces_rmk}
	In order to prove the system is globally well-posed, we first seek for an {\em a priori} estimate for the solution and then prove $T_{\max} = \infty$. 
	To capture the essential decay of the energy, we work with $\|(z, \theta)\|_{X}$, instead of $\|(z, \theta)\|_{\calZ_s \times \calT_s}$, for most part of this section. 
	Although all main results in this section are presented in terms of $X(t)$, their equivalence with the statements in Section \ref{SECTION_MAIN_RESULTS}, given the smallness of the initial data, 
	is shown in Lemma \ref{equivalent_norms}.
\end{remark}

We start by an observation that, for small data, $z$ has one extra order of hidden regularity in space encoded in the definition of $E_2$.

\begin{lemma}[$z$-energy boost] 	\label{energy_boost} 
	For any $t \in (0,T_{\max})$, if $(z,\theta)$ satisfies
	\begin{equation} \label{small_data_for_boost}
		E_2(t) < \epsilon_1 := \tfrac{1}{2\sqrt{C'}} \text{ for some constant } C' > 0 \mbox{ (defined in \eqref{def_C'})},
	\end{equation} 
	there holds
	\begin{equation}
		\|z(t)\|^2_{H^3(\Omega)}  \leq C \big(X(t) + X^{3}(t)\big) \text{ for some } C > 0. \notag
	\end{equation}
\end{lemma}

\begin{proof}

From Equation \eqref{system_stability_2}, $-\eta A^{3/2}\theta = \alpha A^{1/2} z_t + \beta A^{1/2} \theta_t + \sigma A^{1/2} \theta$. 
Hence,
\begin{equation} \label{A_theta_prio}
	\ltwo{A^{3/2}\theta}^2 \leq C\big(\ltwo{A^{1/2} z_t}^2 + \ltwo{A^{1/2} \theta_t}^2 + \ltwo{A^{1/2} \theta}^2\big) \leq C\big(E_2(t)+E_3(t)\big).
\end{equation}
Using H\"{o}lder's inequality
\begin{equation} 
\| a \cdot b\|_{L^2(\Omega)}^2 \leq \| a \|_{L^6(\Omega)}^{2} \cdot \|b\|_{L^3(\Omega)}^{2} \notag
\end{equation}
and Sobolev imbedding theorem $H^1(\Omega) \hookrightarrow L^6(\Omega) \hookrightarrow L^4(\Omega) \hookrightarrow L^3(\Omega)$, we arrive at:
\begin{equation} \label{A_2_prio}
\begin{aligned}
	\big\|A^{1/2}&(6z|\nabla z|^2)\big\|_{L^{2}(\Omega)}^2 \\
	&\leq  C\ltwo{A^{1/2}z \cdot |\nabla z|^2}^2+C\ltwo{z (\nabla z \cdot \nabla A^{1/2}z)}^2 \\
	&\leq  C\|A^{1/2}z\|_{L^{6}(\Omega)}^2 \| \nabla z|\|_{L^{6}(\Omega)}^4+C'\|Az\|_{_{L^{2}(\Omega)}}^2 \|A^{3/2} z\|_{L^2(\Omega)}^2 \ltwo{ A z}^2 \\
	&\leq  C E_2^3(t)+C'E_2^2(t) \|A^{3/2} z\|_{L^2(\Omega)}^2
\end{aligned}
\end{equation}
and
\begin{equation} \label{A_3_prio}
\begin{aligned}
	\ltwo{ A^{1/2}(1+3  z^2) \cdot Az}^2
	& \leq C\ltwo{ z A^{1/2} z \cdot Az} ^2  \\
	& \leq C \|z\|_{H^2(\Omega)}^2  \|A^{1/2} z\|_{H^2(\Omega)}^2 \ltwo{Az}^2  \\
	& \leq C'\|Az\|_{L^{2}(\Omega)}^4 \|A^{3/2} z\|_{L^2(\Omega)}^2   \\
	& \leq C'E_2^2(t) \|A^{3/2} z\|_{L^2(\Omega)}^2 \text{ for some } C, C' > 0.
\end{aligned}
\end{equation}
Now, to estimate $\|A^{3/2} z\|_{L^2(\Omega)}^2$, we successively transform Equation \eqref{system_stability_1} to obtain:
\begin{eqnarray*}
	(1+3  z^2)Az &=& -\left[A^{-1}z_{tt}+\gamma z_{tt} - \alpha A\theta - 6  z|\nabla z|^2\right], \\
	A^{1/2}[(1+3  z^2)Az] &=& -A^{1/2}\left[A^{-1}z_{tt}+\gamma z_{tt} - \alpha A\theta - 6  z|\nabla z|^2\right], \\
	A^{1/2}(1+3  z^2) \cdot Az+(1+3  z^2)A^{3/2}z&=& -A^{1/2}\left[A^{-1}z_{tt}+\gamma z_{tt} - \alpha A\theta - 6  z|\nabla z|^2\right], \\
	(1+3  z^2)A^{3/2}z&=& -A^{1/2}\left[A^{-1}z_{tt}+\gamma z_{tt} - \alpha A\theta - 6  z|\nabla z|^2\right], \\
	& & -A^{1/2}(1+3  z^2) \cdot Az, \\
	A^{3/2}z&=& -\tfrac{1}{1+3  z^2}\Big[A^{-1/2} z_{tt}+ A^{1/2}\gamma z_{tt}-\alpha A^{3/2}\theta \\
	& & -A^{1/2}(6  z|\nabla z|^2) + A^{1/2}(1+3  z^2) \cdot Az \Big].
\end{eqnarray*}
Taking into account $\frac{1}{1+3  z^2} \leq 1$, Equations \eqref{A_theta_prio}, \eqref{A_2_prio} and \eqref{A_3_prio} yield
\begin{align}
	\ltwo{A^{3/2}z}^2
	&\leq
	\left\|\tfrac{1}{1+3  z^2}\right\|_{L^{\infty}(\Omega)}^2 \Big[\ltwo{A^{-1/2} z_{tt}}+\gamma \ltwo{ A^{1/2} z_{tt}}+\alpha \ltwo{A^{3/2}\theta} 	\nonumber	\\
	&+\ltwo{A^{1/2}(6z|\nabla z|^2)} +\ltwo{ A^{1/2}(1+3  z^2) \cdot Az} \Big]^2 	\nonumber	\\[2mm]
	&\leq
	C \Big[\ltwo{A^{-1/2} z_{tt}}^2+\ltwo{ A^{1/2}\gamma z_{tt}}^2+\ltwo{A^{3/2}\theta}^2 	\nonumber	\\
	&+\ltwo{A^{1/2}(6z|\nabla z|^2)}^2 +\ltwo{ A^{1/2}(1+3z^2) \cdot Az}^2 \Big] 	\nonumber	\\[2mm]
	&\leq
	C \Big[E_3(t)+E_3(t)+[E_2(t)+E_3(t)] 	\nonumber	\\
	&+ C E_2^3(t)+C'E_2^2(t) \|A^{3/2} z\|_{L^2(\Omega)}^2 + C'E_2^2(t) \|A^{3/2} z\|_{L^2(\Omega)}^2 \Big]. 		\label{def_C'}
\end{align}
Hence,
\begin{equation} 
	\ltwo{A^{3/2}z}^2 
	\leq 
	\tfrac{C}{1-2E_2^2(t) C'} \big[E_2(t)+E_3(t) + E_2^3(t)\big]. \notag
\end{equation}
By the assumption in Equation \eqref{small_data_for_boost}, $1-2E_2^2(t) C' > \frac{1}{2}$ and, therefore,
\begin{equation} \label{z_H3}
	\ltwo{A^{3/2}z}^2 
	\leq 
	C \big[E_2(t)+E_3(t) + E_2^3(t)\big] 
	\leq 
	C \big[X(t) + X^3(t)\big], \notag
\end{equation}
which finishes the proof.
\end{proof}

\begin{lemma}[\emph{A priori} estimate]  
	\label{lemma_barrier}
	Let Assumption \ref{ASSUMPTION_LOCAL_EXISTENCE} be satisfied for some $s \geq 3$ and $0 < T < T_{\max}$ with $T_{\max} > 0$ denoting the maximal existence time from Theorem \ref{THEOREM_LOCAL_EXISTENCE}
	and let
	\begin{equation}
		\label{small_than_one}
		X(0)<1.
	\end{equation}
	Then,
	\begin{equation} 
		\label{barrier_bound_final}
		X_s(T)+\int_0^T X_s(t) \dt \leq C_1 X_s(0) + C_2 \sum_{i \in I} X_s^{\alpha_i}(T)+ C_3\sum_{j \in J} \int_0^T X_s^{\beta_j}(t) \dt
	\end{equation}
	where $C_k \geq 0$, $k=1,2,3$ are constants, both $I, J \subset \mathbb{N}$ are finite sets, $\alpha_i>1$ for any $i \in I$ and $\beta_j>1$ for any $j \in J$.
\end{lemma}

\begin{proof}
	The proof mainly consists of energy estimates at the energy levels $s = 1, 2$ and $3$. 
	Estimates for higher energy spaces ($s \geq 4$) follow similarly. 
	Throughout this proof, $\inner{\cdot, \cdot}$ denotes the standard $L^2(\Omega)$-inner product. Moreover, without loss of generality, we assume
	\begin{equation}
		\label{small_than_one_forall}
		X(t) < 1 \mbox{ for } t \in [0,T].
	\end{equation}
	If not, by continuity of $X(t)$, we can find a smaller $T$ so that \eqref{small_than_one_forall} holds true. Equations \eqref{small_than_one} and \eqref{small_than_one_forall} are critical in Step 4.2 of this proof. The dimension of the domain is also important by virtue of, for instance, Equation \eqref{Sobolev_Embeddings_2}.
	\medskip

	\noindent {\it Step 1: Level 1 energy identity.} 
	Multiplying Equation \eqref{system_stability_1} with $z_t$ in $L^2(\Omega)$ and integrating by parts, we get
	\begin{eqnarray}
		\tfrac{1}{2}\partial_{t}\ltwo{A^{-1/2}z_t}^2 + \tfrac{\gamma}{2}\partial_{t} \ltwo{z_t}^2+\tfrac{1}{2}\partial_{t} \ltwo{A^{1/2}z}^2-\alpha \inltwo{A\theta, z_t},
		=\inltwo{F, z_t}   
	\label{level1.1}
	\end{eqnarray}
	where $F(z, \nabla z, Az) = -3z^2Az+6z|\nabla z|^2$ from \eqref{definition_F}.
	Similar actions on \eqref{system_stability_2} multiplied by $A\theta$ lead to
	\begin{eqnarray}
		\tfrac{\beta}{2}\partial_{t} \ltwo{A^{1/2}\theta}^2 + \eta\ltwo{A\theta}^2+\sigma\ltwo{A^{1/2}\theta}^2+\alpha\inltwo{z_t, A\theta}=0. \label{level1.2}
	\end{eqnarray}
	From Equations \eqref{level1.1} and \eqref{level1.2}, we get the  $E_1$-identity:
	\begin{equation} \label{level1ei}
		E_1(T)+\int_0^T \left( \eta \ltwo{A\theta}^2 + \sigma \ltwo{A^{1/2}\theta}^2\right) \dt=E_1(0)+\int_0^T \inltwo{F, z_t} \dt.
	\end{equation}

	\noindent {\it Step 2: Level 2 energy estimate.} Recalling from Equation \eqref{E_2_definition}
	\begin{equation}
		E_2(t) =\tfrac{1}{2}\ltwo{z_t(t,\cdot)}^2+\tfrac{\gamma}{2}\ltwo{A^{1/2}z_t(t,\cdot)}^2+ \tfrac{1}{2}\ltwo{Az(t,\cdot)}^2+\tfrac{1}{2}\ltwo{A\theta(t,\cdot)}^2 \notag
	\end{equation}
	and using Equation \eqref{level1ei} along with Poincar\'{e}-Friedrichs inequality, we get:
	\begin{equation} \label{level1cor1}
		E_1(T) \leq E_1(0)+\int_0^T \inltwo{F, z_t} \dt \leq C\left[E_2(0)+\int_0^T \inltwo{F, z_t} \dt \right].
	\end{equation}
	Similarly,
	\begin{align}  
		\int_0^T \ltwo{A^{1/2}\theta}^2 \dt \leq C \int_0^T \ltwo{A\theta}^2 \dt \leq C\left[E_2(0)+\int_0^T \inltwo{F, z_t} \dt \right], 	\label{level1cor3} \\
		\int_0^T \ltwo{A\theta}^2 \dt \leq C \int_0^T \ltwo{A\theta}^2 \dt \leq C\left[E_2(0)+\int_0^T \inltwo{F, z_t} \dt \right], 	\label{level1Atheta} \\
		\max\Big\{\ltwo{z_t(t,\cdot)}^2 , \ltwo{A^{1/2}z(t,\cdot)}^2\big\}
		\leq CE_1(T) \leq C\left[E_2(0)+\int_0^T \inltwo{F, z_t} \dt \right]. 	\label{level1cor5}
	\end{align}
	%
	%
	

	In order to estimate $E_2(t)$ and $\ds \int_0^T E_2(t) \dt$, we employ a set of higher energy multipliers. 
	We first start by multiplying Equation \eqref{system_stability_1} with $Az_t$ and recalling \eqref{level1cor1} to observe that
	\begin{equation}
		\begin{aligned}
		& \tfrac{1}{2}\ltwo{z_t(T)}^2+\tfrac{\gamma}{2}\ltwo{A^{1/2}z_t(T)}^2+ \tfrac{1}{2}\ltwo{Az(T)}^2 - \alpha \int_0^T \inltwo{A\theta, Az_t} \dt 
\\
		\leq \ &
		C\left[E_2(0)+\int_0^T \inltwo{F, z_t}  \dt\right]
		+
		\int_0^T \inltwo{F, Az_t} \dt.  
		\end{aligned}
		\label{level2.1}
	\end{equation}
	Second, in order to estimate $\ltwo{A\theta(T)}^2$, we multiply \eqref{system_stability_2} by $A\theta_t$ to get
	\begin{align}
	& \dfrac{\eta}{2}\ltwo{A\theta(T)}^2 + \dfrac{\sigma}{2} \ltwo{A^{1/2}\theta(T)}^2 + \beta\int_0^T \ltwo{A^{1/2}\theta_t}^2  \nonumber 	\\
	=\ &
	\dfrac{\eta}{2}\ltwo{A\theta(0)}^2 + \dfrac{\sigma}{2} \ltwo{A^{1/2}\theta(0)}^2 - \alpha \int_0^T \langle A^{1/2}z_t, A^{1/2}\theta_t  \rangle.
	\end{align}
We apply Young's inequality to the inner product term, we get the following estimate
	\begin{align}
	& \dfrac{\eta}{2}\ltwo{A\theta(T)}^2  + \beta\int_0^T \ltwo{A^{1/2}\theta_t}^2  \nonumber 	\\
	\leq\  &
	E_2(0) + \dfrac{\sigma}{2} \ltwo{A^{1/2}\theta(T)}^2 + \alpha \dfrac{\beta}{\alpha}\int_0^T \ltwo{A^{1/2}\theta_t}^2 
	+ 
	C\int_0^T \ltwo{A^{1/2}z_t}^2
	\end{align}
After performing cancellations, employing \eqref{level1cor1}, and rescaling, the estimate becomes
	\begin{align} 	\label{Atheta_estimate_final}
	C\ltwo{A\theta(T)}^2
	\leq
	C\left[E_2(0) + \int_0^T \inltwo{F, z_t} \dt \right] +\epsilon\int_0^T \ltwo{A^{1/2}z_t}^2.
	\end{align}
	Last, by Equation \eqref{system_stability_1}, $z_{tt} = -BAz+\alpha BA\theta+BF$, where $B=(A^{-1}+\gamma)^{-1}$. 
	Multiplying Equation \eqref{system_stability_2} by $Az_t$, we get
	\begin{align}
		-\eta \int_0^T \inltwo{A\theta, Az_t} \dt
		&= \alpha \int_0^T \ltwo{A^{1/2}z_t}^2 \dt + \beta \inltwo{\theta, Az_t}\Big|_0^T +\beta \int_0^T \inltwo{A\theta, BAz} \dt \nonumber \\
		&- \alpha \beta \int_0^T \inltwo{A\theta, BA\theta} \dt - \beta \int_0^T \inltwo{A\theta, BF} \dt+\sigma \int_0^T \inltwo{\theta, Az_t} \dt.  \label{level2.2}
	\end{align}
	Recall from the proof of Theorem \ref{THEOREM_LOCAL_EXISTENCE} that $B$ is a continuous operator.
	Therefore, $\|B\| \leq C_{\gamma}$. Adding \eqref{Atheta_estimate_final} and a multiple of Equation \eqref{level2.2} to \eqref{level2.1}, we have:
	\begin{align}
		\begin{split}
			E_2 (T) &+ \tfrac{\alpha^2}{\eta} \int_0^T \ltwo{A^{1/2}z_t}^2 \dt \leq
			C\left[E_2(0) + \int_0^T \inltwo{F, z_t} \dt \right]
			+
			\int_0^T \inltwo{F, Az_t} \dt \\
			&+
			C_{\epsilon}\ltwo{A^{1/2}\theta(T)}^2+\epsilon \ltwo{A^{1/2}z_t(T)}^2 \\
			&+ C \int_0^T \ltwo{A\theta}^2 \dt 
			+
			C_{\epsilon}\int_0^T \ltwo{A^{1/2}\theta}^2 \dt + \epsilon \int_0^T \ltwo{Az}^2 \dt \\
			&+
			C_{\epsilon}\int_0^T \ltwo{\theta}^2 \dt + \epsilon \int_0^T \ltwo{Az}^2 \dt
			+
			\tfrac{\alpha \beta}{\eta} \int_0^T \inltwo{A\theta, BF} \dt.
		\end{split}
		\label{level2_Aohzt}
	\end{align}
	After merging respective terms and applying Equations \eqref{level1cor1} and \eqref{level1cor3}, \eqref{level2_Aohzt} becomes
	\begin{equation}
		\begin{split}
			(1-\epsilon)E_2(T) &+ \tfrac{\alpha^2}{\eta} \int_0^T \ltwo{A^{1/2}z_t}^2 \dt \\
			&\leq 
			C\left[E_2(0)+\int_0^T \inltwo{F, z_t} \dt + \int_0^T \inltwo{F, Az_t}  \dt  + \int_0^T \inltwo{BF, A\theta} \right] \dt \\
			&+ 2\epsilon \int_0^T E_2(t) \dt, \text{ where } C=C(\epsilon, \alpha, \beta, \gamma, \eta, \sigma). 
		\end{split}
		\label{level2.3}
	\end{equation}
	Third, to estimate $\ds \int_0^T \ltwo{Az}^2 \dt$, we multiply\eqref{system_stability_1} by $Az$ and, again, use \eqref{level1cor1} to get
	\begin{equation}
	 	\begin{split}
			\int_0^T \ltwo{Az}^2 \dt
			&\leq \dfrac{\alpha^2}{2\eta C_{\gamma, \Omega}} \left|\int_0^T \inltwo{F,Az} \dt\right| +\epsilon' \int_0^T E_2(t) \dt  +\epsilon' E_2(T) \\
			&+ C_{\epsilon'}\left[E_2(0)+\int_0^T \inltwo{F, z_t} \dt  \right] +\dfrac{\alpha^2}{2\eta} \int_0^T \ltwo{A^{1/2}z_t}^2 \dt.
		\end{split}
		\label{level2.4}
	\end{equation}
	%
	Finally, after combining Equations \eqref{level1cor3}, \eqref{level2.3} and \eqref{level2.4}, we arrive at
	%
	%
	\begin{equation} \label{L2_F_E}
		\begin{aligned}
 			E_2(T)+C_1 \int_0^T E_2(t) \dt &\leq C_2 E_2(0) + C_3\left\{ 
			\left| \int_0^T \inltwo{F,z_t} \dt \right|
			+\left| \int_0^T \inltwo{F,Az} \dt \right| \right. \\
 			&+ \left.\left| \int_0^T \inltwo{F,Az_t} \dt\right| 
			+\left| \int_0^T \inltwo{BF,A\theta} \dt \right|
			\right\}.
		\end{aligned}
	\end{equation}

	\medskip
	
	\noindent {\it Step 3: Level 3 energy estimate.}
	The 3\textsuperscript{rd} level energy space ($E_3$) is one order higher in time than the 2\textsuperscript{nd} level space ($E_2$).
	Hence, after differentiating Equation \eqref{system_stability_1}--\eqref{system_stability_2} in time
	\begin{subequations}
		\begin{align}
			A^{-1}z_{ttt}+\gamma z_{ttt}+Az_t - \alpha A\theta_t &= \partial_{t}F(z,\nabla z, \triangle z) 	 &&\text{ in } (0, \infty) \times \Omega,	\label{sys1_level3}\\
			\beta \theta_{tt} + \eta A\theta_t - \sigma \theta_t  + \alpha z_{tt} &= 0 	 &&\text{ in } (0, \infty) \times \Omega,	 \label{sys2_level3} \\
			z=z_t=\theta=\theta_t &=0 	 &&\text{ in } (0, \infty) \times \partial\Omega, \label{boundary_condition_1_level3} 
		\end{align}
	\end{subequations}
	a procedure similar to Step 2 can be employed.
	Denote the right-hand side of \eqref{sys1_level3} by
	\begin{equation}    \label{definition_G}
		G(z) \equiv \partial_{t} F=\partial_{t} [-3z^2Az+6z|\nabla z|^2]=-6zz_tAz-3z^2Az_t+6z_t|\nabla z|^2 +12z (\nabla z \cdot \nabla z_t)
	\end{equation} 
	and calculate
	\begin{equation}     
		\label{definition_G_t}
		\begin{aligned}
			\partial_{t} G(z)
			&= -6z_t^2Az-6zz_{tt}Az-12z z_t A z_t -3z^2 Az_{tt} \\
			&+ 6z_{tt} |\nabla z|^2+24z_t(\nabla z \cdot \nabla z_t) +12z|\nabla z_t|^2+12z(\nabla z_t \cdot \nabla z_{tt}).
		\end{aligned}
	\end{equation} 
	Letting $\widetilde{z}=z_t$ and $\widetilde{\theta}=\theta_t$, we have
	\begin{equation}
		E_3(t) = \tfrac{1}{2}\ltwo{\widetilde{z}_t}^2 + \tfrac{\gamma}{2}\ltwo{A^{1/2}\widetilde{z}_t}^2 + \tfrac{1}{2}\ltwo{A\widetilde{z}}^2 + \tfrac{1}{2}\ltwo{A\widetilde{\theta}}^2. \notag
	\end{equation}
	Therefore, in a fashion similar to Equation \eqref{L2_F_E}, we get the following 3\textsuperscript{rd} level energy estimate:
	\begin{equation} 
		\label{L3_F_E}
		\begin{aligned}
			E_3(T) + C_1 \int_0^T E_3(t) \dt &\leq C_2 E_3(0) + C_3\left\{ 
			\left| \int_0^T \inltwo{G(z),\tz_t} \dt \right|
			+\left| \int_0^T \inltwo{G(z),A\tz} \dt \right| \right. \\
			&+ \left.\left| \int_0^T \inltwo{G(z),A\tz_t} \dt \right| 
			+\left| \int_0^T \inltwo{BG(z),A\widetilde{\theta}} \dt \right|
			\right\}.
		\end{aligned}
	\end{equation}
	Recalling $X(t) = E_{2}(t) + E_{3}(t)$, combine \eqref{L2_F_E} and \eqref{L3_F_E}:
	\begin{equation} 
		\label{X_goal}
		\begin{aligned}
			X(T) + C_1 \int_0^T X(t) \dt
			&\leq C_2 X(0)+C_3\left\{ 
			\left| \int_0^T \inltwo{F,z_t} \dt \right|
			+\left| \int_0^T \inltwo{F,Az} \dt \right| \right. \\
			&+\left| \int_0^T \inltwo{F,Az_t} \dt \right|
			+\left| \int_0^T \inltwo{BF,A\theta} \dt \right|
			+\left| \int_0^T \inltwo{G,\tz_t} \dt \right| \\
			&+\left| \int_0^T \inltwo{G,A\tz} \dt \right|  \left.
			+\left| \int_0^T \inltwo{G,A\tz_t} \dt \right| 
			+\left| \int_0^T \inltwo{BG,A\widetilde{\theta}} \dt \right|
			\right\}.
		\end{aligned}
	\end{equation}
	
	\medskip
	
	\noindent {\it Step 4:}
	We now need to estimate the integrals on the right-hand side (r.h.s.) of Equation \eqref{X_goal} (eight terms in the brackets) 
	to get \eqref{barrier_bound_final} for $s = 3$. We will be using the fact that 
	\begin{equation} 
		\label{Sobolev_Embeddings_2}
		H^2(\Omega) \hookrightarrow L^{\infty}(\Omega) \quad \mbox{and} \quad H^2(\Omega) \hookrightarrow W^{1,4}(\Omega) \qquad \text{ for } d \in \{2,3\}.
	\end{equation}
	\noindent {\it Step 4.1: The first four terms on the r.h.s. of \eqref{X_goal}.} 
	The embeddings in Equation (\ref{Sobolev_Embeddings_2}) together with Young's inequality lead to an estimate of the first term:
	\begin{equation} 
		\label{X_F_1} 
		\begin{aligned} 
			\left| \int_0^T \inltwo{F,z_t} \dt  \right| 
			&\leq  \left| \int_0^T \inltwo{3z^2Az,z_t} \dt \right| 
			+
			\left| \int_0^T \inltwo{6z|\nabla z|^2,z_t} \dt \right| \\
			&\leq  C_{\epsilon}\int_0^T X^3(t) \dt  +\epsilon \int_0^T X(t) \dt.
		\end{aligned} 
	\end{equation}
	Here, we noted $H^2(\Omega) \hookrightarrow W^{1,4}(\Omega)$. 
	In general, $W^{2,p}(\Omega) \hookrightarrow W^{1,4}(\Omega)$ for $p > {\frac{4d}{4+d}}$. Since $d=2$ or $3$, we chose $p = 2$. 
	Similar arguments apply to the next two terms with the following inequalities:
	\begin{align}
		\label{X_F_2}   
		\left| \int_0^T \inltwo{F,Az} \dt  \right| 
		&\leq  C\int_0^T \ltwo{Az}^3 \ltwo{Az} \dt
		\leq  C \int_0^T X^2(t) \dt, \\
		\label{X_F_3}  
		\left| \int_0^T \inltwo{F,Az_t} \dt  \right| 
		&\leq  C_{\epsilon}\int_0^T X^3(t) \dt +\epsilon \int_0^T X(t) \dt.
	\end{align} 
	Again, by a similar argument, using the continuity of the operator $B$ and Equation \eqref{X_F_1}, we can estimate the 4\textsuperscript{th} term:
	\begin{equation}  
		\label{X_F_4}   
		\begin{aligned} 
			\left| \int_0^T \inltwo{BF,A\theta} \dt  \right|
			&\leq C\left| \int_0^T \ltwo{B(3z^2Az+6z|\nabla z|^2)}^2 \dt  \right| 
			+
			C\left| \int_0^T \ltwo{A\theta}^2  \right| 
			\\
			&\leq C\int_0^T \ltwo{Az}^6 \dt
			+
			C\left[ E_2(0) + \left|\int_0^T \inltwo{F, z_t} \dt \right| \right]
			\\
			&\leq CX(0) + C_{\epsilon} \int_0^T X^3(t) \dt + \epsilon \int_0^T X(t) \dt.
		\end{aligned} 
	\end{equation}
	
	\noindent {\it Step 4.2: The highest order terms on the r.h.s. of \eqref{X_goal}.} 
	In order to estimate the remaining four terms containing $G$, we rewrite $G = G_1 + G_2$, where
	\begin{equation}
		G_1=-6zz_tAz-3z^2Az_t+6z_t|\nabla z|^2
		\quad \mbox{and} \quad
		G_2=12z (\nabla z \cdot \nabla z_t). \notag
	\end{equation}
	Therefore, the 7\textsuperscript{th} term can be bounded as follows (after two integrations by parts):
	\begin{equation} 
		\label{X_G_Aztt} 
		\begin{aligned} 
			\bigg| \int_0^T &\inltwo{G,A\tz_t} \dt \bigg|
			\leq\ \left| \int_0^T \inltwo{G_1(z),Az_{tt}} \dt  \right| + \left| \int_0^T \inltwo{G_2(z),Az_{tt}} \dt  \right| \\
			&\leq\ \left| \inltwo{G_1(z),Az_t} \Big|_0^T \right| 
			+
			\left| \int_0^T \inltwo{\partial_{t}G_1(z),Az_t} \dt  \right| 
			+
			\left| \int_0^T \inltwo{G_2(z),Az_{tt}} \dt  \right| 
			\\
			&\leq \tfrac{1}{2}\ltwo{G_1(z(0))}^2  
			+
			\tfrac{1}{2}\ltwo{Az_t(0)}^2 
			+
			C_{\epsilon} \ltwo{G_1(z(T))}^2  	\\
			&+
			\epsilon \ltwo{Az_t(T)}^2 
			+
			\left| \int_0^T \inltwo{\partial_{t}G_1(z),Az_t} \dt  \right|
			+\left| \int_0^T \inltwo{G_2(z),Az_{tt}} \dt  \right|.
		\end{aligned} 
	\end{equation} 
	
	\noindent {\it Step 4.2.1: The first four terms on the r.h.s. of \eqref{X_G_Aztt}}.
	%
	\begin{align}
		\big\|&G_1(z(0))\big\|_{L^{2}(\Omega)}^{2} \notag \\
		&\leq
		6\ltwo{z(0)z_t(0)Az(0)}^2+3\ltwo{z(0)^2Az_t(0)}^2 +6\ltwo{z_t(0)|\nabla z(0)|^2}^2 \notag \\
		&\leq
		C \Big\{\|z(0)\|_{H^2(\Omega)}^6 +\|z_t(0)\|_{H^2(\Omega)}^6 + \ltwo{Az(0)}^6 +\|z(0)\|_{H^2(\Omega)}^8 \notag \\
		&+ \ltwo{Az_t(0)}^4 +
		\|z_t(0)\|_{H^2(\Omega)}^4 +\|z(0)\|_{H^2(\Omega)}^8 \Big\} \label{X_G_1_1} \\
		&\leq  
		C \Big\{\|z(0)\|_{H^2(\Omega)}^2 +\|z_t(0)\|_{H^2(\Omega)}^2 + \ltwo{Az(0)}^2 +
		\|z(0)\|_{H^2(\Omega)}^2 + \ltwo{Az_t(0)}^2 \notag \\
		&+ \|z_t(0)\|_{H^2(\Omega)}^2 +\|z(0)\|_{H^2(\Omega)}^2 
		\Big\} \notag \\
		&\leq CX(0). \notag
	\end{align}
	Here, we used the `smallness' assumption $X(0) < 1$ from Equation \eqref{small_than_one}.
	An argument similar to Equation \eqref{X_G_1_1} yields
	\begin{align}
			C_{\epsilon} \ltwo{G_1(z(T))}^2 &\leq 
			C_{\epsilon} \ltwo{(-6zz_tAz-3z^2Az_t+6z_t|\nabla z|^2 )(T)}^2 \notag \\
			&\leq
			C_{\epsilon} \Big\{\|z(T)\|_{H^2(\Omega)}^6 +\|z_t(T)\|_{H^2(\Omega)}^6 + \ltwo{Az(T)}^6 \label{X_G_1_3} \\
			&+
			\|z(T)\|_{H^2(\Omega)}^8 + \ltwo{Az_t(T)}^4 
			+
			\|z_t(T)\|_{H^2(\Omega)}^4 +\|z(T)\|_{H^2(\Omega)}^8 
			\Big\} \notag \\
			&\leq\  
			C_{\epsilon} \big[X^4(T)+X^6(T)+X^8(T)\big]. \notag
	\end{align} 
	Trivially,
	\begin{equation} 
		\label{X_G_1_2} 
		\tfrac{1}{2} \ltwo{Az_t(0)}^2 \leq X(0)
		\quad \mbox{and} \quad
		\epsilon \ltwo{Az_t(T)}^2 \leq \epsilon X(T).
	\end{equation}
	
	%
	\noindent {\it Step 4.2.2: The 5\textsuperscript{th} terms on the r.h.s. of \eqref{X_G_Aztt}}.
	Estimating $\ds \left| \int_0^T \inltwo{\partial_{t}G_1(z),Az_t} \dt  \right|$ with $G_1=-6zz_tAz-3z^2Az_t+6z_t|\nabla z|^2$ and
	\begin{equation}
		\label{expression_six_for_partial_G}
		\partial_{t}G_1(z)=-6z_t^2Az-6zz_{tt} Az-12z z_t A z_t -3z^2 Az_{tt}  +6z_{tt}  |\nabla z|^2+12z_t(\nabla z \cdot \nabla z_t)
	\end{equation}
	amounts to dealing with each of the respective six terms.
	\begin{enumerate}[leftmargin=6.5mm, label= ({\alph*})]
		\item First, we use Young's inequality to write
		\begin{equation} 
			\label{X_G_1_5_a} 
			\left| \int_0^T \inltwo{z_t ^2 Az,Az_t }  \right| 
			\leq C\int_0^T \|z_t\|_{H^2(\Omega)}^2 \left(\ltwo{Az}^2+\ltwo{Az_t }^2\right) 
			\leq C \int_0^T  X^2(t).
		\end{equation} 
		
		\item By H\"{o}lder's inequality, choose $p=3, q=3/2$, we get
		\begin{equation} 
			\label{holder_2_3}
			\| a \cdot b\|_{L^2(\Omega)}^2 \leq \| a \|_{L^6(\Omega)}^{2} \cdot \|b\|_{L^3(\Omega)}^{2}.
		\end{equation}
		In bounded domains of $\mathbb{R}^{d}$ for $d = 2, 3$, we have
		$H^1(\Omega) \hookrightarrow L^6(\Omega) \hookrightarrow L^4(\Omega) \hookrightarrow L^3(\Omega)$, i.e.,
		\begin{equation} 
			\|a\|_{L^6(\Omega)} \leq C\|a\|_{H^1(\Omega)}, \qquad \|a\|_{L^4(\Omega)} \leq C\|a\|_{H^1(\Omega)}, \qquad  \|a\|_{L^3(\Omega)} \leq C\|a\|_{H^1(\Omega)}. \notag
		\end{equation}
		Hence,
		\begin{align}  
			\bigg| \int_0^T &\inltwo{z z_{tt} Az,Az_t} \dt \bigg|
			\leq C\int_0^T \|z\|_{H^2(\Omega)} \left( \|z_{tt}\|_{H^1(\Omega)}^2 \cdot \|Az\|_{H^1(\Omega)}^2+\ltwo{Az_t}^2 \right) \dt \notag \\
			&\leq C\int_0^T \Big(\|Az\|_{L^2(\Omega)} \|A^{1/2}z_{tt}\|_{L^2(\Omega)}^2 \|A^{3/2}z\|_{L^2(\Omega)}^2+\|Az\|_{L^2(\Omega)} \ltwo{Az_t}^2\Big) \dt \notag \\
			&\leq C\int_0^T \|Az\|_{L^2(\Omega)}^3 \dt + C\int_0^T \|A^{1/2}z_{tt}\|_{L^2(\Omega)}^6 \dt + C\int_0^T \|A^{3/2}z\|_{L^2(\Omega)}^6 \dt \notag \\
			& + \epsilon \int_0^T \|Az\|_{L^2(\Omega)}^2 \dt + C_{\epsilon} \int_0^T \ltwo{Az_t}^4 \dt \label{X_G_1_5_b} \\
			&\leq C\int_0^T \|Az\|_{L^2(\Omega)}^3 \dt + C\int_0^T \|A^{1/2}z_{tt}\|_{L^2(\Omega)}^6 \dt + C \Big[X(t) + X^3(t) \Big]^3 \dt \notag \\
			&+\epsilon \int_0^T \|Az\|_{L^2(\Omega)}^2 \dt + C_{\epsilon} \int_0^T \ltwo{Az_t}^4 \dt \notag \\
			&\leq C\int_0^T X^{3/2}(t) \dt + C\int_0^T X^3(t) \dt 
			+ \epsilon \int_0^T X(t) \dt + C_{\epsilon} \int_0^T X^2(t) \dt. \notag
		\end{align} 
		Here, we used \eqref{z_H3} and assumption in Equation \eqref{small_than_one_forall} implying $X^{k}(t) \leq X(t)$ for $k \geq 1$.
		
		\item By H\"older's inequality,
		\begin{equation} 
			\begin{aligned} 
				\left| \int_0^T \inltwo{z z_t Az_t,Az_t } \dt \right| 
				&\leq C  \int_0^T \|z\|_{L^{\infty}(\Omega)} \|z_t \|_{L^{\infty}(\Omega)}\ltwo{Az_t }^2 \dt \\
				&\leq \epsilon \int_0^T X(t) \dt + C_\epsilon \int_0^T X^3(t) \dt.
			\end{aligned}
		\end{equation} 
		
		\item 
		Since 
		\begin{equation}
			\partial_{t} \inltwo{z^2 Az_t, Az_t}=\inltwo{2zz_tAz_t, Az_t}+\inltwo{2z^2Az_t,Az_{tt}}, \notag
		\end{equation}
		we have
		\begin{equation} 
			\begin{aligned} 
				\bigg| \int_0^T &\inltwo{z^2Az_{tt},Az_t } \dt \bigg|
				= \left| \tfrac{1}{2} \int_0^T \partial_{t} \inltwo{z^2 Az_t, Az_t} \dt - \int_0^T \inltwo{zz_tAz_t, Az_t} \dt \right| \\
				&\leq
				C\left| \inltwo{z^2 Az_t, Az_t}(T) \right|
				+
				C\left| \inltwo{z^2 Az_t, Az_t}(0) \right|
				+
				C\left|\int_0^T \inltwo{zz_tAz_t, Az_t} \dt \right| \\
				&\leq  
				C\|z(T)\|_{H^2(\Omega)}^2 \ltwo{Az_t(T)}^2
				+
				C\|z(0)\|_{H^2(\Omega)}^2 \ltwo{Az_t(0)}^2 \\
				&+
				C\int_0^T \|z\|_{L^{\infty}(\Omega)} \|z_t\|_{L^{\infty}(\Omega)} \ltwo{Az_t}^2 \dt \\
				&\leq  
				CX^2(0) + CX^2(T) + C\int_0^T \|Az\|_{2} \ltwo{Az_t}^3 \dt \\
				&\leq CX(0) + CX^2(T) + \epsilon \int_0^T X(t) \dt + C_{\epsilon}\int_0^T X^3(t) \dt.
			\end{aligned} 
		\end{equation} 
		Here, $X^2(0) \leq X(0)$ by Equation \eqref{small_than_one}.
		
		\item Again, using Equation \eqref{holder_2_3}, we get
		\begin{equation} 
			\begin{aligned} 
				\left| \int_0^T \inltwo{z_{tt} |\nabla z|^2,Az_t } \dt \right| 
				&\leq 
				C_{\epsilon} \int_0^T \|z_{tt} \|_{L^6(\Omega)}^2  \| |\nabla z|^2 \|_{L^3(\Omega)}^2
				+
				\epsilon \int_0^T \ltwo{Az_t}^2 \dt \\
				&\leq
				C_{\epsilon} \int_0^T X^2(t) \dt + C_{\epsilon}\int_0^T X^4(t) \dt
				+
				\epsilon \int_0^T X(t) \dt.
			\end{aligned} 
		\end{equation} 
		
		\item Similarly,
		\begin{equation}
			\label{X_G_1_5_f}
			\left| \int_0^T \inltwo{z_t (\nabla z, \nabla z_t), Az_t } \dt \right| 
			\leq  
			C\int_0^T X^{3/2}(t) \dt + C\int_0^T X^{3/2}(t) \dt + C\int_0^T X^3(t) \dt.
		\end{equation} 
	\end{enumerate}
	
	Now, collecting Equations \eqref{X_G_1_5_a}--\eqref{X_G_1_5_f} and recalling \eqref{expression_six_for_partial_G}, we get
	\begin{equation} 
		\label{X_G_1_5} 
		\begin{aligned} 
			\left| \int_0^T \inltwo{\partial_{t}G_1(z), Az_t} \dt \right|  
			\leq & 
			CX(0) + CX^2(T) + \epsilon \int_0^T X(t) \dt \\
			& + C_{\epsilon} \int_0^T \left[X^{3/2}(t)  + X^2(t) + X^3(t) + X^4(t)\right] \dt.
		\end{aligned} 
	\end{equation} 
	
	\noindent {\it Step 4.2.3: The 6\textsuperscript{th} (last) term on the r.h.s. of \eqref{X_G_Aztt}}. 
	The estimate is produced in a similar fashion to Equation \eqref{X_G_1_5_f}:
	\begin{equation} 
		\label{X_G_2} 
		\begin{aligned} 
			\left| \int_0^T \inltwo{G_2(z),Az_{tt}} \dt \right| 
			&\leq
			C \left| \int_0^T \|z\|_{H^2(\Omega)} \| |\nabla z| \|_{H^2(\Omega)} \inltwo{A z_t ,A^{1/2}z_{tt}} \dt \right|  \\
			&\leq
			C \int_0^T X^{3/2}(t) \dt + C\int_0^T X^{3/2}(t) \dt + C\int_0^T X^3(t) \dt. 
		\end{aligned} 
	\end{equation} 
	
	\noindent {\it Step 4.3: The 5\textsuperscript{th} and 6\textsuperscript{th} term on the r.h.s. of \eqref{X_goal}.} 
	These are lower-order terms compared to those from Step 4.2. 
	Hence, we skip the details and just state the final results:
	\begin{align}
		\left| \int_0^T \inltwo{G,\tz_t} \dt \right| 
		&\leq
		\epsilon \int_0^T X(t) \dt + C_{\epsilon} \int_0^T \left[ X^{3/2}(t) +X^3(t) + X^4(t) + X^6(t) \right] \dt, \label{X_G_ztt}  \\
		\left| \int_0^T \inltwo{G,A\tz} \dt \right| 
		&\leq 
		\epsilon \int_0^T X(t) \dt + C_{\epsilon} \int_0^T \left[ X^{3/2}(t) +X^3(t) + X^4(t) + X^6(t) \right] \dt.  \label{X_G_Azt} 
	\end{align}
	
	%
	
	\noindent {\it Step 4.4: The 8\textsuperscript{th} (last) term on the r.h.s. of \eqref{X_goal}.}
	By an argument similar to Equation \eqref{level1cor3}, we get
	\begin{equation} 
		\label{Atheta_t}
		\int_0^T \ltwo{A\theta_t}^2 \dt \leq E_3(0)+\int_0^T \inltwo{G,z_{tt}} \dt.
	\end{equation}
	Therefore,
	\begin{equation} 
		\label{X_G_Atheta_t} 
		\begin{aligned} 
			\bigg|& \int_0^T \inltwo{BG,A\widetilde{\theta}} \dt \bigg| 
			\leq 
			C\int_0^T \ltwo{G}^2 + CE_3(0)+C\int_0^T \inltwo{G,z_{tt}} \dt \\
			&\leq\
			\epsilon \int_0^T X(t) \dt + CX(0) +  C_{\epsilon} \int_0^T \left[ X^{3/2}(t) +X^2(t)+X^3(t) + X^4(t) + X^6(t) \right] \dt. \\
		\end{aligned} 
	\end{equation} 
	
	\noindent \emph{Step 5:} Plugging Equations \eqref{X_F_1}--\eqref{X_F_4}, \eqref{X_G_2}--\eqref{X_G_Azt} and \eqref{X_G_Atheta_t} into \eqref{X_goal}, we finally estimate
	\begin{equation*}
		\begin{aligned}
			(1-\epsilon)X(T) &+ (C_1 - 8\epsilon) \int_0^T X(t) \dt \\
			&\leq
			C_{\epsilon} X(0) +
			C_{\epsilon}  \int_0^T\left[ X^{3/2}(t) +X^2(t)+X^3(t) + X^4(t)  + X^6(t) \right] \dt \\
			&+
			C_{\epsilon}  \left[ X^{2}(T) + X^{4}(T) + X^{6}(T) + X^{8}(T) \right],
		\end{aligned}
	\end{equation*}
	that is,
	\begin{equation} 
		\label{X_Final}
		\begin{aligned}
			X(T) &+ C_1 \int_0^T X(t) \mathrm{d}t \\
			&\leq
			C_2 X(0)  + C_3  \int_0^T\left[ X^{3/2}(t) +X^2(t)+X^3(t) + X^4(t) + X^6(t)\right] \dt \\
			&+ C_4  \left[ X^{2}(T) + X^{4}(T) + X^{6}(T) + X^{8}(T) \right],
		\end{aligned}
	\end{equation}
	which finishes the proof.
\end{proof}

\begin{remark} 	\label{equivalence_part1_rmk}
With Equation \eqref{X_Final} at hand, we can now apply the standard `barrier method' (cf. \cite[Lemma 5.1, p 485]{ignatova2014well}) to deduce the globality of the local solution, 
whose existence is guaranteed by Theorem \ref{THEOREM_LOCAL_EXISTENCE} (or Theorem \ref{w_THEOREM_LOCAL_EXISTENCE}) -- \emph{not} in the energy space (endowed with $\|\cdot\|_{X}$), 
but in the phase space (endowed with $\|\cdot\|_{\mathcal{Z}_{s} \times \mathcal{T}_{s}}$) instead. 
Apparently, $\max_{0 \leq t \leq T} \|\cdot\|_{X} \leq \max_{0 \leq t \leq T} \|\cdot\|_{\calZ_3 \times \calT_3}$ for any $0 < T < T_{\mathrm{max}}$. 
In Lemma \ref{equivalence_part1} below, we will show a `reverse' inequality, which is sufficient for a contradiction proof (see the proof of Theorem \ref{global_wellposedness_proof} below). 
After the uniform stability of the energy is established, a second lemma, i.e., Lemma \ref{equivalent_norms} below, will be presented to show the equivalence over the whole time half-line $[0, \infty)$. 
In the spirit of Remark \ref{equiv_energy_spaces_rmk}, we have:
\end{remark}


\begin{lemma}[Controlling $\max_{0 \leq t \leq T} \|\cdot\|_{\calZ_3 \times \calT_3}$ in terms of $\max_{0 \leq t \leq T} \|\cdot\|_{X}$]  \label{equivalence_part1}
	Assume a classical solution $(z, \theta)$ to Equations (\ref{system_stability_1})--(\ref{system_stability_4}) over a time interval $[0, T_{\max})$ satisfies the smallness condition
	$ 
		E_2(t) < \epsilon_1
	$
	from Equation \eqref{small_data_for_boost}, then, there holds for any $T \in (0, T_{\max})$:
	\begin{equation}
		\max_{0 \leq t \leq T} \|(z, \theta)\|_{\calZ_3 \times \calT_3}^{2}  \leq C \max_{0 \leq t \leq T} \big(\|(z,\theta)\|_{X}^{2} + \|(z,\theta)\|_{X}^{6}\big)
		\text{ for some } C = C(T) > 0. \notag
	\end{equation}
\end{lemma}

\begin{proof}
It suffices to consider the four highest energy terms: $\ltwo{A^{3/2} z}, \ltwo{z_{ttt}}, \ltwo{A^2\theta}$ and $\ltwo{A^{1/2}\theta_{tt}}$.
The first one, as shown in Lemma \ref{energy_boost},
is bounded by $C \big(\|(z, \theta)\|_{X} + \|(z,\theta)\|_{X}^{3}\big)$ for any $t \geq 0$. 
Using Equations (\ref{system_stability_1})--(\ref{system_stability_2}), the second and the third terms can be controlled by appropriate lower order terms.
Indeed, applying $\partial_{t}$ and $A$ to Equations (\ref{system_stability_1}) and (\ref{system_stability_2}), respectively,
and exploiting the bounded invertibility of $(A^{-1} + \gamma)$, we estimate
\begin{align}
	\ltwo{z_{ttt}} &\leq C \big(\ltwo{Az_t} + \ltwo{A\theta_t} + \ltwo{F'}\big) \leq C\big(\|(z,\theta)\|_{X} + \|(z,\theta)\|_{X}^{3}\big), \label{z_ttt_cubic}\\
	\ltwo{A^2\theta} &\leq C \big(\ltwo{A\theta_t} + \ltwo{A\theta} + \ltwo{Az_t}\big) \leq  C \|(z,\theta)\|_{X}
\end{align}
for any $t \geq 0$, which remains true after passing to supremum.
The last term is treated in the same fashion as in the proof of Theorem \ref{THEOREM_APPENDIX_LINEAR_HEAT_EQUATION}.
Estimating
\begin{equation} 		\label{maximal_L2_regularity}
	\|z_{ttt}\|_{L^{2}(0, T; L^{2}(\Omega))}^{2} \leq
	\int_{0}^{T} \big(X(t) + X^3(t)\big) \mathrm{d}t \leq T \Big(1 + \max_{0 \leq t \leq T} X^{2}(t)\Big) \max_{0 \leq t \leq T} X(t) 
\end{equation}
via \eqref{z_ttt_cubic} and exploiting the maximal $L^{2}$-regularity of $A$ on $(0, T)$ applied to Equation (\ref{system_stability_2}) differentiated twice in time,
the desired estimate follows.
%
\end{proof}

\begin{theorem}[Global Existence] 
	\label{global_wellposedness_proof}
	Let Assumption \ref{ASSUMPTION_LOCAL_EXISTENCE} be satisfied for some $s \geq 3$.
	Then, there exists a positive number $\epsilon$ such that for any initial data satisfying $X(0) < \epsilon$
	(which roughly means the smallness of $\|z^{0}\|_{H^{3}(\Omega)}^{2} + \|z^{1}\|_{H^{2}(\Omega)}^{2} + \|\theta^{0}\|_{H^{4}(\Omega)}^{2}$),
	the associated local solution of system \eqref{system_stability_1}--\eqref{system_stability_4} from Theorem \ref{THEOREM_LOCAL_EXISTENCE} exists globally, namely, $T_{\mathrm{max}} = \infty$.
\end{theorem}

\begin{proof}
	Without loss of generality, assume $X(0) > 0$. Indeed, if $X(0) = 0$, the only solution to \eqref{system_stability_1}--\eqref{system_stability_4} is the trivial one
	and, therefore, exists globally.
	
	We argue by contradiction. Assume $T_{\max} < \infty$. \medskip

	\noindent \emph{Step 0:} Let 
	\begin{equation}  \label{smallness_epsilon_ultimate}
	\epsilon = \min\{1, \epsilon_1, \epsilon_2\},
	\end{equation} 
	where $\epsilon_1$ is given in Equation \eqref{small_data_for_boost} in Lemma \ref{energy_boost}
	and $\epsilon_2$ is defined in Step 1 below.
	Note that due to the validness of Equation (\ref{small_data_for_boost}) to be established below (cf. proof of Step 3),
	the hyperbolicity of Equation (\ref{system_stability_1}) is satisfied for all times
	and the solution can only cease to exist if $X(t)$ blows-up.
	
	By Theorem \ref{THEOREM_LOCAL_EXISTENCE}, there exists a (local) solution on the maximal interval $[0, T_{\max})$.
	Define
	\begin{align*}
		k(x) &= x-C_4(x^2+x^4+x^6+x^8) \text{ and } \\
		h(x) &= C_1x-C_3\left[ x^{3/2} +x^{2}+x^{3} + x^{4} + x^{9/2} + x^{6} + x^{9}\right],
	\end{align*}
	where $C_1, C_3,$ and $C_4$ come from Equation \eqref{X_Final}. 
	Since $X(0)<\epsilon \leq 1$, Equation \eqref{small_than_one} is satisfied. Therefore, the estimate in Equation \eqref{barrier_bound_final}, or \eqref{X_Final}, holds and can be rewritten as 
	\begin{equation} 
		\label{X_barrier_rewrite}
		k\left(X(T)\right) + \int_0^T h(X(t)) \dt \leq C_2X(0) \text{ for } 0 \leq T < T_{\mathrm{max}}.
	\end{equation}
	Further, we observe that the (algebraic) equation $k(x) = 0$ has a unique positive solution denoted by $\eta$. 
	There also holds $k(x) > 0$ for $x \in [0,\eta)$.
	Similarly, $h(x) = 0$ has a unique positive solution denoted by $\xi$. Besides, $h(x) > 0$ for $x \in [0, \xi)$. \medskip
	
	\noindent \emph{Step 1: Small initial data $X(0)$.}
	Consider the inequality 
	\begin{equation}
		k(y) \leq C_2X(0). \notag
	\end{equation}
	Due to the continuity of $k$, for small $X(0)$, the above inequality implies $y \in [0, \delta_1] \cup [\delta_2,\infty)$, 
	where $\delta_1 \to 0$ and $\delta_2 \to \eta$ as $X(0) \to 0$. 
	Therefore, there exists an $\epsilon_2 > 0$
	such that
if $X(0) < \epsilon_2$, the following hold true:
	\begin{align}
		X(0) < \epsilon_2 &< \delta_2,  \label{X(0)_small_1}\\
		\delta_1 &< \xi,  \label{X(0)_small_2} \\
		\delta_1 &< \epsilon_1 \text{ and } \label{global_bound_for_propagation} \\
		\epsilon_2 &< \xi \mbox{ so that } h\big[X(0)\big] > 0. \notag
	\end{align}
	
	\noindent \emph{Step 2: Barrier method.} We claim
	\begin{equation} 
		\label{barrier_positive}
		h(X(t)) > 0 \mbox{ for any } t \geq 0.
	\end{equation}
	If this is not the case, by the continuity of $h \circ X$, there is a $T^* > 0$ such that $h\big[X(T^*)\big] = 0$. Hence,
	\begin{equation}  
		\label{barrier_equiq_1}
		X(T^*) = \xi.
	\end{equation}
	On the other hand, the above assumption also suggests $h\big[X(t)\big] \geq 0$ for any $t \in [0, T^*]$. 
	After invoking this in Equation \eqref{X_barrier_rewrite}, we arrive at $k\big[X(t)\big] \leq C_2X(0)$ for any $t \in [0,T^*]$. 
	By Step 1, $X(t) \in [0, \delta_1] \cup [\delta_2,\infty)$ for any $t \in [0,T^*]$. 
	However, by continuity of $X$ and Equation \eqref{X(0)_small_1}, we can eliminate the second interval and reduce the inclusion to
	\begin{equation}
		X(t) \in [0, \delta_1] \mbox{ for any } t \in [0,T^*]. \notag
	\end{equation}
	More specifically, $X(T^*) \leq \delta_1 < \xi$ by Equation \eqref{X(0)_small_2}, which contradicts \eqref{barrier_equiq_1}. \medskip
	%
	%
	
	
	\noindent \emph{Step 3: Uniform boundedness of $X(t)$.} 
	Exploiting Equations \eqref{barrier_positive} and \eqref{X_barrier_rewrite}, 
	we get $k\big[X(t)\big] \leq C_2X(0)$ for any $t \geq 0$, which leads to the global bound on $X(t)$
	\begin{equation} 		\label{uniform_bound_on_X_ult}
		X(t) \leq \delta_1 \mbox{ for any } t \geq 0 
	\end{equation}
	by a similar argument as above.
	In particular, $X(T) \leq \delta_1 \leq \epsilon_1$ by Equation \eqref{global_bound_for_propagation}. Equation \eqref{uniform_bound_on_X_ult} meets the condition of Lemma \ref{equivalence_part1}. 
	Hence, after a possible rescaling,
	Assumption \ref{ASSUMPTION_LOCAL_EXISTENCE} is satisfied by $(z, z_{t}, \theta)(T_{\mathrm{max}}, \cdot)$. Therefore,
	Theorem \ref{THEOREM_LOCAL_EXISTENCE} implies the solution exists on $[T_{\mathrm{max}}, T')$ for some $T' > T_{\mathrm{max}}$, which contradicts the maximality of $[0, T_{\mathrm{max}})$.
\end{proof}

\begin{corollary}[Uniform Stability] 	\label{cor_unif_stability}
	Under the assumptions of Theorem \ref{global_wellposedness_proof} with  $X(0) < \tilde{\epsilon}$ for some positive number $\tilde{\epsilon}$ (possibly smaller than $\epsilon$ from Equation \eqref{smallness_epsilon_ultimate} of Theorem \ref{global_wellposedness_proof}), there exist positive constants $C$ and $k$ such that
	\begin{equation} 		\label{exponential_decay_expression}
		X(t) \leq e^{-kt} CX(0)  \text{ for } t \geq 0.
	\end{equation}
\end{corollary}

\begin{proof}
	We rewrite Equation \eqref{X_Final} again as follows:
	\begin{equation} 
		\label{X_Final_rewritten_2}
		\begin{aligned}
			X(T) &\left\{1 - C_4  \left[ X(T) + X^{3}(T) + X^{5}(T) + X^{7}(T) \right]\right\} \\
			&+ \int_0^T X(t) \big\{1-C_4  \big[ X^{1/2}(t) +X^{1}(t)+X^{2}(t) + X^{3}(t) \\
			&+ X^{7/2}(t) + X^{5}(t) + X^{8}(t) \big]\big\} \dt \\
			&\leq C_2 X(0).
		\end{aligned}
	\end{equation}
	Choosing a bound $\epsilon_3$ on $X(0)$ small enough, we can make the global bound $\delta_1$ of $X(t)$ satisfy 
	\begin{align*}
		1 - C_4  \left[ \delta_1 + \delta_1^{3} + \delta_1^{5} + \delta_1^{7} \right] &\geq \tfrac{1}{2}, \\ 
		1 - C_4  \left[ \delta_1^{1/2} +\delta_1^{1}+\delta_1^{2} + \delta_1^{3} + \delta_1^{7/2} + \delta_1^{5} + \delta_1^{8} \right] &\geq \tfrac{1}{2}.
	\end{align*}
	This together with Equation \eqref{X_Final_rewritten_2} implies
	\begin{equation} 
		\label{X_Final_rewritten_3}
		X(T) + \int_0^T X(t) \dt \leq 2C_2 X(0),
	\end{equation}
	which gives $X(T) \leq 2C_2X(0)$ for any $T>0$. 
	Now we impose the final assumption on $X(0)$. Recall the number $\epsilon$ from \eqref{smallness_epsilon_ultimate} and
	let 
	\begin{equation}  	\label{smallness_tilde_epsilon}
		\tilde{\epsilon} = \min\left\{\tfrac{\epsilon}{2C_2}, \epsilon_3\right\}.
	\end{equation}
	Since $X(0)<\tilde{\epsilon} \leq \tfrac{\epsilon}{2C_2}$, 
	then $X(t) \leq \epsilon$ for any $t>0$.
	Thus, Equation \eqref{X_Final_rewritten_3} can be extended to
	\begin{equation} 
		\label{X_Final_rewritten_4}
		X(T) + \int_s^T X(t)  \dt \leq 2C_2 X(s)
	\end{equation}
	for any $s \in (0, T]$. 
	Hence, 
	\begin{equation} 
		\label{lower_bound_X(t)}
		X(t) \geq \tfrac{1}{2C_2}X(T) \text{ for } t \in [0, T].
	\end{equation}
	Combining Equation \eqref{lower_bound_X(t)} with \eqref{X_Final_rewritten_3}, we get $X(T) + \tfrac{T}{2C_2} X(T) \leq 2C_2 X(0)$. Therefore,
	\begin{equation} 
		\label{X_Final_stability_1}
		X(T) \leq \frac{1}{1 + \frac{T}{2C_2}} X(0) \mbox{ for any } T > 0.
	\end{equation}
	By choosing $T$ large enough, we get 
	\begin{equation} 
		\label{X_Final_stability}
		X(T) \leq \kappa X(0) \mbox{ for some } \kappa < 1.
	\end{equation}
	Repeating the procedure on $[T, 2T]$, $[2T, 3T]$, etc., we arrive at
	\begin{equation}
		X(t) \leq \kappa^{\lceil t/T\rceil} X(0) \leq \kappa^{t/T} X(0) \leq e^{-\big(|\ln(\kappa)|/T\big) t} X(0) \text{ for } t \geq 0, \notag
	\end{equation}
	which finishes the proof.
\end{proof}

We have now proved all main results stated in Section \ref{SECTION_MAIN_RESULTS} in the energy space associated with $\sup_{0 \leq t<\infty} X(t)$ from Equation \eqref{X_definition}. 
As announced in Remark \ref{equiv_energy_spaces_rmk} and stated in Lemma \ref{equivalent_norms} below, 
the Banach space generated by the energy (supremum) $\sup_{0 \leq t <\infty} X(t)$ is isomorphic to the solution space in Equation \eqref{s_spaces} from our Theorem \ref{THEOREM_LOCAL_EXISTENCE} for $s = 3$ 
when the initial data are sufficiently small (by virtue of Equation \eqref{smallness_epsilon_ultimate}).  
%
%
%
\begin{lemma}[Equivalence of $\sup_{0 \leq t < \infty} \|\cdot\|_{\calZ_3 \times \calT_3}$ and $\sup_{0 \leq t < \infty} \|\cdot\|_{X}$]  \label{equivalent_norms}
	If a classical solution $(z, \theta)$ to Equations (\ref{system_stability_1})--(\ref{system_stability_4}) is global, satisfies the smallness condition
	$ 
		E_2(t) < \epsilon_1
	$
	from Equation \eqref{small_data_for_boost} and decays exponentially as in Equation (\ref{exponential_decay_expression}), both norms mentioned above are equivalent:
	\begin{equation}
		c_1 \sup_{0 \leq t < \infty} \|(z,\theta)\|_{X}
		\leq
		\|(z, \theta)\|_{\calZ_3 \times \calT_3} 
		\leq 
		c_2 \sup_{0 \leq t < \infty} \|(z,\theta)\|_{X}\quad \text{ for some } c_1 \mbox{ and } c_2 > 0. 
	\end{equation}
\end{lemma}

\begin{proof}
The former inequality is trivial. For the latter one, in contrast to Lemma \ref{equivalence_part1}, all superlinear terms are linearly dominated because they are bounded by 1, 
and we are only left to show Equation \eqref{maximal_L2_regularity} with a constant independent of $T$. 
However, due to the exponential decay of $X(t)$ (Equation \eqref{exponential_decay_expression}), \eqref{maximal_L2_regularity} becomes 
\begin{equation} 		\label{maximal_L2_regularity_true_equivalence}
	\|z_{ttt}\|_{L^{2}(0, \infty; L^{2}(\Omega))}^2 
	\leq
	\tilde{C} \int_{0}^{\infty} X(t) \mathrm{d}t 
	\leq 
	\tilde{C} C \int_{0}^{\infty} e^{-kt}X(0) \mathrm{d}t
	\leq 
	c_2X(0) \leq c_2 \sup_{0 \leq t <\infty} X(t).
\end{equation}
With $A$'s maximal $L^{2}$-regularity on $(0, \infty)$, the estimate for $\sup_{0 \leq t < \infty} \ltwo{A^{1/2}\theta_{tt}}$ follows.
%
\end{proof}

Since the smallness assumption is satisfied in both Equation \eqref{smallness_epsilon_ultimate} of Theorem \ref{global_wellposedness_proof} 
and Equation \eqref{smallness_tilde_epsilon} of Corollary \ref{cor_unif_stability}, 
we resubstitute $w = A^{-1} z$ and conclude with the desired results Theorem \ref{w_global_wellposedness} and \ref{w_uniform_stability}.



\begin{appendix}
	\section{Existence Theory for Linear Evolution Equations} \label{APPENDIX}
	Let $\Omega \subset \mathbb{R}^{d}$ be a bounded domain with a $C^{s}$-boundary $\partial \Omega$ for some $s \geq \lfloor \tfrac{d}{2}\rfloor + 2$
	and let $T > 0$ be arbitrary, but fixed.
	The following well-posedness results are based on Kato's solution theory \cite{Ka1985} for abstract time-dependent evolution equations
	and its improved version presented by Jiang and Racke in \cite[Appendix A]{JiaRa2000}
	as well as maximal $L^{p}$-regularity theory (see, e.g., \cite{KuWe2004}).
	
	Thoughout this appendix and in the proof of Theorem \ref{THEOREM_LOCAL_EXISTENCE},
	we employ the following notation. For $n \geq 0$, we define
	\begin{equation}
		  \bar{D}^{n} := \big((\partial_{t}, \nabla)^{\alpha} \,|\, 0 \leq |\alpha| \leq n\big) \text{ and }
		  H^{0}_{0}(\Omega) \equiv H^{0}(\Omega) := L^{2}(\Omega). \notag
	\end{equation}
	
	Let $\phi_{\delta} \colon \mathbb{R} \to [0, \infty)$
	denote the one-dimensional Friedrichs' mollifier with a `bandwidth' $\delta > 0$.
	For an $L^{1}$-function $z \colon [0, T] \times \Omega \to \mathbb{R}$, we let
	\begin{equation}
		z_{\delta}(t, \cdot) = \int_{0}^{T} \phi_{\delta}(t - s) z(s, \cdot) \mathrm{d}s \text{ for } t \in [0, T] \text{ in } \Omega. \notag
	\end{equation}
	For details on approximation properties of mollifiers, we refer to \cite[Chapters 8 and 9]{SchuKaHoKa2012}.
	The following result is known from \cite[Lemma A.12]{JiaRa2000}.
	\begin{lemma}
		\label{LEMMA_MOLLIFIER_PROPERTIES}
		Let $a \in C^{1}\big([0, T], L^{\infty}(\Omega)\big)$,
		$v \in C^{0}\big([0, T], L^{2}(\Omega)\big)$ and
		$w \in L^{2}\big(0, T; H^{-1}(\Omega)\big)$.
		Then, for any sufficiently small $\varepsilon > 0$, there holds
		\begin{equation}
			\begin{split}
				\int_{\varepsilon}^{T - \varepsilon} \big\|\partial_{t}\big((av)_{\delta}(t, \cdot) - av_{\delta}(t, \cdot)\big)\big\|_{L^{2}(\Omega)}^{2} \mathrm{d}t
				&\to 0 \text{ and } \\
				\int_{\varepsilon}^{T - \varepsilon} \|w_{\delta}(t, \cdot)\|_{H^{-1}(\Omega)}^{2} \mathrm{d}t
				&\to \int_{\varepsilon}^{T - \varepsilon} \|w(t, \cdot)\|_{H^{-1}(\Omega)}^{2} \mathrm{d}t
				\text{ as } \delta \to 0.
			\end{split}
			\notag
		\end{equation}
	\end{lemma}
	
	\subsection{Linear Wave Equation}
	We consider a general linear wave equation with time- and space-dependent coefficients:
	\begin{subequations}
	\begin{align}
		z_{tt}(t, x) - \bar{a}_{ij}(t, x) \partial_{x_{i}} \partial_{x_{j}} z(t, x) &= \bar{f}(t, x)\phantom{0} \text{ for } (t, x) \in (0, T) \times \Omega,
		\label{EQUATION_LINEAR_WAVE_EQUATION_PDE} \\
		z(t, x) &= 0\phantom{\bar{f}(t, x)} \text{ for } (t, x) \in [0, T] \times \partial \Omega,
		\label{EQUATION_LINEAR_WAVE_EQUATION_BC} \\
		z(0, x) = z^{0}(x), \quad z_{t}(0, x) &= z^{1}(x)\phantom{00} \text{ for } x \in \Omega.
		\label{EQUATION_LINEAR_WAVE_EQUATION_IC}
	\end{align}
	\end{subequations}
	
	\begin{assumption}
		\label{ASSUMPTION_LINEAR_WAVE_EQUATION}
		Let $s \geq \lfloor\frac{d}{2}\rfloor + 2$ be a fixed integer
		and let $\gamma_{0}, \gamma_{1}$ be positive numbers.
		Assume the following conditions are satisfied.
		\begin{enumerate}
			\item \emph{Coefficient symmetry}:
			$\bar{a}_{ij}(t, x) = \bar{a}_{ji}(t, x)$ for $(t, x) \in [0, T] \times \bar{\Omega}$.
			
			\item \emph{Coefficient regularity}:
			$\bar{a}_{ij} \in C^{0}\big([0, T] \times \bar{\Omega}\big)$ and
			\begin{equation}
				\partial_{x_{k}} \bar{a}_{ij} \in L^{\infty}\big(0, T; H^{s - 1}(\Omega)\big), \quad
				\partial_{t}^{m} \bar{a}_{ij} \in L^{\infty}\big(0, T; H^{s - 1 - m}(\Omega)\big) \notag
			\end{equation}
			for $m = 1, 2, \dots, s - 1$.
			
			\item \emph{Coercivity}: For $z \in H^{1}_{0}(\Omega)$ and $t \in [0, T]$,
			\begin{equation}
				\|z\|_{H^{1}(\Omega)}^{2} \leq
				\gamma_{0} \Big(\langle \bar{a}_{ij} \partial_{x_{i}} z, \partial_{x_{j}} z\rangle_{L^{2}(\Omega)} +
				\|z\|^{2}_{L^{2}(\Omega)}\Big). \notag
			\end{equation}
			
			\item \emph{Elliptic regularity}:
			For $m = 0, 1, \dots, s - 2$, $z(t, \cdot) \in H^{1}_{0}(\Omega)$ and
			$\bar{a}_{ij}(t, \cdot) \partial_{x_{i}} \partial_{x_{j}} z(t, \cdot) \in H^{m}(\Omega)$ for a.e. $t \in [0, T]$ implies
			$u(t, \cdot) \in H^{m + 2}(\Omega)$ and
			\begin{equation}
				\|z(t, \cdot)\|_{H^{m}(\Omega)} \leq \gamma_{1}\Big(\|\bar{a}_{ij}(t, \cdot) \partial_{x_{i}} \partial_{x_{j}} z(t, \cdot)\|_{H^{m}(\Omega)} +
				\|z(t, \cdot)\|_{L^{2}(\Omega)}\Big) \text{ for a.e. } t \in [0, T]. \notag
			\end{equation}
			
			\item \emph{Right-hand side regularity}:
			For $m = 0, 1, \dots, s - 2$,
			\begin{equation}
				\partial_{t}^{m} \bar{f} \in C^{0}\big([0, T], H^{s - 2 - m}(\Omega)\big), \quad
				\partial_{t}^{s - 1} \bar{f} \in L^{2}(0, T; L^{2}(\Omega)\big). \notag
			\end{equation}
			
			\item \emph{Compatibility conditions}:
			For $m = 0, 1, \dots, s - 1$,
			\begin{equation}
				\bar{z}^{m} \in H^{s - m}(\Omega) \cap H^{1}_{0}(\Omega), \quad
				\bar{z}^{s} \in L^{2}(\Omega), \notag
			\end{equation}
			where $\bar{z}^{m}$ is recursively defined by
			\begin{equation}
				\begin{split}
					\bar{z}^{0}(x) &= z^{0}(x), \quad \bar{z}^{1}(x) = z^{1}(x), \\
					\bar{z}^{m}(x) &=
					\Big(\sum_{n = 0}^{m - 2} {m - 2 \choose n}
					\partial_{t}^{n} \bar{a}_{ij} \partial_{x_{i}} \partial_{x_{j}} \bar{z}^{m - 2 - n} +
					\partial_{t}^{m - 2} \bar{f}_{i}\Big)(0, x) \text{ for } m \geq 2
				\end{split}
				\notag
			\end{equation}
			for $x \in \Omega$.
		\end{enumerate}
	\end{assumption}
	
	\noindent
	Note that Assumption \ref{ASSUMPTION_LINEAR_WAVE_EQUATION}.2 differs from \cite[Assumption A.2.1.1]{JiaRa2000}.
	This extra regularity for $\bar{a}_{ij}$ will enable us
	to prove our {\it a priori} estimate at an energy level which is one order lower than in \cite[Theorem A.13]{JiaRa2000}.
	
	\begin{theorem}
		\label{THEOREM_APPENDIX_LINEAR_WAVE_EQUATION}
		Under Assumption \ref{ASSUMPTION_LINEAR_WAVE_EQUATION},
		the initial boundary value problem (\ref{EQUATION_LINEAR_WAVE_EQUATION_PDE})-(\ref{EQUATION_LINEAR_WAVE_EQUATION_IC})
		possesses a unique classical solution, which satisfies
		\begin{equation}
			z \in \bigcap_{m = 0}^{s - 1} C^{m}\big([0, T], H^{s - m}(\Omega) \cap H^{1}_{0}(\Omega)\big) \cap
			C^{s}\big([0, T], L^{2}(\Omega)\big). \notag
		\end{equation}
		Moreover, for $d \in \{2, 3\}$, letting
		\begin{equation}
			\begin{split}
				\phi_{0} &= \|\bar{a}_{ij}(0, \cdot)\|_{L^{\infty}(\Omega)} +
				\|\partial_{x_{k}} \bar{a}_{ij}(0, \cdot)\|_{H^{s - 1}(\Omega)}, \\
				\phi &= \sup_{0 \leq t \leq T} \Big(\|\bar{a}_{ij}(t, \cdot)\|_{L^{\infty}(\Omega)} +
				\|\partial_{x_{k}} \bar{a}_{ij}(t, \cdot)\|_{H^{s - 1}(\Omega)} +
				\sum_{m = 1}^{s - 1} \|\partial_{t}^{m} \bar{a}_{ij}(t, \cdot)\|_{H^{s - 1 - m}(\Omega)}\Big),
			\end{split}
			\notag
		\end{equation}
		there exists a positive number $K_{1}$,
		which is a continuous function of $\phi_{0}$, $\gamma_{0}$ and $\gamma_{1}$,
		and a positive number $K_{2}$, which continuously depends on $\phi$, $\gamma_{0}$ and $\gamma_{1}$,
		such that
		\begin{equation}
			\sup_{0 \leq t \leq T} \|\bar{D}^{s} z(t, \cdot)\|_{L^{2}(\Omega)}^{2} \leq
			K_{1} \Lambda_{0} \exp\big(K_{2} T^{1/2}(1 + T^{1/2} + T + T^{3/2})\big),
			\notag
		\end{equation}
		where
		\begin{equation}
			\Lambda_{0} :=
			\sum_{m = 0}^{s} \|\bar{z}\|_{H^{s - m}(\Omega)}^{2} +
			(1 + T) \sup_{0 \leq t \leq T} \big\|\bar{D}^{s - 2} \bar{f}(t, \cdot)\big\|_{L^{2}(\Omega)}
			+ T^{1/2} \int_{0}^{T} \|\partial_{t}^{s - 1} \bar{f}(t, \cdot)\|_{L^{2}(\Omega)}^{2} \mathrm{d}t.
			\notag
		\end{equation}
	\end{theorem}
	
	\begin{proof}
		Our proof is based on an abstract well-posedness and regularity result \cite[Theorems A.3 and A.9]{JiaRa2000}. \medskip \\
		{\em Existence and uniqueness at basic regularity level.}
		Similar to the proof of \cite[Theorem A.11]{JiaRa2000},
		we define for $t \in [0, T]$ a bounded linear operator
		\begin{equation}
			A(t) :=
			\begin{pmatrix}
				0 & -1 \\
				-\bar{a}_{ij}(t, \cdot) \partial_{x_{i}} \partial_{x_{j}} & 0
			\end{pmatrix}
			\colon Y_{1} \longrightarrow X_{0},
			\label{EQUATION_OPERATOR_FAMILY_LINEAR_WAVE_EQUATION}
		\end{equation}
		where the Hilbert space $X_{0} := H^{1}_{0}(\Omega) \times L^{2}(\Omega)$
		is equipped with the standard inner product induced by the product topology,
		whereas the inner product on the Hilbert space
		$Y_{1} := H^{2}(\Omega) \cap H^{1}_{0}(\Omega)$ reads as
		\begin{equation}
			\langle V, \bar{V}\rangle_{t} :=
			\big\langle \bar{a}_{ij}(t, \cdot) \partial_{x_{i}} z, \partial_{x_{j}} \bar{z}\big\rangle_{L^{2}(\Omega)} +
			\langle y, \bar{y}\rangle_{L^{2}(\Omega)} \quad
		\end{equation}
		for $V = (z, y)$ and $\bar{V} = (\bar{z}, \bar{y}) \in X_{0}$.
		Due to uniform coercivity of $\bar{a}_{ij}$ and by virtue of Poincar\'{e}-Friedrichs' inequality,
		each of the norms induced by $\langle \cdot, \cdot\rangle_{t}$ for any $t \in [0, T]$
		is equivalent to the standard norm on $X_{0}$.
		With this notation, letting $V := (z, \partial_{t} z)$,
		Equations (\ref{EQUATION_LINEAR_WAVE_EQUATION_PDE})--(\ref{EQUATION_LINEAR_WAVE_EQUATION_IC})
		can be rewritten as an abstract Cauchy problem
		\begin{equation}
			\partial_{t} V(t) + A(t) V(t) = F(t) \text{ in } (0, T), \quad V(0) = V^{0}
			\label{EQUATION_LINEAR_WAVE_EQUATION_CD_SYSTEM}
		\end{equation}
		with $F = (0, \bar{f})$ and $V^{0} = (z^{0}, z^{1})$.

		We want to show that the triple $\big(A; X_{0}, Y_{1}\big)$ is a CD-system
		in sense of \cite[Section A.1]{JiaRa2000}.
		For $t \in [0, T]$, consider the elliptic problem
		\begin{equation}
			\big(A(t) + \lambda\big) V = F \text{ with } F \in X_{0}. \notag
		\end{equation}
		Recalling Assumption \ref{ASSUMPTION_LINEAR_WAVE_EQUATION}.3,
		Lemma of Lax \& Milgram implies the resolvent estimate
		\begin{equation}
			\big\|\big(A(t) + \lambda\big)^{-1}\big\|_{L(X_{0})} \leq
			\tfrac{1}{\lambda - C} \text{ for } \lambda > \beta \text{ for some constants } \beta, C > 0,
			\label{EQUATION_RESOLVENT_ESTIMATE_A_OF_T_LINEAR_WAVE_EQUATION}
		\end{equation}
		where we used Assumption \ref{ASSUMPTION_LINEAR_WAVE_EQUATION}.2 and Sobolev's imbedding theorem to deduce
		\begin{equation}
			a_{ij}(t, \cdot) \in W^{1, \infty}(\Omega) \text{ for any } t \in [0, T]. \notag
		\end{equation}
		The continuity of the bilinear form follows similarly.
		By standard elliptic regularity theory applied to $A(t)$,
		which is possible because of Assumption \ref{ASSUMPTION_LINEAR_WAVE_EQUATION}.2 and \ref{ASSUMPTION_LINEAR_WAVE_EQUATION}.4
		as well as $C^{s}$-smoothness of $\partial \Omega$,
		the maximal domain of $A(t)$ coincides with $Y_{1}$. Hence, the operator $A(t)$ is closed.
		This along with Equation (\ref{EQUATION_RESOLVENT_ESTIMATE_A_OF_T_LINEAR_WAVE_EQUATION}) implies
		$(\beta, \infty) \subset \rho\big(A(t)\big)$.
		Therefore, $\big(A(t); t \in [0, T]\big)$ is a stable family of infinitesimal negative generators
		of $C_{0}$-semigroups on $X_{0}$ with stability constants $1, \beta$.
		Taking into account regularity conditions from Assumption \ref{ASSUMPTION_LINEAR_WAVE_EQUATION}.5,
		we can apply \cite[Theorem A.3]{JiaRa2000}, we get a unique classical solution
		\begin{equation}
			V \in C^{0}\big([0, T], Y_{1}\big) \cap C^{1}\big([0, T], X_{0}\big) \notag
		\end{equation}
		at the at basic regularity level, which is equivalent to
		\begin{equation}
			z \in C^{2}\big([0, T], L^{2}(\Omega)\big) \cap
			C^{1}\big([0, T], H^{1}_{0}(\Omega)\big) \cap
			C^{0}\big([0, T], H^{2}(\Omega) \cap H^{1}_{0}(\Omega)\big). \notag
		\end{equation}
		\noindent {\em Higher regularity.}
		For the proof of higher solution regularity,
		we consider the following increasing double scale $(X_{j}, Y_{j})$ of Hilbert spaces
		\begin{equation}
			\begin{split}
				X_{j} &= \big(H^{j+1}(\Omega) \cap H^{1}_{0}(\Omega)\big) \times H^{j}(\Omega) \text{ for } j \geq 1, \\
				Y^{j} &= \big(H^{j+1}(\Omega) \cap H^{1}_{0}(\Omega)\big) \times \big(H^{j}(\Omega) \cap H^{1}_{0}(\Omega)\big) \text{ for } j \geq 1,
				\quad Y_{0} = X_{0}.
			\end{split}
			\notag
		\end{equation}
		By virtue of Equation (\ref{EQUATION_OPERATOR_FAMILY_LINEAR_WAVE_EQUATION}), the condition
		\begin{equation}
			\partial_{t} A \in \mathrm{Lip}\big([0, T], L(Y_{j+r+1}, X_{j})\big)
			\text{ for } j = 0, \dots, s - r - 1 \text{ and }
			r = 0, \dots, s - 2 \notag
		\end{equation}
		is equivalent to
		\begin{equation}
			\partial_{t}^{r} \bar{a}_{ij}(t, \cdot) \partial_{x_{i}} \partial_{x_{j}} \in
			\mathrm{Lip}\big([0, T], L\big(H^{j+r+2}(\Omega) \cap H^{1}_{0}(\Omega), H^{j}(\Omega)\big)\big)
		\end{equation}
		for $j = 0, \dots, s - r - 1$ and $r = 0, \dots, s - 2$,
		while the latter is a direct consequence of Assumption \ref{ASSUMPTION_LINEAR_WAVE_EQUATION}.2
		and Sobolev imbedding theorem due to the fact
		$H^{\lfloor d/2\rfloor + 1}(\Omega) \hookrightarrow L^{\infty}(\Omega)$.
		Similarly, exploiting Assumption \ref{ASSUMPTION_LINEAR_WAVE_EQUATION}.4,
		one can easily verify for $j = 0, \dots, s-2$ and $\phi \in Y_{1}$ and a.e. $t \in [0, T]$
		that $A(t) \phi \in X_{j}$ implies
		\begin{equation}
			\phi \in Y_{j+1} \text{ and }
			\|\phi\|_{Y_{j+1}} \leq K \big(\|A(t) \phi\|_{X_{j}} + \|\phi\|_{X_{0}}\big) 
			\text{ for some constant } K > 0, \notag
		\end{equation}
		which does not depend on $\phi$.
		Further, Assumption \ref{ASSUMPTION_LINEAR_WAVE_EQUATION}.5 yields
		\begin{equation}
			\partial_{t} F \in C^{0}\big([0, T], X_{s-1-k}\big)
			\text{ for } k = 0, \dots, s-2 \text{ and }
			\partial_{t}^{s-1} F \in L^{1}(0, T; X_{0}). \notag
		\end{equation}
		Finally, Assumption \ref{ASSUMPTION_LINEAR_WAVE_EQUATION}.6 implies compatibility conditions
		in sense of \cite[Equations (A.8) and (A.9)]{JiaRa2000}.
		Hence, applying \cite[Theorem A.9]{JiaRa2000} at the energy level $s - 1$,
		we obtain additional regularity for the classical solution satisfying
		\begin{equation}
			V \in \bigcap_{m = 0}^{s-1} C^{m}\big([0, T], Y_{s - 1 - m}\big). \notag
		\end{equation}
		Rewriting $z$ in terms of $V$, this yields the desired regularity for $z$. \medskip \\
		{\em Energy estimates.}
		For $n = 1, \dots, s - 1$, applying the $\partial_{t}^{n-1}$-operator
		to Equation (\ref{EQUATION_LINEAR_WAVE_EQUATION_PDE}),
		we obtain a linear wave equation for $\partial_{t}^{n-1} z$ reading as
		\begin{equation}
			\partial_{t}^{2} \big(\partial_{t}^{n - 1} z\big) -
			\bar{a}_{ij} \partial_{x_{i}} \partial_{x_{j}} \big(\partial_{t}^{n - 1} z\big) = h^{n - 1} \text{ in } (0, \infty) \times \Omega,
			\label{EQUATION_LINEAR_WAVE_EQUATION_DIFFERENTIATED_WRT_TIME}
		\end{equation}
		where we used Leibniz' rule to compute
		\begin{equation}
			h^{n - 1} = \partial_{t}^{n - 1} \bar{f} +
			\sum_{m = 1}^{n - 1} \binom{n - 1}{m} \big(\partial_{t}^{m} \bar{a}_{ij}\big)
			\partial_{x_{i}} \partial_{x_{j}} \partial_{t}^{n - 1 - m} z.
			\label{EQUATION_FUNCTION_N_MINUS_ONE_DEFINITION}
		\end{equation}
		
		Multiplying Equation (\ref{EQUATION_LINEAR_WAVE_EQUATION_DIFFERENTIATED_WRT_TIME})
		in $L^{2}(\Omega)$ with $\partial_{t}^{n} z$, applying Green's formula
		and using Young's inequality, we obtain the estimate
		\begin{equation}
			\begin{split}
				\tfrac{1}{2} \partial_{t} \big(\|\partial_{t}^{n} &z(t, \cdot)\|_{L^{2}(\Omega)}^{2} +
				\|\bar{a}(t, \cdot) \nabla \partial_{t}^{n-1} z(t, \cdot)\|_{L^{2}(\Omega)}^{2}\big) \\
				&\leq \tfrac{1}{2} \big\|\big(\partial_{x_{i}} \bar{a}_{ij}(t, \cdot)\big) \partial_{x_{j}} z(t, \cdot)\|_{L^{2}(\Omega)}^{2} +
				\|\partial_{t}^{n} z(t, \cdot)\|_{L^{2}(\Omega)}^{2} + \tfrac{1}{2} \|h^{n-1}(t, \cdot)\|_{L^{2}(\Omega)}^{2}.
			\end{split}
			\notag
		\end{equation}
		Integrating w.r.t. to $t$ over $[0, T]$, exploiting Assumption \ref{ASSUMPTION_LINEAR_WAVE_EQUATION}.3
		and recalling the definition of $\phi_{0}$, we get
		\begin{equation}
			\begin{split}
				\|\partial_{t}^{n} z(t, \cdot)\|_{L^{2}(\Omega)}^{2} +
				&\|\partial_{t}^{n-1} z(t, \cdot)\|_{H^{1}(\Omega)}^{2}
				\leq
				C(\gamma_{0}, \phi_{0}) \big(\|\bar{z}^{n}\|_{L^{2}(\Omega)}^{2} +
				\|\bar{z}^{n - 1}\|_{H^{1}(\Omega)}^{2}\big) \\
				&+ C(\gamma_{0}, \phi)
				\int_{0}^{t} \big(\|\partial_{t}^{n} z(\tau, \cdot)\|_{L^{2}(\Omega)}^{2} +
				\|\partial_{t}^{n - 1} z(\tau, \cdot)\|_{H^{1}(\Omega)}^{2}\big) \mathrm{d}\tau \\
				&+ C(\gamma_{0}) \int_{0}^{t} \|h^{n - 1}(\tau, \cdot)\|_{L^{2}(\Omega)}^{2} \mathrm{d}\tau,
			\end{split}
			\label{EQUATION_DIRECT_ESTIMATE_FOR_LINEAR_WAVE_EQUATION_DIFFERENTIATED_WRT_TIME}
		\end{equation}
		where we used Sobolev imbedding theorem to estimate
		\begin{equation}
			\max_{0 \leq t \leq T} \|\partial_{x_{i}} \bar{a}_{ij}(t, \cdot)\|_{L^{\infty}(\Omega)} \leq
			C \max_{0 \leq t \leq T} \|\partial_{x_{i}} \bar{a}_{ij}(t, \cdot)\|_{H^{s - 1}(\Omega)} \leq \phi. \notag
		\end{equation}
		Here and in the sequel, $C$ denotes a positive generic constant which does not depend on
		the unknown function $z$.
		
		To derive an estimate for $\partial_{t}^{s} z$ and $\nabla \partial_{t}^{s-1} z$,
		we need to employ a molifier technique similar to \cite[Section A.2]{JiaRa2000}.
		First, we select $0 < \delta < \varepsilon < T$.
		Convolving Equations (\ref{EQUATION_LINEAR_WAVE_EQUATION_DIFFERENTIATED_WRT_TIME}) for $n = s - 1$
		with $\phi_{\delta}$, we obtain for $t \in [\varepsilon, T - \varepsilon]$
		\begin{equation}
			(\partial_{t}^{s} z)_{\delta} - \bar{a}_{ij} \big(\partial_{x_{i}} \partial_{x_{j}} \partial_{t}^{s - 2} z\big)_{\delta}
			= (h^{s - 2})_{\delta} + \eta^{s}(\cdot, \cdot; \delta)
			\label{EQUATION_LINEAR_WAVE_EQUATION_DIFFERENTIATED_WRT_TIME_MOLIFIED}
		\end{equation}
		with a correction term
		\begin{equation}
			\eta^{s}(t, \cdot; \delta) =
			\big(\bar{a}_{ij} \partial_{t}^{s - 2} \partial_{x_{i}} \partial_{x_{j}} z\big)_{\delta} -
			\bar{a}_{ij} \big(\partial_{t}^{s - 2} \partial_{x_{i}} \partial_{x_{j}} z\big)_{\delta}
			\text{ for } t \in [0, T]. \notag
		\end{equation}
		Differentiating Equation (\ref{EQUATION_LINEAR_WAVE_EQUATION_DIFFERENTIATED_WRT_TIME_MOLIFIED}) w.r.t. $t$
		\begin{equation}
			\partial_{t} (\partial_{t}^{s} z)_{\delta} - \partial_{t} \big(\bar{a}_{ij} \big(\partial_{x_{i}} \partial_{x_{j}} \partial_{t}^{s - 2} z\big)_{\delta}\big)
			= \partial_{t} (h^{s - 2})_{\delta} + \partial_{t} \eta^{s}(\cdot, \cdot; \delta),
			\notag
		\end{equation}
		multiplying the resulting equation in $L^{2}(\Omega)$ with $\partial_{t}^{s} z_{\delta}$,
		applying Green's formula and using Young's inequality, we estimate
		\begin{equation}
			\begin{split}
				\tfrac{1}{2} \partial_{t} \big(\|(\partial_{t}^{n} z)_{\delta}&(t, \cdot)\|_{L^{2}(\Omega)}^{2} +
				\|\bar{a}(t, \cdot) \nabla (\partial_{t}^{n-1} z)_{\delta}(t, \cdot)\|_{L^{2}(\Omega)}^{2}\big) \\
				&\leq \tfrac{1}{2} \big\|\big(\partial_{x_{i}} \bar{a}_{ij}(t, \cdot)\big) \partial_{x_{j}} z_{\delta}(t, \cdot)\big\|_{L^{2}(\Omega)}^{2} +
				\tfrac{1}{2} (2 + T^{-1/2}) \|(\partial_{t}^{n} z)_{\delta}(t, \cdot)\|_{L^{2}(\Omega)}^{2} \\
				&+ \tfrac{1}{2} T^{1/2} \|\partial_{t} h^{s - 2}_{\delta}(t, \cdot)\|_{L^{2}(\Omega)}^{2} +
				\tfrac{1}{2} \|\eta^{s}(t, \cdot; \delta)\|_{L^{2}(\Omega)}^{2}.
			\end{split}
			\label{EQUATION_LINEAR_WAVE_EQUATION_DIFFERENTIATED_WRT_TIME_MOLIFIED_MULTIPLIED_IN_X}
		\end{equation}
		Here, we exploited the fact $(\partial_{t} z)_{\delta} = \partial_{t} z_{\delta}$ if $w$ is once weakly differentiable w.r.t. $t$,
		$z_{\delta}|_{\partial \Omega} = z|_{\partial \Omega}$ and
		$(\partial_{x_{i}} z)_{\delta} = \partial_{x_{i}} z_{\delta}$ if $w$ is once weakly differentiable w.r.t. $x_{i}$.
		Now, integrating Equation (\ref{EQUATION_LINEAR_WAVE_EQUATION_DIFFERENTIATED_WRT_TIME_MOLIFIED_MULTIPLIED_IN_X})
		w.r.t. $t$ over $[\varepsilon, T - \varepsilon]$,
		letting $\delta$ and then $\varepsilon$ go to zero,
		exploiting the regularity of $z$, applying Lemma \ref{LEMMA_MOLLIFIER_PROPERTIES}
		and using Assumption \ref{ASSUMPTION_LINEAR_WAVE_EQUATION}.3 , we get
		\begin{equation}
			\begin{split}
				\|\partial_{t}^{s} &z(t, \cdot)\|_{L^{2}(\Omega)}^{2} +
				\|\partial_{t}^{s - 1} z(t, \cdot)\|_{H^{1}(\Omega)}^{2}
				\leq
				C(\gamma_{0}, \phi_{0}) \big(\|\bar{z}^{s}\|_{L^{2}(\Omega)}^{2} +
				\|\bar{z}^{s - 1}\|_{H^{1}(\Omega)}^{2}\big) \\
				&+ C(\gamma_{0}, \phi) (1 + T^{-1/2})
				\int_{0}^{t} \big(\|\partial_{t}^{n} z(\tau, \cdot)\|_{L^{2}(\Omega)}^{2} +
				\|\partial_{t}^{n - 1} z(\tau, \cdot)\|_{H^{1}(\Omega)}^{2}\big) \mathrm{d}\tau \\
				&+ C(\gamma_{0}) T^{1/2} \int_{0}^{t} \|h^{s - 2}(\tau, \cdot)\|_{L^{2}(\Omega)}^{2} \mathrm{d}\tau.
			\end{split}
			\label{EQUATION_DIRECT_ESTIMATE_FOR_LINEAR_WAVE_EQUATION_DIFFERENTIATED_WRT_TIME_HIGHEST_ENERGY_LEVEL}
		\end{equation}
		Combining Equations (\ref{EQUATION_DIRECT_ESTIMATE_FOR_LINEAR_WAVE_EQUATION_DIFFERENTIATED_WRT_TIME}) and
		(\ref{EQUATION_DIRECT_ESTIMATE_FOR_LINEAR_WAVE_EQUATION_DIFFERENTIATED_WRT_TIME_HIGHEST_ENERGY_LEVEL})
		leads to
		\begin{equation}
			\begin{split}
				\sum_{n = 1}^{s} &\big(\|\partial_{t}^{n} z(t, \cdot)\|_{L^{2}(\Omega)}^{2} +
				\|\partial_{t}^{m - 1} z(t, \cdot)\|_{H^{1}(\Omega)}^{2}\big)
				\leq
				C(\gamma_{0}, \phi_{0}) \Lambda_{0} \\
				&+ C(\gamma_{0}, \phi) (1 + T^{-1/2}) \int_{0}^{t} \|\bar{D}^{s} z(\tau, \cdot)\|_{L^{2}(\Omega)}^{2} \mathrm{d}\tau \\
				&+ C(\gamma_{0}) \sum_{n = 1}^{s - 1} \int_{0}^{t} \|h^{n - 1}(\tau, \cdot)\|_{L^{2}(\Omega)}^{2} \mathrm{d}\tau
				+ C(\gamma_{0}) T^{1/2} \int_{0}^{t} \|\partial_{t} h^{s-2}(\tau, \cdot)\|_{L^{2}(\Omega)}^{2} \mathrm{d}\tau.
			\end{split}
			\label{EQUATION_DIRECT_ESTIMATE_FOR_LINEAR_WAVE_EQUATION_DIFFERENTIATED_WRT_TIME_COMBINED}
		\end{equation}
		Using Sobolev imbedding theorem
		\begin{equation}
			W^{1, 2}(\Omega) \hookrightarrow L^{6}(\Omega) \hookrightarrow L^{4}(\Omega) \text{ for } d \leq 3, \notag
		\end{equation}
		we can estimate
		\begin{align*}
			\sum_{m = 1}^{n - 1} &\|(\partial_{t}^{m} \bar{a}_{ij}) \partial_{t}^{n - 1 - m} \partial_{x_{i}} \partial_{x_{j}} z\big\|_{L^{2}(\Omega)}^{2}
			\leq \sum_{m = 1}^{\min\{n - 1, 1\}} \|\partial_{t}^{m} \bar{a}_{ij}\|_{L^{\infty}(\Omega)}^{2} \|\partial_{t}^{n - 1 - m} z\|_{H^{2}(\Omega)}^{2} \\
			&+ \sum_{m = \min\{n, 2\}}^{n - 1} \|\partial_{t}^{m} \bar{a}_{ij}\|_{L^{4}(\Omega)}^{2}
			\|\partial_{t}^{n - 1 - m} z\|_{W^{2, 4}(\Omega)}^{2} \\
			&\leq C(\phi) \|\bar{D}^{n} z\|_{L^{2}(\Omega)}^{2}
			+ C \sum_{m = \min\{n, 2\}}^{n - 1} \|\partial_{t}^{m} \bar{a}_{ij}\|_{H^{1}(\Omega)}^{2}
			\|\partial_{t}^{n - 1 - m} z\|_{W^{3, 2}(\Omega)}^{2} \\
			&\leq C(\phi) \|\bar{D}^{n} z\|_{L^{2}(\Omega)}^{2}
			+ C \sum_{m = \min\{n, 2\}}^{n - 1} \|\partial_{t}^{m} \bar{a}_{ij}\|_{H^{s - 1 - m}(\Omega)}^{2}
			\|\bar{D}^{n-m+2} z\|_{L^{2}(\Omega)}^{2} \\
			&\leq C(\phi) \|\bar{D}^{s-1} z\|_{L^{2}(\Omega)}^{2}.
		\end{align*}
		Recalling Equation (\ref{EQUATION_FUNCTION_N_MINUS_ONE_DEFINITION}), we obtain
		\begin{equation}
			\begin{split}
				\int_{0}^{t} \|h^{n-1}(\tau, \cdot)\|_{L^{2}(\Omega)}^{2} \mathrm{d}\tau
				&\leq t \max_{0 \leq \tau \leq t} \|\partial_{t}^{n-1} \bar{f}(\tau, \cdot)\|_{L^{2}(\Omega)}^{2} \\
				&+ C(\phi) \int_{0}^{t} \|\bar{D}^{s - 1} z(t, \cdot)\|_{L^{2}(\Omega)}^{2} \mathrm{d}\tau
			\end{split}
			\label{EQUATION_FUNCTION_H_N_MINUS_ONE_ESTIMATE}
		\end{equation}
		for $n = 1, \dots, s - 1$ and $t \in [0, T]$.
		Similarly, for $t \in [0, T]$,
		\begin{equation}
			\int_{0}^{t} \|\partial_{t} h^{s - 2}(\tau, \cdot)\|_{L^{2}(\Omega)}^{2} \mathrm{d}\tau
			\leq \int_{0}^{t} \|\partial_{t}^{s - 1} \bar{f}(t, \cdot)\|_{L^{2}(\Omega)}^{2} \mathrm{d}\tau +
			C(\phi) \int_{0}^{t} \|\bar{D}^{s} z(\tau, \cdot)\|_{L^{2}(\Omega)}^{2} \mathrm{d}\tau.
			\label{EQUATION_FUNCTION_H_S_MINUS_TWO_TIME_DERIVATIVE_ESTIMATE}
		\end{equation}
		Now, combining Equation (\ref{EQUATION_DIRECT_ESTIMATE_FOR_LINEAR_WAVE_EQUATION_DIFFERENTIATED_WRT_TIME_COMBINED})
		as well as Equations (\ref{EQUATION_FUNCTION_H_N_MINUS_ONE_ESTIMATE}) and
		(\ref{EQUATION_FUNCTION_H_S_MINUS_TWO_TIME_DERIVATIVE_ESTIMATE}), we arrive at
		\begin{equation}
			\begin{split}
				\sum_{n = 1}^{s} \big(&\|\partial_{t}^{n} z(t, \cdot)\|_{L^{2}(\Omega)}^{2} +
				\|\partial_{t}^{n - 1} z(t, \cdot)\|_{H^{1}(\Omega)}^{2}\big) \\
				&\leq C(\gamma_{0}, \phi_{0}) \Lambda_{0}
				+ C(\gamma_{0}, \phi) (1 + T^{1/2} + T^{-1/2}) \int_{0}^{t} \|\bar{D}^{s} z(\tau, \cdot)\|_{L^{2}(\Omega)}^{2} \mathrm{d}\tau
				\text{ for } t \in [0, T].
			\end{split}
			\label{EQUATION_LINEAR_WAVE_EQUATION_HIGHER_ENERGY_ESTIMATES_WRT_TIME}
		\end{equation}
		
		To finish the proof, we need to establish estimates for the remaining derivatives.
		For $n = 1, \dots, s - 1$, consider Equation (\ref{EQUATION_LINEAR_WAVE_EQUATION_DIFFERENTIATED_WRT_TIME}).
		Application of the elliptic regularity (viz. Assumption \ref{ASSUMPTION_LINEAR_WAVE_EQUATION}.4)
		with $m = s - n - 1$ yields
		\begin{equation}
			\begin{split}
				\|\partial_{t}^{n - 1} z(t, \cdot)\|_{H^{m + 2}(\Omega)}^{2} \leq
				\gamma_{1} \big(&\|\partial_{t}^{n + 1} z(t, \cdot)\|_{H^{m}(\Omega)}^{2} +
				\|h^{n - 1}(t, \cdot)\|_{H^{m}(\Omega)}^{2} \\
				+ &\|\partial_{t}^{n - 1} z(t, \cdot)\|_{H^{m}(\Omega)}^{2}\big) \text{ for } t \in [0, T].
			\end{split}
			\label{EQUATION_LINEAR_WAVE_EQUATION_DIFFERENTIATED_WRT_TIME_ELLIPTIC_ESTIMATE}
		\end{equation}
		Using Assumption \ref{ASSUMPTION_LINEAR_WAVE_EQUATION}.1,
		Sobolev imbedding theorem and Jensen's inequality
		and applying the fundamental theorem of calculus to the second term
		on the right-hand side of Equation (\ref{EQUATION_LINEAR_WAVE_EQUATION_DIFFERENTIATED_WRT_TIME_ELLIPTIC_ESTIMATE}),
		we obtain
		\begin{equation}
			\begin{split}
				\|h^{n - 1}&(t, \cdot)\|_{H^{m}(\Omega)}^{2} \leq \\
				&\leq C(\phi_{0}) \Lambda_{0} + C T \sum_{k = 1}^{n - 1} \int_{0}^{t} \big\|\partial_{t}\big((\partial_{t}^{k} \bar{a}_{ij})
				\partial_{t}^{n - 1 - k} \partial_{x_{i}} \partial_{x_{j}} z\big)\big\|_{H^{m}(\Omega)}^{2}(\tau, \cdot) \mathrm{d}\tau \\
				&\leq C(\phi_{0}) \Lambda_{0} + C(\phi) T \int_{0}^{t} \|\bar{D}^{s} z(\tau, \cdot)\|_{L^{2}(\Omega)}^{2} \mathrm{d}\tau.
			\end{split}
			\label{EQUATION_LINEAR_WAVE_EQUATION_FUNCTION_H_N_MINUS_1_HIGHER_ENERGY_ESTIMATE}
		\end{equation}
		Note that this estimate is only true if $s \geq 3$,
		which is trivially satisfied due to Assumption \ref{ASSUMPTION_LINEAR_WAVE_EQUATION}.
		Combining Equations (\ref{EQUATION_LINEAR_WAVE_EQUATION_HIGHER_ENERGY_ESTIMATES_WRT_TIME}),
		(\ref{EQUATION_LINEAR_WAVE_EQUATION_DIFFERENTIATED_WRT_TIME_ELLIPTIC_ESTIMATE}) and
		(\ref{EQUATION_LINEAR_WAVE_EQUATION_FUNCTION_H_N_MINUS_1_HIGHER_ENERGY_ESTIMATE})
		finally yields
		\begin{equation}
			\begin{split}
				\|\bar{D}^{s} z(t, \cdot)\|_{L^{2}(\Omega)}^{2} &\leq
				C(\gamma_{0}, \gamma_{1}, \phi_{0}) \Lambda_{0} \\
				&+ C(\gamma_{0}, \gamma_{1}, \phi) (1 + T^{1/2} + T + T^{-1/2})
				\int_{0}^{t} \|\bar{D}^{s} z(\tau, \cdot)\|_{L^{2}(\Omega)}^{2} \mathrm{d}\tau
			\end{split}
			\notag
		\end{equation}
		for any $t \in [0, T]$.
		The claim is now a direct consequence of Gronwall's inequality.
	\end{proof}
	
	\begin{remark}
		It should be pointed out that our proof differs from that of Jiang and Racke \cite{JiaRa2000}
		as we can carry it out at the energy level $s \geq [\tfrac{d}{2}] + 2$
		whereas Jiang and Racke \cite{JiaRa2000} require $s \geq [\tfrac{d}{2}] + 3$.
		This ``improvement'' is possible since Theorem \ref{THEOREM_APPENDIX_LINEAR_WAVE_EQUATION}
		is applied to a quasilinear wave equation
		with the quasilinearity depending on the function itself and not its gradient.
		A comment on this issue can also be found in \cite[Remark 14.4]{Ka1985}.
	\end{remark}
	
	\subsection{Linear Heat Equation}
	In this appendix section, we consider an initial-boundary-value problem with Dirichlet boundary conditions for
	the linear homogeneous isotropic heat equation reading as
	\begin{subequations}
	\begin{align}
		\theta_{t}(t, x) - a \triangle \theta(t, x) &= \bar{g}(t, x)\phantom{0} \text{ for } (t, x) \in (0, T) \times \Omega,
		\label{EQUATION_LINEAR_HEAT_EQUATION_PDE} \\
		\theta(t, x) &= 0\phantom{\bar{g}(t, x)} \text{ for } (t, x) \in [0, T] \times \partial \Omega,
		\label{EQUATION_LINEAR_HEAT_EQUATION_BC} \\
		\theta(0, x) &= \theta^{0}(x)\phantom{0}\;\, \text{ for } x \in \Omega.
		\label{EQUATION_LINEAR_HEAT_EQUATION_IC}
	\end{align}
	\end{subequations}
	We present a well-posedness result for Equations (\ref{EQUATION_LINEAR_HEAT_EQUATION_PDE})--(\ref{EQUATION_LINEAR_HEAT_EQUATION_IC}).
	In contrast to \cite[Chapter A.3]{JiaRa2000}, our proof is based on the operator semigroup theory,
	in particular, the maximal $L^{2}$-regularity theory, which is equivalent to analyticity of the semigroup generated by Dirichlet-Laplacian.
	A different technique is employed here to obtain a higher solution regularity needed for the fixed-point iteration in Theorem \ref{THEOREM_LOCAL_EXISTENCE}.
	Besides, the topologies used for the data $(\theta^{0}, \bar{g})$ and the solution $\theta$ differ from those in \cite[Chapter A.3]{JiaRa2000}.
	
	\begin{assumption}
		\label{ASSUMPTION_LINEAR_HEAT_EQUATION}
		Let $s \geq 2$ and $a > 0$.
		Assume the following assumptions are satisfied.
		\begin{enumerate}
			\item {\em Right-hand side regularity: }
			For $k = 0, 1, \dots, s - 1$,
			$\partial_{t}^{k} \bar{g} \in C^{0}\big([0, T], H^{s - 1 - k}(\Omega)\big)$.
			Recall $H^{0}_{0}(\Omega) \equiv H^{0}(\Omega) := L^{2}(\Omega)$.
			
			\item {\em Regularity and compatibility conditions: }
			For $k = 0, 1, \dots, s - 2$, let
			\begin{equation}
				\bar{\theta}^{k} \in H^{s + 1 - k}(\Omega) \cap H^{1}_{0}(\Omega) \qquad \text{ and } \qquad
				\bar{\theta}^{s - 1} \in H^{1}_{0}(\Omega), \notag
			\end{equation}
			where $\bar{\theta}^{k}$'s are given by
			\begin{align*}
				\bar{\theta}^{l}(x) = a^{l} \triangle^{l} \theta^{0}(x) + \sum_{n = 0}^{l - 1} a^{n} \triangle^{n} \theta^{0}(x) \, \partial_{t}^{l - 1 - n} \bar{g}(0, x)
				\text{ for } x \in \Omega \text{ and } l = 0, \dots, s - 1. \notag
			\end{align*}
		\end{enumerate}
	\end{assumption}
	
	\begin{theorem}
		\label{THEOREM_APPENDIX_LINEAR_HEAT_EQUATION}
		Under Assumption \ref{ASSUMPTION_LINEAR_HEAT_EQUATION},
		the initial-boundary-value problem
		(\ref{EQUATION_LINEAR_HEAT_EQUATION_PDE})--(\ref{EQUATION_LINEAR_HEAT_EQUATION_IC})
		possesses a unique classical solution satisfying
		\begin{align*}
			\partial_{t}^{k} \theta &\in C^{0}\big([0, T], H^{s + 1 - k}(\Omega) \cap H^{1}_{0}(\Omega)\big) \text{ for } k = 0, \dots, s - 2, \\
			\partial_{t}^{s - 1} \theta &\in C^{0}\big([0, T], H^{1}_{0}(\Omega)\big) \cap L^{2}\big(0, T; H^{2}(\Omega) \cap H^{1}_{0}(\Omega)\big)
			\text{ and } \partial_{t}^{s} \theta \in L^{2}\big(0, T; L^{2}(\Omega)\big). \notag
		\end{align*}
		Moreover, there exists a constant $C > 0$ such that
		\begin{align*}
			\sum_{k = 0}^{s - 2} \max_{0 \leq t \leq T} \|\partial_{t}^{k} \theta(t, \cdot)\|_{H^{s + 1 - k}(\Omega)}^{2} &+
			\max_{0 \leq t \leq T} \|\partial_{t}^{s - 1} \theta(t, \cdot)\|_{H^{1}(\Omega)}^{2} \\
			&+ \int_{0}^{T} \big(\|\triangle \partial_{t}^{s - 1} \theta(t, \cdot)\|_{L^{2}(\Omega)}^{2} + \|\partial_{t}^{s} \theta(t, \cdot)\|_{L^{2}(\Omega)}^{2}\big) \mathrm{d}t
			\leq C \Theta_{0},
		\end{align*}
		where
		\begin{equation}
			\Theta_{0} =
			(1 + T) \Big(
			\sum_{k = 0}^{s - 2} \|\bar{\theta}^{k}\|_{H^{s + 1 - k}(\Omega)}^{2} + \|\bar{\theta}^{s - 1}\|_{H^{1}(\Omega)}^{2} +
			\sum_{k = 0}^{s - 1} \max_{0 \leq t \leq T} \|\partial_{t}^{k} \bar{g}(t, \cdot)\|_{H^{s - 1 - k}(\Omega)}^{2} \mathrm{d}t\Big).
			\notag
		\end{equation}
	\end{theorem}

	\begin{proof}
		Let $A := \triangle_{D}$ denote the $L^{2}$-realization of the Dirichlet-Laplacian with the domain
		\begin{equation}
			D(A) := \big\{\theta \in H^{1}_{0}(\Omega) \,\big|\, \triangle \theta \in L^{2}(\Omega)\big\}
			= H^{2}(\Omega) \cap H^{1}_{0}(\Omega), \notag
		\end{equation}
		where the latter identity follows by standard elliptic regularity theory.
 		(Note the difference in sign over Sections \ref{SECTION_EXISTENCE_AND_UNIQUENESS} and \ref{SECTION_LONG_TIME_BEHAVIOR}.)
		Using Lax \& Milgram lemma to prove the resolvent identity
		\begin{equation}
			\sup_{\lambda \in \mathbb{C} \backslash (-\infty, 0]}
			\big\|\lambda (\lambda - A)^{-1}\big\|_{L(L^{2}(\Omega))} < \infty, \notag
		\end{equation}
		we conclude that $a A$ generates a bounded analytic semigroup of angle $\frac{\pi}{2}$ on $L^{2}(\Omega)$.
		Due to the Hilbert space structure, \cite[1.7 Corollary]{KuWe2004} further implies $a A$ has the maximal $L^{p}$-regularity property.
		
		Consider the solution map $\mathscr{S}$ sending $(\tilde{\theta}^{0}, \tilde{g})$ to the (mild) solution $\tilde{\theta}$ of
		\begin{equation}
			\tilde{\theta}_{t} - a A \tilde{\theta} = \tilde{g} \text{ in } (0, T), \quad
			\tilde{\theta}(0, \cdot) = \tilde{\theta}^{0}. \label{EQUATION_CAUCHY_PROBLEM_LINEAR_HEAT_EQUATION_MAXREG}
		\end{equation}
		By classic $C_{0}$-semigroup theory and the maximal $L^{p}$-regularity theory, we have:
		\begin{itemize}
			\item The mapping
			\begin{equation}
				\mathscr{S} \colon H^{1}_{0}(\Omega) \times L^{2}\big(0, T; L^{2}(\Omega)\big) \to
				H^{1}\big(0, T; L^{2}(\Omega)\big) \cap L^{2}\big(0, T; H^{2}(\Omega) \cap H^{1}_{0}(\Omega)\big) \label{EQUATION_MAXREG_SOLUTION_MAP}
			\end{equation}
			is well-defined as an isomorphism between the two spaces.
			
			\item The mapping
			\begin{equation}
				\begin{split}
					\mathscr{S} \colon \big(H^{2}(\Omega) \cap H^{1}_{0}(\Omega)\big) \times &C^{1}\big([0, T], L^{2}(\Omega)\big) \to \\
					&C^{1}\big([0, T], L^{2}(\Omega)\big) \cap C^{0}\big([0, T], H^{2}(\Omega) \cap H^{1}_{0}(\Omega)\big)
				\end{split}
				\label{EQUATION_CLASSICAL_SOLUTION_MAP}
			\end{equation}
			is well-defined and continuous in respective topologies.
		\end{itemize}
		\noindent {\em Existence and uniqueness at basic level:}
		On the strength of Assumption \ref{ASSUMPTION_LINEAR_HEAT_EQUATION},
		we (in particular) have $\theta^{0} \in H^{2}_{0}(\Omega) \cap H^{1}_{0}$ and $\bar{g} \in C^{1}\big([0, T], L^{2}(\Omega)\big)$.
		Hence, there exists a unique classical solution
		\begin{equation}
			\theta \in C^{1}\big([0, T], L^{2}(\Omega)\big) \cap C^{0}\big([0, T], H^{2}(\Omega) \cap H^{1}_{0}(\Omega)\big)
		\end{equation}
		to Equations (\ref{EQUATION_LINEAR_HEAT_EQUATION_PDE})--(\ref{EQUATION_LINEAR_HEAT_EQUATION_IC}). 
		\medskip \\
		\noindent {\em Higher regularity in time:}
		We argue by induction over $k = 1, \dots, s - 2$ starting at $k = 1$.
		Applying $\partial_{t}^{k}$ to Equation (\ref{EQUATION_LINEAR_HEAT_EQUATION_PDE})
		and using Assumption \ref{ASSUMPTION_LINEAR_HEAT_EQUATION}.2, we obtain (in distributional sense)
		\begin{equation}
			\partial_{t} (\partial_{t}^{k} \theta) - a A (\partial_{t}^{k} \theta) = \partial_{t}^{k} \bar{g} \text{ in } (0, T). 
			\label{EQUATION_CAUCHY_PROBLEM_LINEAR_HEAT_EQUATION_TIME_DIFFERENTIATED}
		\end{equation}
		This motivates to consider Equation (\ref{EQUATION_CAUCHY_PROBLEM_LINEAR_HEAT_EQUATION_MAXREG}) with
		\begin{equation}
			\tilde{\theta}^{0} = \bar{\theta}^{k} \in H^{2}(\Omega) \cap H^{1}_{0}(\Omega) \text{ and }
			\tilde{g} = \partial_{t}^{k} \bar{g} \in C^{1}\big([0, T], L^{2}(\Omega)\big).
			\label{EQUATION_TILDE_THETA_0_TILE_G}
		\end{equation}
		By Equation (\ref{EQUATION_CLASSICAL_SOLUTION_MAP}),
		\begin{equation}
			\tilde{\theta} \in C^{0}\big([0, T], H^{2}(\Omega) \cap H^{1}_{0}(\Omega)\big) \cap C^{1}\big([0, T], L^{2}(\Omega)\big). \notag
		\end{equation}
		We now show the function
		\begin{equation}
			\bar{\theta}(t, \cdot) = \sum_{l = 0}^{k - 1} \frac{t^{l}}{l!} \bar{\theta}^{l} + \int_{0}^{t} \int_{0}^{t_{1}} \dots \int_{0}^{t_{k-1}}
			\tilde{\theta}(\tau, \cdot) \mathrm{d}\tau \mathrm{d}t_{k-1} \dots \mathrm{d}t_{1} \notag
		\end{equation}
		coincides with $\partial_{t}^{k} \theta$.
		By construction, $\bar{\theta}$ satisfies
		\begin{equation}
			\partial_{t} (\partial_{t}^{k} \bar{\theta}) - a A (\partial_{t}^{k} \bar{\theta}) = \partial_{t}^{k} \bar{g} \text{ in } (0, T). 
			\label{EQUATION_CAUCHY_PROBLEM_LINEAR_HEAT_EQUATION_BAR_THETA_PLUGGED_IN}
		\end{equation}
		Subtracting Equation (\ref{EQUATION_CAUCHY_PROBLEM_LINEAR_HEAT_EQUATION_BAR_THETA_PLUGGED_IN}) from Equation (\ref{EQUATION_CAUCHY_PROBLEM_LINEAR_HEAT_EQUATION_TIME_DIFFERENTIATED}), 
		multiplying with $\partial_{t}^{k - 1} (\bar{\theta} - \theta)(t, \cdot)$ in $L^{2}(\Omega)$ and using Green's formula, we get
		\begin{equation}
			\tfrac{1}{2} \partial_{t} \big\|\partial_{t}^{k-1} (\bar{\theta} - \theta)(t, \cdot)\big\|_{L^{2}(\Omega)} +
			a \big\|\nabla \partial_{t}^{k-1} (\bar{\theta} - \theta)(t, \cdot)\big\|_{L^{2}(\Omega)} = 0.
			\label{EQUATION_BAR_THETA_MINUS_THETA_ENERGY_IDENTITY}
		\end{equation}
		This along with the fact $\partial_{t}^{l} \bar{\theta}(0, \cdot) \equiv \bar{\theta}^{l} \equiv \partial_{t}^{l} \theta(0, \cdot)$ for $l = 0, \dots, k - 1$
		enables us to use the Gronwall's inequality together with the fundamental theorem of calculus to deduce $\bar{\theta} \equiv \theta$.
		Therefore, we have shown
		\begin{equation}
			\partial_{t}^{k} \theta \equiv \partial_{t}^{k} \bar{\theta} \equiv \tilde{\theta} \in C^{0}\big([0, T], H^{2}(\Omega) \cap H^{1}_{0}(\Omega)\big) \cap C^{1}\big([0, T], L^{2}(\Omega)\big).
			\label{EQUATION_PARTIAL_T_K_THETA_PRELIMIARY_REGULARITY}
		\end{equation}
		
		For $k = s - 1$, a slightly modified argument needs to be utilized.
		In this case, we only have
		\begin{equation}
			\tilde{\theta}^{0} = \theta^{s - 1} \in H^{1}_{0}(\Omega) \text{ and }
			\tilde{g} = \partial_{t}^{s - 1} \bar{g} \in C^{0}\big([0, T], L^{2}(\Omega)\big) \hookrightarrow L^{2}\big(0, T; L^{2}(\Omega)\big) \notag
		\end{equation}
		to plug into Equation (\ref{EQUATION_CAUCHY_PROBLEM_LINEAR_HEAT_EQUATION_MAXREG}).
		Instead of Equation (\ref{EQUATION_CLASSICAL_SOLUTION_MAP}), we use (\ref{EQUATION_MAXREG_SOLUTION_MAP}) to infer
		\begin{equation}
			\tilde{\theta} \in H^{1}\big(0, T; L^{2}(\Omega)\big) \cap L^{2}\big(0, T; H^{2}(\Omega) \cap H^{1}_{0}(\Omega)\big).
		\end{equation}
		Defining
		\begin{equation}
			\bar{\theta}(t, \cdot) = \sum_{l = 0}^{s - 2} \frac{t^{l}}{l!} \bar{\theta}^{l} + \int_{0}^{t} \int_{0}^{t_{1}} \dots \int_{0}^{t_{s-2}}
			\tilde{\theta}(\tau, \cdot) \mathrm{d}\tau \mathrm{d}t_{s-2} \dots \mathrm{d}t_{1}, \notag
		\end{equation}
		by the same kind of argument, we have $\bar{\theta} \equiv \theta$ and, therefore,
		\begin{equation}
			\partial_{t}^{s - 1} \theta \equiv \tilde{\theta} \in H^{1}\big(0, T; L^{2}(\Omega)\big) \cap L^{2}\big(0, T; H^{2}(\Omega) \cap H^{1}_{0}(\Omega)\big)
			\hookrightarrow C^{0}\big(0, T; H^{1}_{0}(\Omega)\big).
			\label{EQUATION_PARTIAL_T_K_THETA_PRELIMIARY_REGULARITY_S_MINUS_1}
		\end{equation}
		\noindent {\em Higher regularity in space:}
		For $k = 0, \dots, s - 2$, applying $\partial_{t}^{k}$ to Equation (\ref{EQUATION_LINEAR_HEAT_EQUATION_PDE}), we observe
		\begin{equation}
			A \partial_{t}^{k} \theta = -\tfrac{1}{a} \partial_{t}^{k + 1} \theta + \tfrac{1}{a} \partial_{t}^{k} \bar{g}.
			\label{EQUATION_PARTIAL_T_K_THETA_ELLIPTIC_BOOST}
		\end{equation}
		Hence, using the fact $A$ is an isomorphism between $H^{s + 1 - k}(\Omega) \cap H^{1}_{0}(\Omega)$ and $H^{s - 1 - k}(\Omega)$ along with
		Assumption \ref{ASSUMPTION_LINEAR_HEAT_EQUATION}.1 and Equations (\ref{EQUATION_PARTIAL_T_K_THETA_PRELIMIARY_REGULARITY}), (\ref{EQUATION_PARTIAL_T_K_THETA_PRELIMIARY_REGULARITY_S_MINUS_1}),
		we inductively obtain (beginning at $k = s - 2$ and going downward to $k = 0$)
		\begin{equation}
			\partial_{t}^{k} \theta \in C^{0}\big([0, T], H^{s + 1 - k}(\Omega) \cap H^{1}_{0}(\Omega)\big) \text{ for } k = 0, 1, \dots, s - 2. \notag
		\end{equation}
		The `remaining' case $k = s - 1$ has already been treated in the previous step so that
		\begin{equation}
			\partial_{t}^{s - 1} \theta \in C^{0}\big([0, T], H^{1}_{0}(\Omega)\big) \cap
			H^{1}\big(0, T; L^{2}(\Omega)\big) \cap L^{2}\big(0, T; H^{2}(\Omega) \cap H^{1}_{0}(\Omega)\big). \notag
		\end{equation}
		\noindent {\em Energy estimate.}
		The energy estimate easily follows from the solution operator continuity.
	\end{proof}

\end{appendix}

\section*{Acknowledgment}
{\small
Michael Pokojovy and Xiang Wan have been partially supported by a research grant from the Young Scholar Fund
funded by the Deutsche Forschungsgemeinschaft (ZUK 52/2) at the University of Konstanz, Germany.
Michael Pokojovy would like to express his gratitude to Irena Lasiecka and Roberto Triggiani
for the kind hospitality during his research stays at the University of Virginia at Charlottesville, VA as well as the University of Memphis, TN.
A partial financial support of Michael Pokojovy through the Department of Mathematics at the University of Virginia is also greatly appreciated.
}

\bibliographystyle{plain}
\bibliography{bibliography}

\begin{thebibliography}{10}

\bibitem{AmBeMi1984}
S.~A. Ambartsumian, M.~V. Belubekian, and M.~M. Minassian.
\newblock On the problem of vibrations of non-linear-elastic electroconductive
  plates in transverse and longitudinal magnetic fields.
\newblock {\em International Journal of Non-Linear Mechanics}, 19(2):141--149,
  1984.

\bibitem{Am1970}
S.~A. Ambartsumyan.
\newblock {\em Theory Of Anisotropic Plates: Strength, Stability, Vibrations},
  volume~II of {\em Progress in Material Science Series}.
\newblock Technomic Publishing Co. Inc., Stamford, CT, 1970.

\bibitem{ChuLa2010}
I.~Chueshov and I.~Lasiecka.
\newblock {\em Von Karman Evolution Equations: Well-Posedness and Long Time
  Dynamics}.
\newblock Springer Science \& Business Media, New York -- Dordrecht --
  Heidelberg -- London, 2010.

\bibitem{DeRaShi2009}
R.~Denk, R.~Racke, and Y.~Shibata.
\newblock {$L_{p}$} theory for the linear thermoelastic plate equations in
  bounded and exterior domains.
\newblock {\em Advances in Differential Equations}, 14(7/8):685--715, 2009.

\bibitem{DeSch2015}
R.~Denk and R.~Schnaubelt.
\newblock A structurally damped plate equation with {D}irichlet-{N}eumann
  boundary conditions.
\newblock {\em Journal of Differential Equations}, 259(4):1323--1353, 2015.

\bibitem{ignatova2014well}
Mihaela Ignatova, Igor Kukavica, Irena Lasiecka, and Amjad Tuffaha.
\newblock On well-posedness and small data global existence for an interface
  damped free boundary fluid--structure model.
\newblock {\em Nonlinearity}, 27(3):467, 2014.

\bibitem{Il2004}
A.~A. Ilyushin.
\newblock {\em Plasti\v{c}nost'. Uprugo-Plasti\v{c}eskije Deformacii}.
\newblock Klassi\v{c}eskij Universitetskij U\v{c}ebnik. Logos, Moscow, 2004.

\bibitem{JiaRa2000}
S.~Jiang and R.~Racke.
\newblock {\em Evolution Equations in Thermoelasticity}, volume 112 of {\em
  Monographs and Surveys in Pure and Applied Mathematics}.
\newblock Chapman \& Hall/CRC, Boca Raton, London, New York, Washington D.C.,
  2000.

\bibitem{Ka1985}
T.~Kato.
\newblock {\em Abstract Differential Equations and Nonlinear Mixed Problems}.
\newblock Fermi Lectures. Scuola Normale Superiore, Pisa, 1985.

\bibitem{KuWe2004}
P.~C. Kunstmann and L.~Weis.
\newblock {\em Maximal $L_{p}$-Regularity for Parabolic Equations, Fourier
  Multiplier Theorems and $H^{\infty}$-Functional Calculus}.
\newblock Functional Analytic Methods for Evolution Equations. Springer, Berlin
  -- Heidelberg, 2004.

\bibitem{LaLi1988}
J.~Lagnese and J.~L. Lions.
\newblock {\em Modelling, Analysis and Control of Thin Plates}, volume~6 of
  {\em Recherches en Math\'{e}matiques Appliqu\'{e}es}.
\newblock Mason, Paris, 1988.

\bibitem{LaMaSa2008}
I.~Lasiecka, S.~Maad, and A.~Sasane.
\newblock Existence and exponential decay of solutions to a quasilinear
  thermoelastic plate system.
\newblock {\em Nonlinear Differential Equations and Applications}, 15:689--715,
  2008.

\bibitem{LaWi2012}
I.~Lasiecka and M.~Wilke.
\newblock Maximal regularity and global existence of solutions to a quasilinear
  thermoelastic plate system.
\newblock {\em Discrete and Continuous Dynamical Systems A}, 33:5189--5202,
  2013.

\bibitem{No1986}
W.~Nowacki.
\newblock {\em Thermoelasticity}.
\newblock Pergamon Press, Oxford -- New York -- Toronto, etc., 2nd ed. edition,
  1986.

\bibitem{Po2011}
M.~Pokojovy.
\newblock {\em Zur Theorie w\"armeleitender Reissner-Mindlin Platten}.
\newblock PhD thesis. University of Konstanz, 2011.

\bibitem{SchuKaHoKa2012}
T.~Schuster, B.~Kaltenbacher, B.~Hofmann, and K.~S. Kazimierski.
\newblock {\em Regularization Methods in Banach Spaces}, volume~10 of {\em
  Radon Series on Computational and Applied Mathematics}.
\newblock Walter de Gruyter, Berlin/Boston, 2012.

\end{thebibliography}
\end{document}